\documentclass[]{article}

\usepackage{amsmath,amssymb}
\usepackage{amsthm}
\usepackage{enumerate}
\usepackage{subfigure}
\usepackage{tikz}
\usetikzlibrary{arrows}
\usepackage{url}

\newcommand{\R}{{\mathcal{R}}}
\newcommand{\LL}{{\mathcal{L}}}
\newcommand{\T}{S^1}
\newcommand{\RR}{\mathbb{R}}
\newcommand{\N}{\mathbb{N}}
\newcommand{\Z}{\mathbb{Z}}
\newcommand{\Q}{\mathbb{Q}}

\usepackage{xcolor}
\usepackage[normalem]{ulem}

\usepackage{cancel}

\newcommand{\x}{{\mathbf x}}
\newcommand{\y}{{\mathbf y}}

\newcommand{\condsC}{{\it C.1--C.4}}
\newcommand{\condsH}{{\it H.1--H.3}}
\newcommand{\condsHp}{{\it h.1--h.3}}
\newcommand{\condsHpp}{{\it h'.1--h'.3}}
\newcommand{\condsHppp}{{\it h''.1--h''.3}}
\newcommand{\condsA}{ {A.1--A.2}}
\newcommand{\tf}{\bar{f}}
\newcommand{\tfL}{\bar{f}_\LL}
\newcommand{\tfR}{\bar{f}_\R}
\newcommand{\bx}{\bar{x}}
\newcommand{\by}{\bar{y}}

\newcommand{\s}{\mathfrak{s}}

\newcommand{\U}{\mathcal{U}}
\newcommand{\SYSTEMWR}{(\ref{eq:reset})-\eqref{eq:forcing}}
\newcommand{\byr}{\by_\R}
\newcommand{\byl}{\by_\LL}

\newtheorem{theorem}{Theorem}[section]
\newtheorem{definition}[theorem]{Definition}
\newtheorem{remark}[theorem]{Remark}
\newtheorem{proposition}[theorem]{Proposition}
\newtheorem{example}[theorem]{Example}
\newtheorem{lemma}[theorem]{Lemma}
\newtheorem{corollary}[theorem]{Corollary}

\numberwithin{equation}{section}

\renewcommand*\thefigure{\thesection.\arabic{figure}}

\makeatletter
\renewcommand\p@subfigure{\thefigure}
\makeatother
\graphicspath{
{figures/}
{figures/linear_adding/regions/}
{figures/linear_adding/periods/}
{figures/linear_adding/devils_staircase/}
{figures/incrementing/regions2/}
{figures/incrementing/periods/}
{figures/incrementing/eta-number/}
{figures/maps/invertible/}
{figures/maps/invertible2/}
{figures/maps/inverted/}
{figures/maps/inverted2/}
{figures/maps/circle_reduction/}
{figures/incrementing/map1/}
{figures/incrementing/unique_an/}
{figures/incrementing/only_rln2/}
{figures/cantor/}
{figures/lift/}
{figures/examples/IF/T1d9/}
{figures/examples/IF/FN/}
{figures/examples/IF/1dscann/}
{figures/examples/sliding_control/system_with_relay/}
{figures/examples/sliding_control/1d_system_relay/xstar-b/}
{figures/examples/sliding_control/1d_system_relay/xstar-b_periods/}
{figures/examples/IF/boundary/}
{figures/examples/IF/map_boundary/}
{figures/examples/IF/T-per_L-bif/}
{figures/examples/IF/T-per_R-bif/}
{figures/examples/IF/}
{figures/examples/IF/firing-rate_typical/}
{figures/cycles-concat/}
{figures/examples/2d_system_relay/}
{figures/examples/zad_boost/fig_system/}
{figures/examples/zad_boost/bif_diagram/}
{figures/examples/zad_boost/periods/}
{figures/examples/zad_boost/regions/}
}

\author{Albert Granados\footnotemark[2] \and Llu\'is
Alsed\`a\footnotemark[3] \and  Maciej
Krupa\footnotemark[4]}
\begin{document}
\title{The period adding and incrementing bifurcations: from rotation
theory to applications\thanks{This work has been partially supported
by MINECO-FEDER MTM2012-31714 and MINECO MTM2011-26995-C02-01 Spanish
grants.}}

\date{}
\maketitle
\renewcommand{\thefootnote}{\fnsymbol{footnote}}
\footnotetext[2]{algr@dtu,dk, Department of Applied Mathematics and
Computer Science, Technical University of Denmark, Building 303B, 2800
Kgns. Lyngby, Denmark.}
\footnotetext[3]{Department of Mathematics,
Edifici Cc, Universitat Aut\`onoma de Barcelona, 08193 Cerdanyola del
Vall\`es, Spain}
\footnotetext[4]{NeuroMathComp Project-Team,
Inria Sophia-Antipolis Research Center, 2004 route des Lucioles, BP93,
06902 Sophia Antipolis cedex, France.}
\renewcommand{\thefootnote}{\arabic{footnote}}
\begin{abstract}
This survey article is concerned with the study of bifurcations of
piecewise-smooth maps. We review the literature in circle maps and
quasi-contractions and provide paths through this literature to prove
sufficient conditions for the occurrence of two types of bifurcation
scenarios involving rich dynamics.  The first scenario consists of the
appearance of periodic orbits whose symbolic sequences and
``rotation'' numbers follow a Farey tree structure; the periods of the
periodic orbits are given by consecutive addition. This is called the
{\em period adding} bifurcation, and its proof relies on results for
maps on the circle. In the second scenario, symbolic sequences are
obtained by consecutive attachment of a given symbolic block and the
periods of periodic orbits are incremented by a constant term. It is
called the {\em period incrementing} bifurcation, in its proof relies
on results for maps on the interval.\\
We also discuss the expanding cases, as some of the partial results
found in the literature also hold when these maps lose
contractiveness. The higher dimensional case is also discussed by
means of {\em quasi-contractions}.\\
We also provide applied examples in control theory, power electronics
and neuroscience where these results can be applied to obtain precise
descriptions of their dynamics.
\end{abstract}
\noindent \textbf{Keywords:} Piecewise-smooth maps, discontinuous
circle maps, period adding, devil's staircase, Farey Tree, period incrementing
\clearpage
\tableofcontents
\section{Introduction}\label{sec:introduction}
Piecewise-smooth (piecewise-defined or non-smooth) systems are
non-regular or discontinuous systems induced by dynamics associated
with sharp changes in position, velocity, or other magnitudes
undergoing a jump in their value. This type of systems provide more natural
and simpler models in many applications,
such as switching systems in power
electronics~\cite{BanVer01,ZhuMos03,FosHogSea09,AvrFosGraSch11},
sliding-mode techniques in control
theory~\cite{Utk93,BerBudCha98,FosGriBie00,BerKowNor03}, hybrid
systems with resets in
neuroscience~\cite{Bre04,Izh07,MenHugRin12,GraKruCle14} or impact
systems in mechanics~\cite{Hog89,GraHogSea12,GraHogSea14}. Using
non-smooth modeling can reduce the dimension of the system, but may
result in more complicated dynamics.\\
As a consequence of its broad field of application, the interest in
such type of systems has considerably grown in the last decade
(see~\cite{MakLam12} for a recent general survey). In particular,
piecewise-smooth maps have captured the attention of many researchers,
who have studied them from very different perspectives.  One of the
most reported dynamical aspects are the different bifurcations
scenarios that they may exhibit, which turn out to be extraordinarily
rich. Seduced by their graphical beauty and mainly supported by
computations, many authors have recurrently observed and reported
these bifurcation phenomena. However, although they assemble many well
known results on circle maps, they have been considered as new and
exclusive of piecewise-smooth maps.

In this review article we show that many of the well known results for
circle maps developed in the 80's and early 90's can be used to obtain
rigorous proofs of general results that can be systematically applied
to one-dimensional piecewise-smooth contracting maps.  Moreover, we
also show how many of these results scattered in the literature are
also valid not only for piecewise-smooth expanding maps, but also for
higher-dimensional ones.\\
In the general setting for the one-dimensional case we consider a
piecewise-smooth map undergoing a discontinuity at $x=0$, and consider
as parameters the two lateral images at this point. To our knowledge,
such a map was first studied by Leonov~\cite{Leonov59}, and later on
obtained as an approximation of a Poincar\'e map of smooth flow near a
homoclinic bifurcation of the figure eight and butterfly
types~\cite{Spa82,Hom96} (see Sections~\ref{sec:inc-inc_overview}
and~\ref{sec:co-dimension_two} for more references).\\
Depending on the signs and magnitude of the slopes of the map at both
sides of the discontinuity, the bifurcation scenario in this
two-dimensional parameter space may be very different. These signs are
determined by the number of twists  exhibited by the invariant
manifolds (their orientability) involved in the homoclinic bifurcation
(\cite{GhrHol94,Hom96}). When the map is contracting in both sides of
the discontinuity and both slopes have different sign, the so-called
{\em period incrementing} scenario occurs. This bifurcation was
reported in~\cite{Hom96}, and the details of the proof for this case
were given in~\cite{AvrGraSch11}.  However, if both slopes are
positive, then the so-called {\em period adding} scenario occurs.

Although this latter scenario has been widely reported
in the literature (see Section~\ref{sec:inc-inc_overview}) its proof is scattered
throughout the literature in form of partial results using different
approaches. In the late 80's, Gambaudo {\em et al.}
(\cite{GamLanTre84,GamGleTre88,GamTre88}) provided strong rigorous
results in this direction by means of the {\em maximin} approach (see
Section~\ref{sec:maximin_approach}). These provide very precise
information on the symbolic sequences. Moreover, they  were stated for
maps in metric spaces of arbitrary dimension, which permits us to
adapt them to provide results for piecewise-smooth maps in $\RR^n$. As
a counterpart, they do not allow one to distinguish straightforwardly
between the two mentioned bifurcation scenarios.\\
In order to prove the bifurcation scenario for the
increasing-increasing case, we propose to reduce the piecewise-smooth
map to a discontinuous circle map. We then assemble and adapt many
well known results for rotation theory to the discontinuous case in
order to provide the most straight path to prove the period adding
bifurcation scenario. The main advantage with respect to the maximin
approach is the fact that this method does not require the map to be
contracting but it only relies on its invertibility, and hence
requires weaker assumptions. However, the maximin approach, although
it requires a lot of contractiveness, the obtained results are valid
in higher  dimensions.\\
Other approaches suggest to proceed with renormalization
arguments~\cite{GamProThoTre86,ProThoTre87,Gle90,Zak93,Hom96,Homburg00}
to prove the occurrence of the period adding bifurcation scenario.\\

Beyond its relation with homoclinic bifurcations for flows, such type
of bifurcation scenarios have been observed in more applied contexts
modeled by both one and $n$-dimensional piecewise-smooth maps.
Examples of such applications where these bifurcation scenarios appear
are power electronics
(\cite{BerGarGliVas98,BanGre99,RakAprBan10,GiaBanImr12,ZhuMosAndMik13,AmaCasGraOliHur14,Kabeetal07,ZhuMosDeBan08}),
control theory
(\cite{BerGarIanVas02,GalYu11,FosGra13,ZhuMos03,Zhuetal01,FosGra11,FosGra13}),
economics (\cite{TraGarWes11}) or neuroscience
(\cite{FreGal11,GraKruCle14,JimMihBroNieRub13,KeeHopRin81,MenHugRin12,TieFellSej02,SigTouVid15,TouBre08,TouBre09,Ton14}).\\
Typically, these bifurcation scenarios are numerically observed in
two-dimensional parameter spaces near codimension-two bifurcation
points. Such points involve the emergence of an infinite number of
bifurcation curves, and were called {\em big bang} bifurcations by the
non-smooth community~\cite{AvrSch06}.\\
In this survey we also provide illustration on
how rotation theory can be applied to provide a rigorous basis to
analyze bifurcations of piecewise-smooth maps in four
different applied contexts: non-smooth dynamics, control-theory,
mathematical neuroscience and power electronics. We first study
(Section~\ref{sec:co-dimension_two}) bifurcation scenarios around
codimension-two bifurcation points given by the collision of two
periodic orbits with the boundary (big bang bifucations).  In the
second example (Section~\ref{sec:relay}) we consider a nonlinear
system subject to sliding-mode control in order to stabilize it around
a desired ``equilibrium'' point. In the third example
(Section~\ref{sec:IF}), we consider a periodically forced
integrate-and-fire model, a hybrid system widely used in neuroscience.
For these two last examples we show how the period adding bifurcation
scenario explains the dynamics of the systems and how the symbolic
dynamics help to obtain relevant properties from the applied point of
view. We finally provide two more examples leading to
period-adding like bifurcations for planar piecewise-smooth maps: a
higher order system subject to sliding-mode control with relays
(Section~\ref{sec:high-order_relays}) and a DC-DC boost converter
controlled with so-called ZAD Strategy
(Section~\ref{sec:boost_zad}).\\

This work is organized as follows. In Section~\ref{sec:overview} we
provide basic definitions and a detailed statement of the results for
the one-dimensional case. In Section~\ref{sec:adding_proof} we review
and extend results for circle maps to provide a proof of the period
adding bifurcation scenario (increasing-increasing or orientation
preserving case). A detailed summary of the proof is given in
Section~\ref{sec:summary}.  In Section~\ref{sec:remarks_expansive} we
emphasize up to which extend the previous results are also valid in
the presence of expansiveness.. In
Section~\ref{sec:incrementing_proof} we revisit the proof provided
in~\cite{AvrGraSch11} for the period incrementing bifurcation scenario
(increasing-decreasing or non-orientable case).  In
Section~\ref{sec:maximin_approach} we review the maximin approach, and
show how it can be applied to obtain results for piecewise-smooth maps
in $\RR^n$.  Section~\ref{sec:examples} is dedicated to  illustrate
how these results can be used in different applied fields by applying
them five examples.  Finally, we conclude in
Section~\ref{sec:conclusions_future} with some discussions and
proposals for future directions.

\section{Basic definitions and overview of results}\label{sec:overview}
\subsection{System definition and properties}\label{sec:defitions_properties}
Let us consider a piecewise-smooth map
\begin{equation}
f(x)=\left\{
\begin{aligned}
&\mu_\LL+f_\LL(x)&&\text{if }x<0\\
&-\mu_\R+f_\R(x)&&\text{if }x>0,
\end{aligned}\right.
\label{eq:normal_form}
\end{equation}
with $x\in\RR$ and $f_\LL,f_\R$ smooth functions satisfying
\begin{enumerate}[{\it h.1}]
\item $f_\LL(0)=f_\R(0)=0$
\item $0<(f_\LL(x))'<1$, $x\in(-\infty,0)$\label{hip:2}
\item $0<\left|(f_\R(x))'\right|<1$, $x\in(0,\infty)$.
\end{enumerate}
We wish to describe the possible bifurcation structures 
obtained when parameters $\mu_\LL$ and $\mu_\R$ are varied.\\
\begin{remark}\label{rem:contractiveness}
The global contractiveness of the maps $f_\LL$ and $f_\R$ is assumed
for simplicity reasons. This allows to state results on bifurcations
for arbitrarily large values of $\mu_\LL$ and $\mu_\R$. However, if
contractiveness holds only locally at the origin, then all the results
presented here are still valid for values of these parameters close
enough to the origin. Their validity when contractiveness i
lost is discussed in Section~\ref{sec:remarks_expansive}.
\end{remark}
\begin{remark}
For convenience, we do not define at this point the map $f$ at $x=0$.
Roughly speaking, the only difference given by the election of the
value of $f$ at $x=0$ will consist of the existence or not of fixed
points and periodic orbits at their bifurcation values. We remark that
one cannot only consider the value of $f$  at $x=0$ by choosing one
lateral image, but also one can consider both images or non.  We will
focus on this question whenever it becomes relevant.
\end{remark}
Due to condition \emph{h}.1, the map~\eqref{eq:normal_form} is
discontinuous at $x=0$ if $\mu_\LL\neq \mu_\R$. As we will show, this
discontinuity introduces exclusive dynamical phenomena which are not
possible in smooth ($C^1$) one-dimensional systems. As discussed in
the introduction, one observes similar phenomena (the bifurcation
scenarios described below) in smooth flows of dimension three near
homoclinic bifurcations.  They are also observed in smooth maps, when
restricted to the circle instead of $\RR$.\\
This discontinuity represents a boundary in the state space abruptly
separating two different dynamics: the ones given by the maps $f_\LL$
and $f_\R$.  These dynamics will strongly depend on the sign of
$f'_\R(x)$ on $(0,\infty)$ leading to completely different families of
periodic orbits. Note that the cases when $f_\LL(x)$ and $f_\R(x)$
have different slopes in their respective domains are conjugate
through the symmetry $x\longleftrightarrow -x$. Moreover, as it will
be shown below, when both $f_\LL(x)$ and $f_\R(x)$ are decreasing
functions in their respective domains, the possible dynamics will be
easy.  Therefore we can restrict to the case that $f_\LL(x)$ is an
increasing function in $(-\infty,0)$, as stated in \emph{h}.2.

One of the differences between the families of periodic orbits that
one can find depending on the sign of $f'_\R(x)$ in $(0,\infty)$ will
be given by the sequence of steps that periodic orbits perform at each
side of the boundary $x=0$.  Therefore, we will introduce the symbolic
dynamics given by the following symbolic encoding.  Given a point
$x\in\RR$, we associate to its trajectory by $f$,
$(x,f(x),f^2(x),\dots)$, a symbolic sequence given by
\begin{equation}
I_f(x)=\left( a(x),a\left( f(x) \right),a\left( f^2(x) \right),\dots
\right),
\label{eq:symbolic_sequence}
\end{equation}
where
\begin{equation}
a(x)=
\left\{
\begin{aligned}
&\R&&\text{if }x> 0\\
&\LL&&\text{if }x<0.
\end{aligned}
\right.
\label{eq:encoding_LR}
\end{equation}
As $a(x)$ provides a symbol of length one ($\LL$ or $\R$), one can
omit the comas separating the symbols in
Equation~\eqref{eq:symbolic_sequence} without introducing imprecisions.\\
We call this the {\it itinerary} of $x$ by $f$ or the {\it symbolic
sequence} associated with the trajectory of $x$ by $f$.\\
Let us now consider the shift operator acting on symbolic sequences
\begin{equation}
\sigma\left( \x_1\x_2\x_3\dots \right)=(\x_2\x_3\dots),
\label{eq:shift}
\end{equation}
where $\x_i\in\left\{ \LL,\R \right\}$.\\
Clearly, the shift operator satisfies
\begin{equation}
\sigma\left( I_f(x) \right)=I_{f}\left( f\left( x \right) \right).
\label{eq:shift_on_trajectory}
\end{equation}

Of special interest  for us will be the symbolic sequences associated
with periodic orbits. In this case, the symbolic sequences will be
also periodic and we will represent them by the repetition of the
generating symbolic block. For example, let $(x_1,x_2)$ be a periodic
orbit,
\begin{align*}
f(x_1)&=x_2\\
f(x_2)&=x_1,
\end{align*}
and assume $x_1<0$ and $x_2>0$. Then, the symbolic sequences
associated with $x_1$ and $x_2$ are
\begin{align*}
I_f(x_1)&=\left( \LL\R\LL\R\dots \right):=(\LL\R)^\infty\\
I_f(x_2)&=\left( \R\LL\R\LL\dots
\right):=(\R\LL)^\infty,
\end{align*}
where $\infty$ indicates infinite repetition.\\
Due to property~\eqref{eq:shift_on_trajectory}, the shift operator
acts on the generating blocks as a cyclic permutation of offset $1$,
as it moves the first symbol to the last position. More precisely, if
$(x_1,\dots,x_n)$, $x_i\in\RR$, is a periodic orbit of $f$ and
$(\x_1\dots\x_n)^\infty$, $\x_i\in\left\{ \LL,\R \right\}$, is  the
symbolic sequence associated with the periodic trajectory of $x_1$,
then
\begin{equation*}
\sigma(I_f(x_1))=\left( \x_2\x_3\dots\x_n\x_1 \right)^\infty.
\end{equation*}
Hence, a periodic orbit of length $n$ can be represented by $n$ different
symbolic sequences obtained by cyclic permutations one from each other.
\begin{definition}\label{def:order}
Symbolic sequences can be ordered by lexicographical order induced by $\LL<\R$.
That is,
\begin{equation*}
(\x_1\x_2\dots)< (\y_1\y_2\dots)
\end{equation*}
if and only if $\x_1=\LL$ and $\y_1=\R$ or $\x_1=\y_1$ and there
exists some $j\ge 1$ such that
\begin{align*}
\x_i&=\y_i,\,\text{for all}\;i< j\\
\x_j&=\LL\\
\y_j&=\R.
\end{align*}
\end{definition}
\begin{example}
The sequences $(\x_1,\x_2,\x_3,\x_4)^\infty=(\LL^2,\R,\LL)^\infty$ and
$(\y_1,\y_2,\y_3,\y_4)^\infty=(\LL^2,\R^2)^\infty$ satisfy
\begin{equation*}
(\LL^2,\R,\LL)^\infty<(\LL^2,\R^2)^\infty,
\end{equation*}
as $x_4=\LL$ and $\y_4=\R$ and $\x_i=\y_i$ for $1\le i\le 3$.
\end{example}
\begin{definition}
Given a periodic symbolic sequence $\x=(\x_1\dots\x_q)^\infty$, we
will say that it is minimal if
\begin{equation*}
\x=\min_{0\le k< q}\sigma^k(\x),
\end{equation*}
and similarly for a maximal symbolic sequence.
\end{definition}
Note that the $\min$ and $\max$ operators act on sequences
following the order given in Definition~\ref{def:order}, and hence its
output is also a sequence.
\begin{example}
The sequence $(\LL^2,\R)^\infty$ is minimal, whereas
$(\R,\LL^2)^\infty$ is maximal, ase we get
\begin{align*}
(\LL^2,\R)^\infty&=\min \left\{ (\LL^2,\R)^\infty,
(\LL,\R,\LL)^\infty,(\R,\LL^2)^\infty \right\}\\
(\R,\LL^2)^\infty&=\max \left\{ (\LL^2,\R)^\infty,
(\LL,\R,\LL)^\infty,(\R,\LL^2)^\infty \right\}
\end{align*}
\end{example}
\begin{definition}
We will say that a periodic orbit of length $n$ is a $\x_1\dots\x_n$-periodic
orbit, $\x_i\in\left\{ \LL,\R \right\}$, if there exists some point of this
periodic orbit, $x_i\in\RR$ such that
\begin{equation*}
I_f(x_i)=(\x_1\dots\x_n)^\infty.
\end{equation*}
\end{definition}
\begin{example}
The periodic orbit $(-1,-0.5,1,-2)$, satisfying
\begin{equation*}
f(-1)=-0.5\quad f(-0.5)=1\quad f(1)=-2\quad f(-2)=-1,
\end{equation*}
is a $\LL^2\R\LL$, $\LL\R\LL^2$, $\R\LL^3$ and $\LL^3\R$-periodic
orbit, as
\begin{align*}
I_f(-1)&=(\LL^2\R\LL)^\infty\\
I_f(-0.5)&=(\LL\R\LL^2)^\infty\\
I_f(1)&=(\R\LL^3)^\infty\\
I_f(-2)&=(\LL^3\R)^\infty.
\end{align*}
\end{example}
Usually, in order to represent the symbolic sequence of a periodic
orbit we will choose its minimal representative. For example, assume
$(x_1,\dots,x_5)$ is a $5$-periodic orbit such that
$I_f(x_1)=(\LL\R\LL\LL\R)^\infty$. Then, we will say that
$(x_1,\dots,x_5)$ is a $\LL^2\R\LL\R$-periodic orbit or a periodic
orbit of type $\LL^2\R\LL\R$, where the superindex $2$ means that
there two consecutive symbols $\LL$.\\

Given a periodic orbit, besides its period, one important
characteristic associated with its symbolic sequence is the number of
symbols $\R$ and $\LL$ and how are they distributed along the
sequence. The latter will be explained in detail below. For the former
we need the following
\begin{definition}\label{def:Wpq}
We call $W_{p,q}$ the set of periodic symbolic sequences generated
by a symbolic block of length $q$ containing $p$ symbols $\R$:
\begin{equation*}
W_{p,q}=\left\{\y\in\left\{ \LL,\R
\right\}^\mathbb{N}\,|\,\y= \x^\infty,\,\x\in\left\{ \LL,\R \right\}^q\,\text{and $\x$ contains
$p$ symbols $\R$} \right\}
\end{equation*}
\end{definition}
\begin{example}
The sets $W_{2,7}$ and $W_{3,7}$ become
\begin{equation*}
W_{2,7}=\left\{ (\LL^5,\R^2)^\infty,(\LL^4,\R,\LL)^\infty \right\}
\end{equation*}
and
\begin{equation*}
W_{3,7}=\left\{
(\LL^4,\R^3)^\infty,(\LL^3,\R^2,\LL,\R)^\infty,(\LL^2,\R,\LL,\R,\LL,\R)^\infty \right\}
\end{equation*}
\end{example}
Let now $W$ be the set of all periodic symbolic sequences:
\begin{equation*}
W=\bigcup_{(p,q)=1}W_{p,q},
\end{equation*}
where $(\cdot,\cdot)$ refers to the greatest common divisor.\\
We next define the $\eta$-number associated with a periodic symbolic
sequence:
\begin{definition}\label{def:eta-number}
Let $\x=(\x_1\dots\x_q)\in \left\{ \LL,\R \right\}^q$ be a
symbolic sequence, and let $p$ be the number of symbols $\R$ contained
in $\x$. We then define the $\eta$-number of $\x$ as
\begin{equation}
\begin{array}[]{cccc}
\eta:&W&\longrightarrow&\Q\\
&\x^\infty&\longmapsto&\frac{p}{q}
\end{array}
\label{eq:eta_number}
\end{equation}
if $\x\in W_{p,q}$, $(p,q)=1$.
\end{definition}
\begin{example}
The $\eta$-number of the sequences $(\R^3,\LL)^\infty$ and
$(\LL^2,\R,\LL,\R)^\infty$ become
\begin{equation*}
\eta\left((\R^3,\LL)^\infty\right)=\frac{3}{4}
\end{equation*}
and
\begin{equation*}
\eta\left( (\LL^2,\R,\LL,\R)^\infty \right)=\frac{2}{5}.
\end{equation*}
\end{example}
As it will detailed below (Section~\ref{sec:period_adding}, see
Remark~\ref{rem:eta-rotation-number}), under certain conditions, the
piecewise-smooth map~\eqref{eq:normal_form} becomes a circle map with
{\em rotation number} the $\eta$-number. Hence, the $\eta$-number as
defined above is frequently referred to as rotation number in the
context of piecewise-smooth maps, even when these conditions are
not satisfied (see for example~\cite{GamTre88,GamGleTre88}).

We now focus on the question of, for a map $f$ of
type~\eqref{eq:normal_form} satisfying~\condsHp{}, what are the
possible periodic orbits, their symbolic sequences and their
bifurcations in the parameter space given by the offsets,
$\mu_\LL\times\mu_\R$.\\
To this end, we first note that, if $\mu_\LL,\mu_\R<0$, as the maps
$f_\LL$ and $f_\R$ are contracting, then $f$ possesses two attracting
coexisting $\LL$ and $\R$-fixed points $x_\LL<0$ and $x_\R>0$:
\begin{align*}
f_\LL(x_\LL)&=x_\LL\\
f_\R(x_\R)&=x_\R.
\end{align*}
The domains of attraction are separated by the boundary $x=0$.
Indeed, if both $f_\LL$ and $f_\R$ are increasing maps, then
these domains become $(-\infty,0)$ and $(0,\infty)$, respectively.
Note that, although $x=0$ is not an invariant point (an equilibrium),
it acts as a separatrix between these domains of attraction.\\
If one of these two parameters vanishes and becomes positive (for
example $\mu_\LL$), the fixed point $x_\LL$ collides with the boundary
$x=0$ and undergoes a border collision bifurcation. Depending on how
the map $f$ given in Equation~\eqref{eq:normal_form} is defined at
$x=0$, at the moment of the bifurcation this fixed point may still
exist or not.  Just after this bifurcation, $x_\LL$ no longer exists,
and the fixed point $x_\R$ becomes the unique global attractor. As
will be discussed below, this fixed point may coexist with a
two-periodic orbit. A similar situation occurs when the parameter
$\mu_\R$ crosses $0$, replacing in the previous argument $x_\LL$ by
$x_\R$. Hence, the origin of this parameter space consists of a
codimension-two bifurcation point.  But then the question arises: what
does exist when both parameters are positive and both fixed points
disappear in border collision bifurcations?  The answer to this
question (summarized in~Section~\ref{sec:inc-inc_overview}
and Section~\ref{sec:inc-dec_overview}) depends on the
signs of the slopes
of the maps $f_\LL$ and $f_\R$ for $x<0$ and $x>0$, respectively.
Recalling that the increasing decreasing and decreasing-cases are
conjugate, we will only distinguish between two cases:
increasing-increasing and increasing-decreasing. These are also
typically referred as orientation preserving and non-orientable
cases.\\
As will be argued in Section~\ref{sec:summarizing_theorem}, the
decreasing-decreasing case is straightforward under the assumption of
contractiveness.

\subsection{Overview of the orientation preserving case: the period adding}\label{sec:inc-inc_overview}
The bifurcation scenario when both $f_\LL$ and $f_\R$ are both
increasing is shown in Figures~\ref{fig:regions_adding}
and~\ref{fig:1dscann}.
\unitlength=\textwidth
\begin{figure}
\begin{center}
\includegraphics[angle=-90,width=0.9\textwidth]{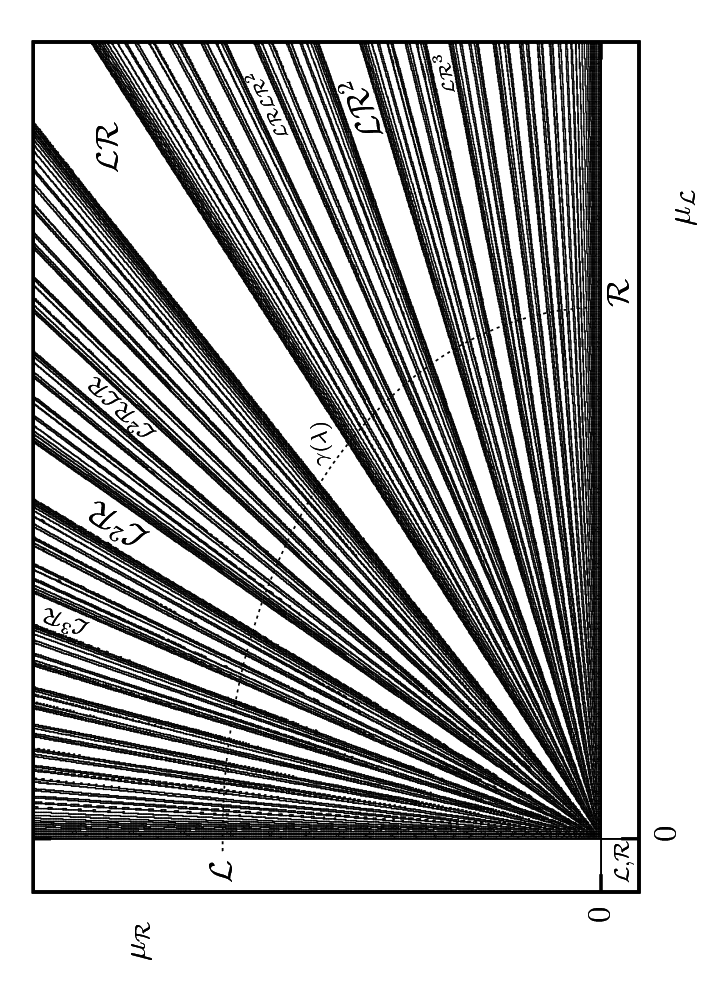}
\end{center}
\caption{Bifurcation curves on the two-dimensional parameter space
$\mu_\LL\times\mu_\R$ for a linear increasing-increasing
piecewise-smooth map. The periods of the periodic orbits found along
the ccurve $\gamma(\lambda)$ are shown in
Figure~\ref{fig:1dscann_periods}.}
\label{fig:regions_adding}
\end{figure}
As shown in Figure~\ref{fig:regions_adding}, there exist an infinite
number of bifurcation curves emerging from the origin of the parameter
space $\mu_\LL\times\mu_\R$. Note that these curves are straight
lines due the fact that the chosen map for the simulations was a
linear one. Obviously, similar non-linear curves are obtained
otherwise. Note also that these curves extend to infinity due to the
fact that the chosen maps are globally contracting. If contractiveness
was lost, other bifurcations would appear.

The bifurcation curves shown in Figure~\ref{fig:regions_adding}
separate regions of existence of periodic orbits. The periods of these
periodic orbits are given by ``successive addition'' of the ones of
``neighbouring regions''\footnote{The term neighbour refers to the
concept of Farey neighbours, see Section~\ref{sec:period_adding}}. We
will make this more precise in Section~\ref{sec:adding_proof}.
However, this can be seen in Figure~\ref{fig:1dscann_periods}, where
we show the periods of the periodic orbits found along the curve shown
in Figure~\ref{fig:regions_adding} parametrized counterclockwise by a
parameter $\lambda$ which will be clarified in
Section~\ref{sec:summarizing_theorem}. As one can also see there, the
symbolic sequences (some of them are labeled) of obtained periodic
orbits are given by successive concatenation of the ``neighbouring''
ones.

Of relevant interest is the evolution of the $\eta$-number (see
Definition~\ref{def:eta-number}) along the mentioned curve in
Figure~\ref{fig:regions_adding}. This is shown in
Figure~\ref{fig:1dscann_eta-number}, and follows a \emph{Devil's
staircase} (a continuous and monotonically increasing function which
is constant locally almost everywhere).

As explained in Section~\ref{sec:introduction}, to our knowledge, this
bifurcation scenario was first described by Leonov in the late 1950s
(\cite{Leonov59,Leonov60a,Leonov60b,Leonov62}), when studying a
piecewise-linear map similar to~\eqref{eq:normal_form} by means of
direct computations. Later on, this phenomenon was studied in more
detail from different perspectives. It was observed when studying
homoclinic bifurcations for flows
(\cite{Spa82,CouGamTre84,GamProThoTre86,TurShi87,ProThoTre87,GamGleTre88,GamTre88,LyuPikZak89,GhrHol94,Hom96})
(see Section~\ref{sec:co-dimension_two} for more details), but also in
electronic circuits given by the Van der Pol oscillator
(\cite{KenChu86,Lev90}) or circle maps~\cite{Kee80,Mira87}.  Later on,
it was rediscovered by the non-smooth community and called
in~\cite{AvrSch06} \emph{period adding}.  This is precisely defined in
Definition~\ref{def:period_adding}.

\begin{figure}
\begin{center}
\begin{picture}(1,0.5)
\put(0,0.35){
\subfigure[\label{fig:1dscann_periods}]{
\includegraphics[angle=-90,width=0.5\textwidth]{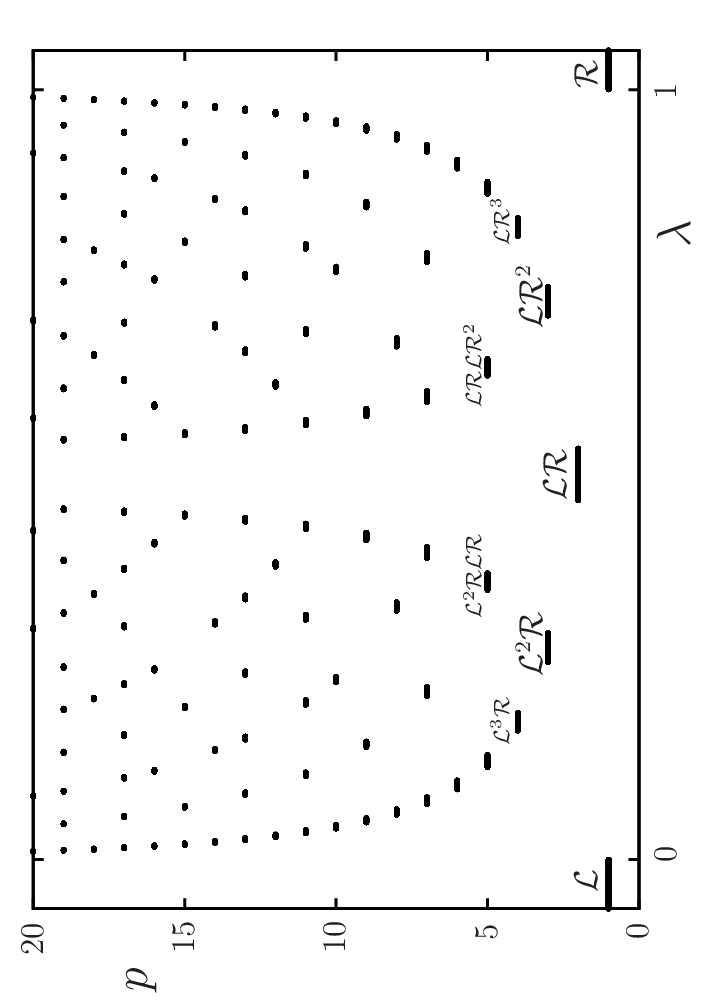}}
}
\put(0.5,0.35){
\subfigure[\label{fig:1dscann_eta-number}]{
\includegraphics[angle=-90,width=0.5\textwidth]{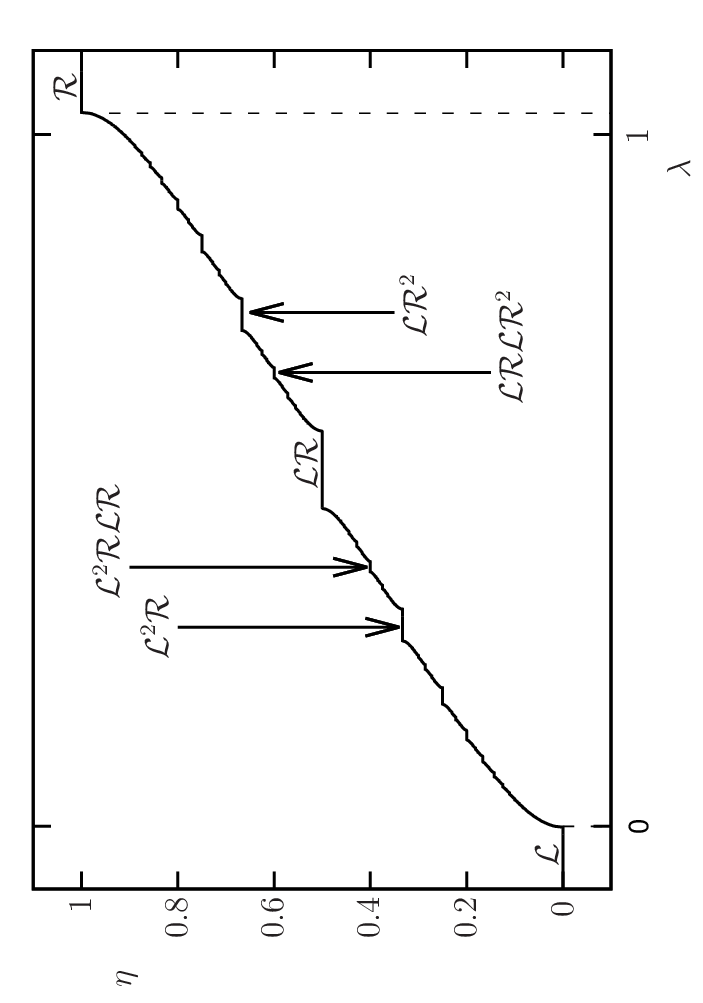}}
}
\end{picture}
\end{center}
\caption{Bifurcation scenario along the curve, $\gamma(\lambda)$,
shown in Figure~\ref{fig:regions_adding} and defined in
Equation~\eqref{eq:1d_scann_curve}. (a) periods of the existent
periodic orbits. (b) $\eta$-number: ratio between number of $\R$'s and
the period.}
\label{fig:1dscann}
\end{figure}

\subsection{Overview of the non-orientable case: the period incrementing}\label{sec:inc-dec_overview}
For the increasing-decreasing case, one finds the bifurcation scenario
shown in Figures~\ref{fig:regions_incrementing}
and~\ref{fig:1dscann_incrementing}. As one can see in
Figure~\ref{fig:regions_incrementing}, as in the previous case, there
exist an infinite number of bifurcation curves emerging from the origin
of the parameter space $\mu_\LL\times\mu_\R$. Also as before, the
chosen map to perform the simulations was linear and globally
contracting. Therefore, the observed bifurcation curves are straight
lines extending to infinity. For a non-linear case these lines would
becomes non-linear curves and, if contractiveness was lost for larger
values of the parameters, new bifurcations would be observed.

Unlike in the orientable case,  only families of periodic orbits of
the form $\LL^n\R$ exist for the non-orientable one.  Moreover, there
exist regions in the parameter space (marked in grey in
Figures~\ref{fig:regions_incrementing}
and~\ref{fig:1dscann_incrementing}) where periodic orbits with
symbolic sequences $\LL^{n}\R$ and $\LL^{n+1}\R$ coexist.
\begin{figure}
\begin{center}
\includegraphics[angle=-90,width=0.9\textwidth]{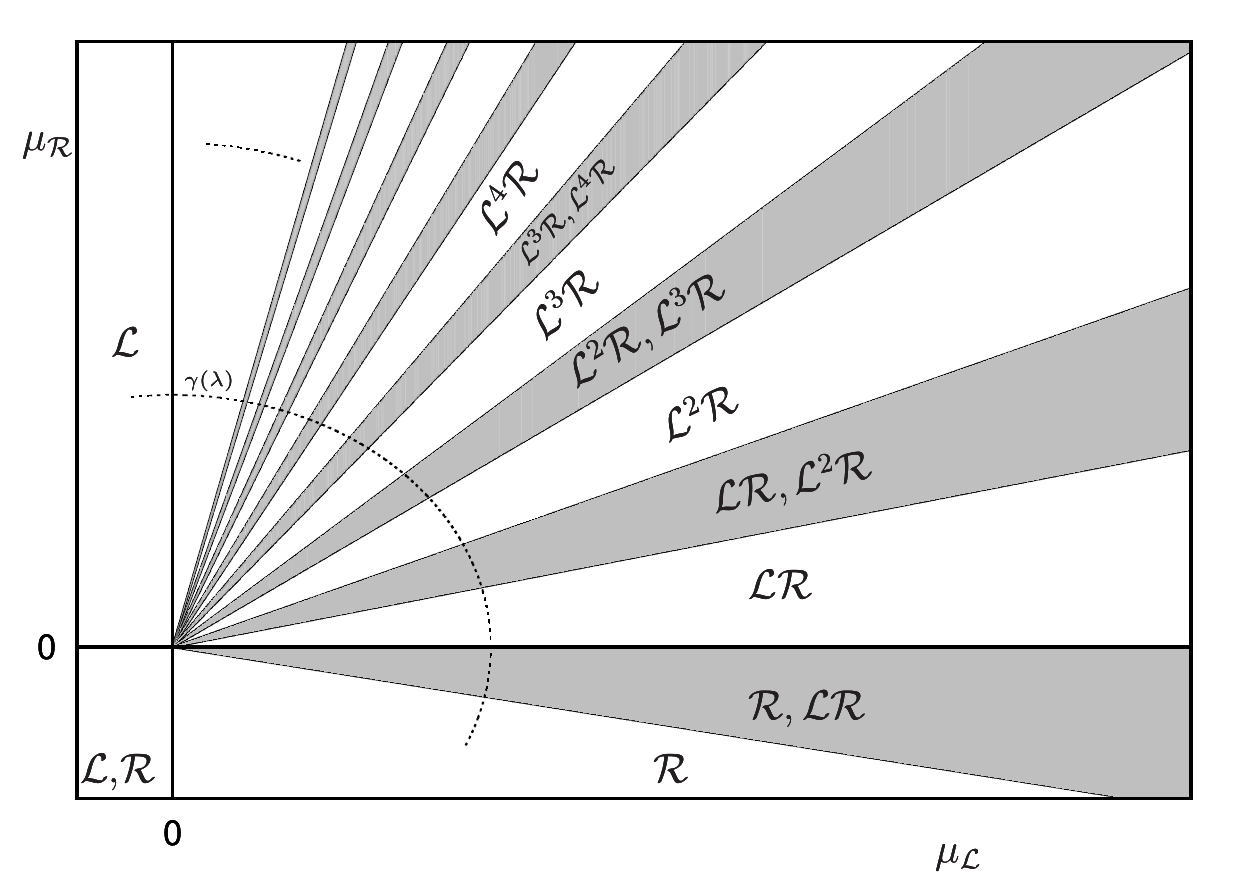}
\end{center}
\caption{Bifurcation curves on the two-dimensional parameter space
$\mu_\LL\times\mu_\R$ for  a linear increasing-decreasing
piecewise-smooth map.  Regions of coexistence are filled in gray. The
periods of the periodic orbits found along the curve $\gamma(\lambda)$
are shown in Figure~\ref{fig:1dscann_periods_incrementing}.}
\label{fig:regions_incrementing}
\end{figure}

\begin{figure}
\begin{center}
\begin{picture}(1,0.5)
\put(0,0.35){
\subfigure[\label{fig:1dscann_periods_incrementing}]{
\includegraphics[angle=-90,width=0.5\textwidth]{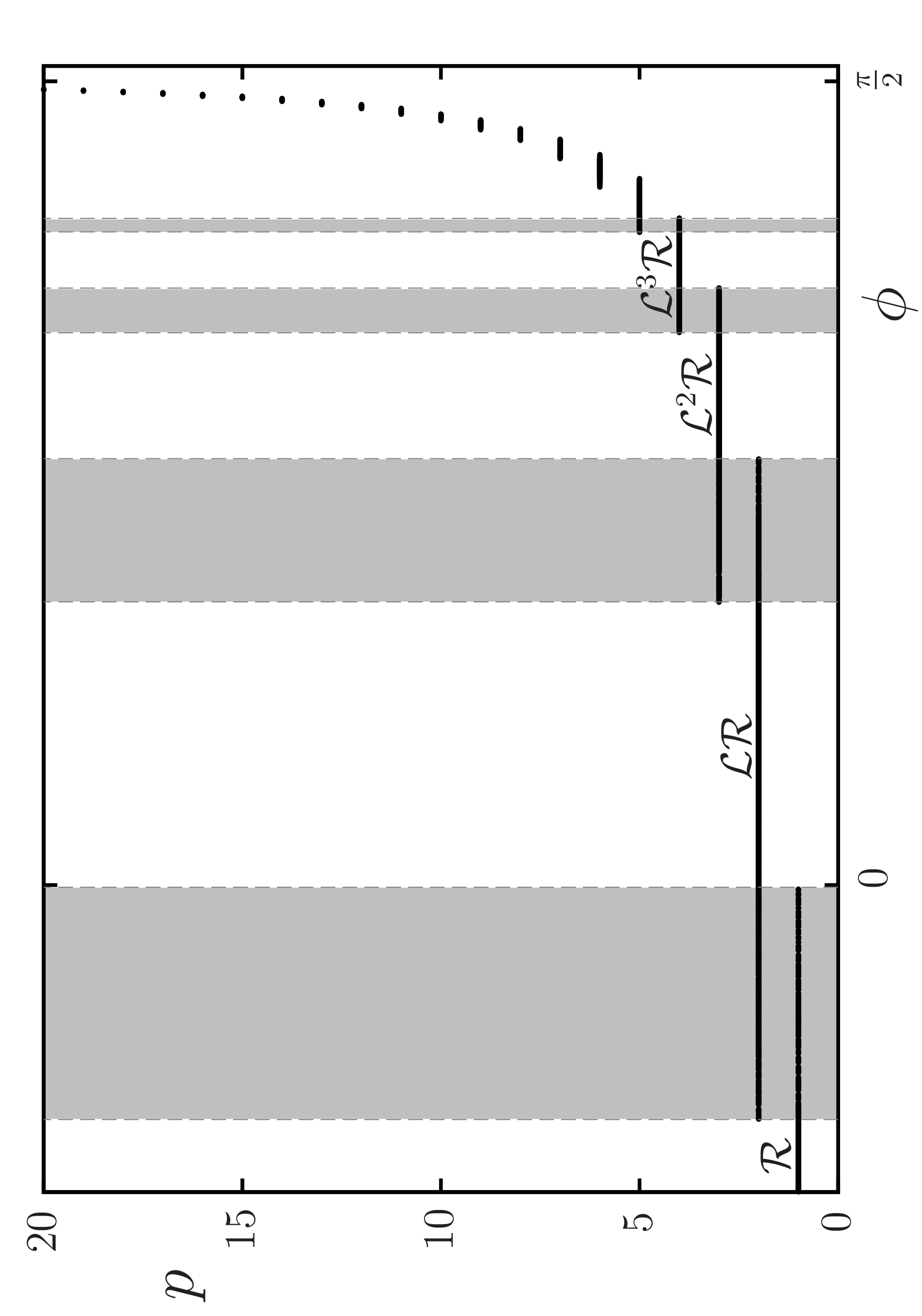}}
}
\put(0.5,0.35){
\subfigure[\label{fig:1dscann_eta-number_incrementing}]{
\includegraphics[angle=-90,width=0.5\textwidth]{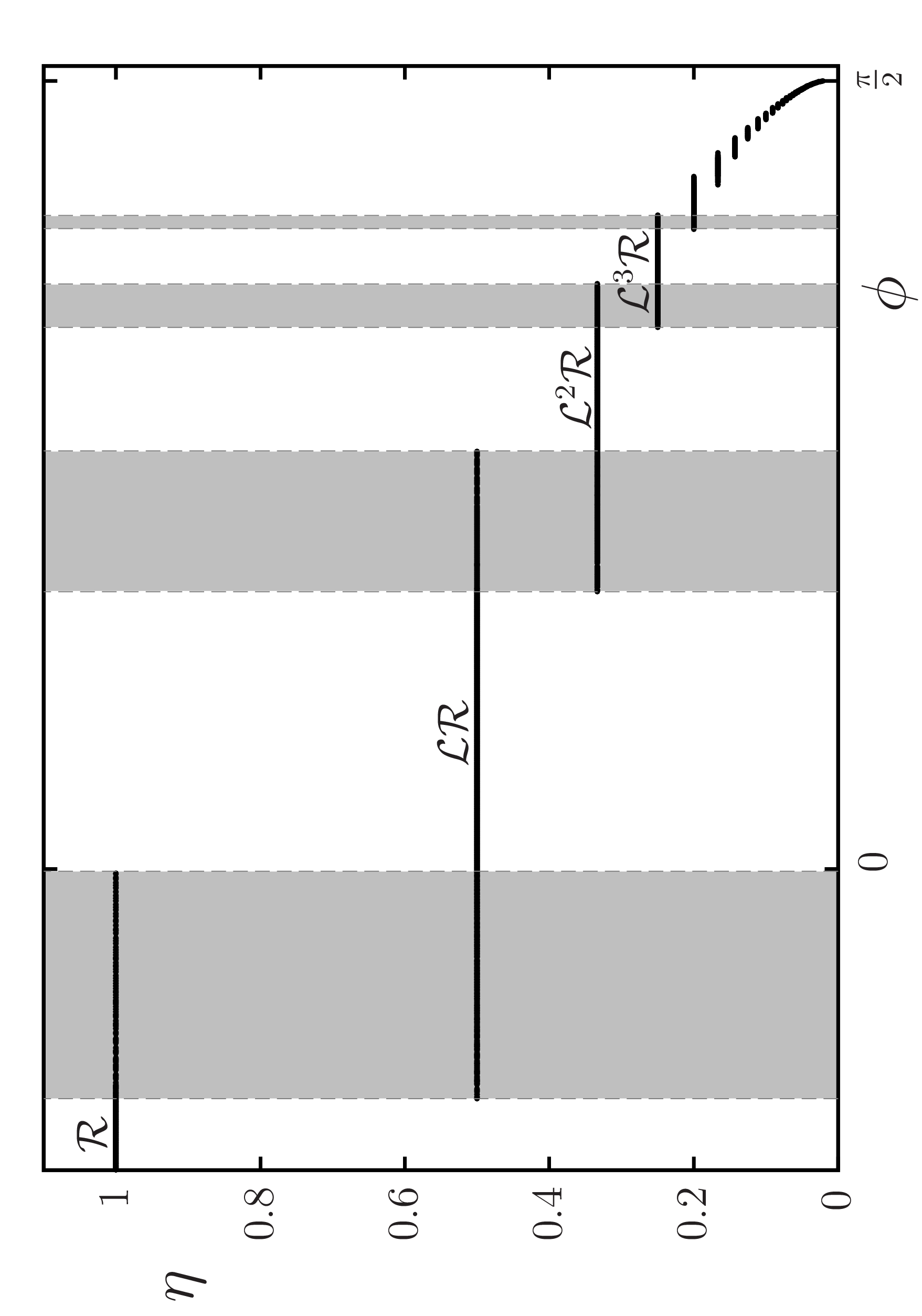}}
}
\end{picture}
\end{center}
\caption{Bifurcation scenario along the curve, $\gamma(\lambda)$,
shown in Figure~\ref{fig:regions_incrementing}. (a) periods of the
existent periodic orbits. Regions where periodic orbits coexist are
filled in gray. (b) $\eta$-number: ratio between number of $\R$'s and
the period.}
\label{fig:1dscann_incrementing}
\end{figure}

As mentioned in Section~\ref{sec:introduction}, this bifurcation
scenario was first described by Leonov
\cite{Leonov59,Leonov60a,Leonov60b,Leonov62}. Later on, it was studied
due to its relevance in homoclinic bifurcations involving
non-orientable homoclinic manifolds~\cite{Hom96,GhrHol94}. It was
rediscovered in~\cite{AvrSch06} when studying a linear
piecewise-smooth map and named \emph{period incrementing}. Full
details proving that the increasing-decreasing case leads to the
occurrence of the period-incrementing bifurcation scenario were given
in~\cite{AvrGraSch11}.

\subsection{Summarizing theorem}\label{sec:summarizing_theorem}
The results presented in the previous sections are summarized in the
following
\begin{theorem}\label{theo:adding_incrementing}
Let $f$ be a map as in Equation~\eqref{eq:normal_form} satisfying
conditions~\condsHp{}.  Let $\gamma$ be a $C^1$ curve in the parameter
space $\mu_\LL\times\mu_\R$ (see Figures~\ref{fig:regions_adding}
and~\ref{fig:regions_incrementing})
\begin{equation}
\begin{array}{cccc}
\gamma:&[0,1]&\longrightarrow&\RR^2\\
&\lambda&\longmapsto &(\mu_\LL(\lambda),\mu_\R(\lambda))
\end{array}
\label{eq:1d_scann_curve}
\end{equation}
satisfying
\begin{enumerate}[H.1]
\item $\mu_\LL(\lambda)> 0$ and $\mu_\R(\lambda)>0$ for $\lambda\in(0,1)$
\item $\left( \mu_\LL(\lambda) \right)'>0$ and $\left( \mu_\R(\lambda) \right)'<0$
for $\lambda\in[0,1]$
\item $\mu_\LL(0)=0$, $\mu_\R(1)=0$
\end{enumerate}
Then, the bifurcation diagram exhibited by the map $f_\lambda$
obtained from Equation~\eqref{eq:normal_form} after performing the
reparametrization given by $\gamma$, follows a
\begin{enumerate}[i)]
\item period adding structure if $0<f'_\R(x)<1,\,x\in(0,\infty)$
\item period incrementing structure if $-1<f'_\R(x)<0,\,x\in(0,\infty)$
\end{enumerate}
for $\lambda\in[0,1]$.\\
\end{theorem}

For a description of the bifurcation scenarios announced in \emph{i)}
and \emph{ii)}, the period adding and period incrementing, see,
respectively, Sections~\ref{sec:inc-inc_overview}
and~\ref{sec:inc-dec_overview} for an overview, and
Sections~\ref{sec:adding_proof} and~\ref{sec:incrementing_proof} for
more details and proofs.

We now explain briefly the case when condition {\em h.2} is not
satisfied ($f_\LL$ is decreasing):
\begin{remark}\label{rem:conds_iii-iv}
If condition h.2 is not satisfied and $f'_\LL(x)<0$ in $(-\infty,0)$,
then,
\begin{enumerate}[i)]
\setcounter{enumi}{2}
\item if $f'_\R(x)<0$ for $(0,\infty)$, only an $\LL\R$-periodic orbit
can exist for all $\lambda\in(0,1)$,
\item if $f'_\R(x)>0$, for $x\in(-\infty,0)$ the bifurcation scenario
is equivalent to ii) but interchanging $\LL$ and $\R$ in the symbolic
dynamics.
\end{enumerate}
\end{remark}
Clearly, {\it iii)} comes from the fact that, under these
conditions, $f_\LL( (-\infty,0))\subset (\mu_\LL,\infty)$ and $f_\R(
(0,\infty))\subset (-\infty,-\mu_\R)$, and hence, due to the contractiveness
of these maps,  $f$ must possess an $\LL\R$-periodic orbit.\\
The fact that the cases \emph{iv)} and \emph{ii)} are conjugate comes
from applying the symmetries given by the change of variables
$x\longleftrightarrow -x$.

\section{Orientation preserving case}\label{sec:adding_proof}
\subsection{Detailed description}\label{sec:period_adding}
We first provide a detailed description of the period adding
bifurcation structure by stating some results which will be proved in
the rest of this section.\\

The bifurcation structure given by the so-called \emph{period adding}
is strongly linked with the ordering of the rational numbers given by
the Farey tree. In order to explain how this tree is generated, we
first define the Farey neighbours. Recall that a rational
number $p/q$ is irreducible if $(p,q)=1$.

\begin{definition}\label{def:farey_sequence}
We define the Farey sequence of order $n$ ($n\ge 1$),
$\mathcal{F}_n$, as the succession of rational numbers $p/q\in[0,1]$
in ascending order, starting with $0/1$ and ending with $1/1$, such
that $(p,q)=1$ and $q\le n$.
\end{definition}
For example, the Farey sequence of order $6$ becomes
\begin{equation*}
\frac{0}{1},\,\frac{1}{6},\,\frac{1}{5},\,\frac{1}{4},\,\frac{1}{3},\,\frac{2}{5},\,\frac{1}{2},\,\frac{3}{5},\,\frac{2}{3},\,\frac{3}{4},\,\frac{4}{5},\,\frac{5}{6},\,\frac{1}{1}.
\end{equation*}
Note that all rational numbers are contained in $\mathcal{F}_n$ for
some $n>0$. We then have the following well known result
(see~\cite{HarWri60}).
\begin{theorem}\label{the:Farey_properties}
The following holds:
\begin{enumerate}[i)]
\item $p/q<r/s$ are Farey neighbours in $\mathcal{F}_n$ if, and only
if, $rq-ps=1$
\item if $p/q<r/s$ then
\begin{equation*}
\frac{p}{q}<\frac{p+r}{q+s}<\frac{r}{s}.
\end{equation*}
Moreover, if $p/q$ and $r/s$ are Farey neighbours at $\mathcal{F}_n$
then $p/q<(p+r)/(q+s)$ and $(p+r)/(q+s)<\frac{r}{s}$ are Farey
neighbours at $\mathcal{F}_{q+s}$.
\end{enumerate}
\end{theorem}
\begin{definition}
The fraction $\frac{p+r}{q+s}$ given in
Theorem~\ref{the:Farey_properties} is called the Farey sum or mediant
of the numbers $\frac{p}{q}$ and $\frac{r}{s}$.
\end{definition}

\unitlength=\textwidth
\begin{figure}
\begin{center}
\begin{picture}(1,0.5)
\put(0,0.015){
\subfigure[\label{fig:farey_rho}]{
\begin{picture}(0.5,0.5)
\put(-0.01,0.06){
\scalebox{0.65}{
\begin{tikzpicture}[->,>=stealth',shorten >=1pt,auto,node distance=1cm]
\node[] (a1){};
\node[] (a2) [right of=a1] {};
\node[] (a3) [right of=a2] {};
\node[] (a4) [right of=a3] {$0$};
\node[] (a5) [right of=a4] {};
\node[] (a6) [right of=a5] {$1$};
\node[] (a7) [right of=a6] {};
\node[] (a8) [right of=a7] {};
\node[] (a9) [right of=a8] {};
\node[] (b1) [above of=a1] {};
\node[] (b2) [right of=b1] {};
\node[] (b3) [right of=b2] {};
\node[] (b4) [right of=b3] {};
\node[] (b5) [right of=b4] {$\frac{1}{2}$};
\node[] (b6) [right of=b5] {};
\node[] (b7) [right of=b6] {};
\node[] (b8) [right of=b7] {};
\node[] (b9) [right of=b8] {};
\node[] (c1) [above of=b1] {};
\node[] (c2) [right of=c1] {};
\node[] (c3) [right of=c2] {$\frac{1}{3}$};
\node[] (c4) [right of=c3] {};
\node[] (c5) [right of=c4] {};
\node[] (c6) [right of=c5] {};
\node[] (c7) [right of=c6] {$\frac{2}{3}$};
\node[] (c8) [right of=c7] {};
\node[] (c9) [right of=c8] {};
\node[] (d1) [above of=c1] {};
\node[] (d2) [right of=d1] {$\frac{1}{4}$};
\node[] (d3) [right of=d2] {};
\node[] (d4) [right of=d3] {};
\node[] (d5) [right of=d4] {};
\node[] (d6) [right of=d5] {};
\node[] (d7) [right of=d6] {};
\node[] (d8) [right of=d7] {$\frac{3}{4}$};
\node[] (d9) [right of=d8] {};
\node[] (e1) [above of=d1] {$\frac{1}{5}$};
\node[] (e2) [right of=e1] {};
\node[] (e3) [right of=e2] {};
\node[] (e4) [right of=e3] {$\frac{2}{5}$};
\node[] (e5) [right of=e4] {};
\node[] (e6) [right of=e5] {$\frac{3}{5}$};
\node[] (e7) [right of=e6] {};
\node[] (e8) [right of=e7] {};
\node[] (e9) [right of=e8] {$\frac{4}{5}$};
\path[]
(a4) edge node [left] {} (b5)
(a6) edge node [left] {} (b5)
(b5) edge node [left] {} (c3)
(a4) edge node [left] {} (c3)
(b5) edge node [left] {} (c7)
(a6) edge node [] {} (c7)
(c7) edge node [] {} (d8)
(a6) edge node [] {} (d8)
(a4) edge node [left] {} (d2)
(c3) edge node [left] {} (d2)
(a4) edge node [left] {} (e1)
(d2) edge node [left] {} (e1)
(a6) edge node [] {} (e9)
(d8) edge node [] {} (e9)
(b5) edge node [] {} (e4)
(c3) edge node [] {} (e4)
(b5) edge node [] {} (e6)
(c7) edge node [] {} (e6)
;
\end{tikzpicture}
}
}
\end{picture}
}
}
\put(0.5,0.015){
\subfigure[\label{fig:farey_sequences}]{
\begin{picture}(0.5,0.5)
\put(-0.01,0.06){
\scalebox{0.65}{
\begin{tikzpicture}[->,>=stealth',shorten >=1pt,auto,node distance=1cm]
\node[] (a1){};
\node[] (a2) [right of=a1] {};
\node[] (a3) [right of=a2] {};
\node[] (a4) [right of=a3] {$\LL$};
\node[] (a5) [right of=a4] {};
\node[] (a6) [right of=a5] {$\R$};
\node[] (a7) [right of=a6] {};
\node[] (a8) [right of=a7] {};
\node[] (a9) [right of=a8] {};
\node[] (b1) [above of=a1] {};
\node[] (b2) [right of=b1] {};
\node[] (b3) [right of=b2] {};
\node[] (b4) [right of=b3] {};
\node[] (b5) [right of=b4] {$\LL\R$};
\node[] (b6) [right of=b5] {};
\node[] (b7) [right of=b6] {};
\node[] (b8) [right of=b7] {};
\node[] (b9) [right of=b8] {};
\node[] (c1) [above of=b1] {};
\node[] (c2) [right of=c1] {};
\node[] (c3) [right of=c2] {$\LL^2\R$};
\node[] (c4) [right of=c3] {};
\node[] (c5) [right of=c4] {};
\node[] (c6) [right of=c5] {};
\node[] (c7) [right of=c6] {$\LL\R^2$};
\node[] (c8) [right of=c7] {};
\node[] (c9) [right of=c8] {};
\node[] (d1) [above of=c1] {};
\node[] (d2) [right of=d1] {$\LL^3\R$};
\node[] (d3) [right of=d2] {};
\node[] (d4) [right of=d3] {};
\node[] (d5) [right of=d4] {};
\node[] (d6) [right of=d5] {};
\node[] (d7) [right of=d6] {};
\node[] (d8) [right of=d7] {$\LL\R^3$};
\node[] (d9) [right of=d8] {};
\node[] (e1) [above of=d1] {$\LL^4\R$};
\node[] (e2) [right of=e1] {};
\node[] (e3) [right of=e2] {};
\node[] (e4) [right of=e3] {$\LL^2\R\LL\R$};
\node[] (e5) [right of=e4] {};
\node[] (e6) [right of=e5] {$\LL\R\LL\R^2$};
\node[] (e7) [right of=e6] {};
\node[] (e8) [right of=e7] {};
\node[] (e9) [right of=e8] {$\LL\R^4$};
\path[]
(a4) edge node [left] {} (b5)
(a6) edge node [left] {} (b5)
(b5) edge node [left] {} (c3)
(a4) edge node [left] {} (c3)
(b5) edge node [left] {} (c7)
(a6) edge node [] {} (c7)
(c7) edge node [] {} (d8)
(a6) edge node [] {} (d8)
(a4) edge node [left] {} (d2)
(c3) edge node [left] {} (d2)
(a4) edge node [left] {} (e1)
(d2) edge node [left] {} (e1)
(a6) edge node [] {} (e9)
(d8) edge node [] {} (e9)
(b5) edge node [] {} (e4)
(c3) edge node [] {} (e4)
(b5) edge node [] {} (e6)
(c7) edge node [] {} (e6)
;
\end{tikzpicture}
}
}
\end{picture}
}
}
\end{picture}
\caption{(a) Farey tree of rational numbers. (b) Farey tree of
symbolic sequences. Both trees are isomorphic by construction but also
through dynamical properties of periodic orbits of orientation
preserving circle maps (see Section~\ref{sec:dyna_orient_preserv}.)}
\end{center}
\end{figure}
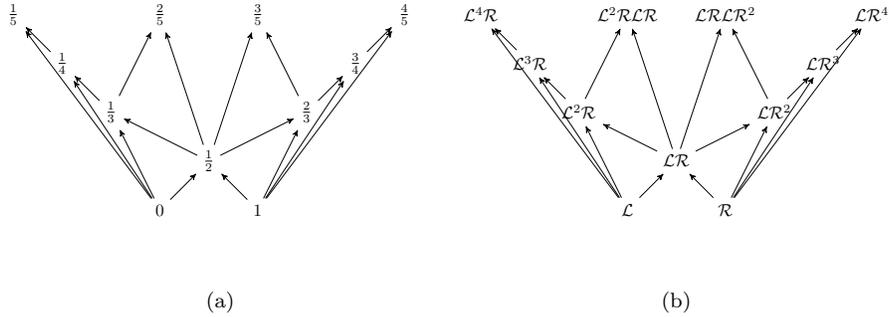

We now construct the Farey tree of rational numbers (see
Figure~\ref{fig:farey_rho}).  Starting with the Farey neighbours $0/1$
and $1/1$, the Farey tree is generated by obtaining rational numbers
by adding their numerators and denominators. That is, given two Farey
neighbours $p/q$ and $p'/q'$, they generate the child given by their
mediant, $(p+p')/(q+q')$, and $p/q$ and $p'/q'$ become the Farey
parents of the rational number $(p+p')/(q'+q')$. Note that this
provides all the Farey sequences $\mathcal{F}_n$ and, hence, all
rational numbers are found in the Farey tree and are uniquely
identified with their Farey parents.

Consder a piecewise-smooth map as defined in
Equation~\eqref{eq:normal_form}, and recall that we are interested in
the simultaneous variation of the parameters $\mu_\LL$ and $\mu_\R$
parametrized by $\lambda\in[0,1]$ through the re-parametrization shown
in Equation~\eqref{eq:1d_scann_curve}. Then we get that, under the
conditions {\em i)} of Theorem~\ref{theo:adding_incrementing}, when
this parameter is varied from $0$ to $1$, the periods of these
periodic are given by the denominators of the rational numbers numbers
given in the Farey tree (see Figure~\ref{fig:farey_rho}).\\
As noted by some authors~\cite{FreGal11}, other trees, as the
Stern-Brocot, can also generate this sequence of periods. However, the
most interesting relation between the Farey tree and the sequence of
periodic orbits given by the period adding regards their associated
symbolic sequences, and not only their periods. To explain this, we
recall that the $\eta$-number (Definition~\ref{def:eta-number}) is
given by the ratio between the number of $\R's$ contained in a
symbolic sequence and its length.
\begin{remark}\label{rem:eta-rotation-number}
As we will show in Section~\ref{sec:dyna_orient_preserv}
(Corollary~\ref{cor:eta-number_rotation-number}), under conditions i)
of Theorem~\ref{theo:adding_incrementing}, the map $f$ given in
Equation~\eqref{eq:normal_form} can be reduced to a circle map and the
number $\eta$ of a symbolic sequence of a periodic orbit becomes the
rotation number of the circle map (see
Remark~\ref{rem:rotation-eta_number}).
\end{remark}

Under conditions {\em i)} of Theorem~\ref{theo:adding_incrementing},
when $\lambda$ in Equation~\eqref{eq:1d_scann_curve} is varied from
$0$ to $1$, one gets periodic orbits with symbolic sequences
$\x^\infty\in W$, with $\x=\x(\lambda)\in W_{p,q}$. If one then
considers the $\eta$-number associated with these symbolic sequences,
$\eta(\lambda)$, it turns out that it is a continuous and
monotonically increasing function of $\lambda$ whose image are all
rational values between $0$ and $1$.  However, the set of values of
$\lambda$ for which $\eta(\lambda)$ is is not defined (it ``becomes
irrational'') forms a Cantor set of zero measure. Hence,
$\eta(\lambda)$ is a {\em devil's staircase}, as it is continuous,
monotonically increasing and locally constant almost everywhere. This
function is shown in Figure~\ref{fig:1dscann_eta-number}, and is the
well known one formed by the rotation numbers of the periodic orbits
of the so-called Arnold circle map
\begin{equation*}
\theta_{n+1}=\theta_n+\Omega -\frac{1}{2\pi}\sin\left( 2\pi\theta_n \right)
\end{equation*}
when $\Omega$ is varied from $0$ to $1$.\\
Note that , for $\eta=0$ and $\eta=1$, the map possesses a fixed
points with symbolic sequence $\LL$ and $\R$, respectively. These
undergo border collision bifurcations for $\lambda=0$ and $\lambda=1$,
respectively.\\

Although each periodic orbit for $\lambda\in(0,1)$ is in one-to-one
correspondence with a rational number in the Farey tree through the
$\eta$-number, their symbolic sequences are, in principle, not
uniquely identified. For example, assume that for a certain value
$\lambda=\lambda_{2/5}$ there exists a periodic orbit with $\eta=2/5$.
Its symbolic sequence could be given by any of the generating
minimal blocks $\LL^2\R\LL\R$ or $\LL^3\R^2$, as they have length $5$
and contain two $\R$ symbols. It turns out that the symbolic sequence
of a periodic orbits with $\eta$-number $p/q$ i uniquely determined.
To explain this correspondance between periodic orbits and symbolic
sequences we construct the Farey tree of symbolic sequences
as follows. Starting with the sequences $\LL$ and
$\R$, one concatenates those sequences whose $\eta$-numbers are Farey
numbers. By construction, this provides a unique correspondence
between rational numbers and symbolic sequences through their
$\eta$-number and the Farey tree. That is, to each rational number
$P/Q$ in the Farey tree one associates a symbolic sequence $\Delta$
given by the concatenation
\begin{equation*}
\Delta=\alpha\beta,
\end{equation*}
where $\alpha$ and $\beta$ are the (minimal) symbolic sequences of the
Farey neighbours rational numbers $p/q$ and $p'/q'$, respectively:
\begin{itemize}
\item $p/q<p'/q'$
\item $p'q-pq'=1$
\item $P/Q=(p+p')/(q+q')$
\end{itemize}
As will be discussed in Section~\ref{sec:maximin_approach}, this
concatenation provides the so-called {\it maximin} sequences
(see Definition~\ref{def:maximin}).\\
We will prove in Section~\ref{sec:dyna_orient_preserv}
(Propositions~\ref{pro:well_ordered_symb-seq}
and~\ref{pro:concatenatetion}) that the following correspondance
holds. The symbolic sequence of the periodic orbit with $\eta$-number
$p/q$ is the one which corresponds to rational number $p/q$ in the
Farey tree of symblic sequences.\\

We are now ready to provide a formal description of the period adding
bifurcation structure.
\begin{definition}\label{def:period_adding}
We say that $f_\lambda$, as in Theorem~\ref{theo:adding_incrementing}
i), undergoes a period adding bifurcation structure if
\begin{enumerate}[i)]
\item for all values of $\lambda$, except for a Cantor set with zero
measure, the map $f_\lambda$ possesses a unique attracting periodic
orbit.
\item the $\eta$-number, $\eta\in\Q$, as a function of $\lambda$
follows a devil's staircase: it is a monotonically increasing
continuous function not defined in a Cantor set of zero measure (it is
hence locally constant for almost all values of $\lambda$)
\item if $f_{\lambda_\Delta}$ possesses a $\Delta$-periodic orbit with
$\eta=P/Q$, $(P,Q)=1$, then there exist values
$\lambda_\alpha<\lambda_\Delta<\lambda_\beta$ such that
$f_{\lambda_\alpha}$ and $f_{\lambda_\beta}$ possess $\alpha$ and
$\beta$-periodic orbits, respectively, whose $\eta$-numbers are Farey
neighbours and their mediant is $P/Q$.  Moreover, $\Delta$ is the
concatenation of $\alpha$ and $\beta$: $\Delta=\alpha\beta$
\end{enumerate}
\end{definition}

\subsection{Summary of the proof}\label{sec:summary}
In order to facilitate the lecture of the proof of the result
described in detail in Section~\ref{sec:period_adding}, we provide a
schematic summary of the necessary steps.
\begin{enumerate}[1)]
\item By performing the change of variables given in
Equation~\eqref{eq:change_variables} we first show that a piecewise-defined
map $f$ as in~\eqref{eq:normal_form} satisfying {\em i)} in
Theorem~\ref{theo:adding_incrementing} is an orientation preserving
circle map (see Definition~\ref{def:orient_preserving}). Although this
circle map will be discontinuous, in
Section~\ref{sec:properties_circle_maps} we will show how classical
results for continuous circle maps also hold. In particular, we will
show:
\begin{enumerate}
\item Given a map as above, its rotation number exists and is unique
(Proposition~\ref{pro:existence_rotation_number}).
\item If a map as above has a rational rotation number, then it
possesses a unique and stable periodic orbit
(Proposition~\ref{pro:po_rational_rho} + contractivness).
\item When a map as above has a rational rotation number, the unique
periodic orbit must be $p,q$-ordered (it is a twist orbit, see
Definition~\ref{def:well-ordered_sequence_of_points})
(Proposition.~\ref{pro:pq-ordred}).
\end{enumerate}
\item In Section~\ref{sec:dyna_orient_preserv} we provide symbolic
properties of the itineraries of periodic orbits of orientation
preserving circle maps. More precisely:
\begin{enumerate}
\item In Proposition~\ref{pro:well_ordered_symb-seq} we show that the
symbolic itinerary of a twist periodic orbit  is a $p,q$-ordered
symbolic sequence. This identifies each periodic orbit of an
orientation preserving map with a unique symbolic itinerary through
its rotation number.
\item In Proposition~\ref{pro:concatenatetion} we show that the
itineraries of $p,q$-ordered periodic orbits are given by
concatenation  of the itineraries of the periodic orbits with rotation
numbers the Farey parents of $p/q$. Hence, they are  in the Farey
tree of symbolic sequences shown in Figure~\ref{fig:farey_sequences}.
\end{enumerate}
\item In Section~\ref{sec:dyna_orient_preserv} we also
study one-parameter
families of orientation preserving (discontinuous) circle maps. This
is equivalent to varying the parameter $\lambda$ under the conditions of
Theorem~\ref{theo:adding_incrementing} {\em i)}. Using the continuity
of the rotation number (Proposition~\ref{pro:continuity_rot_num}) and
a result of Boyd (Theorem~\ref{theo:Boy85}), we show that the rotation
number (and hence the $\eta$-number) follows a devil's staircase
leading to the adding scenario when $\lambda$ is varied from
$0$ to $1$.
\end{enumerate}
We emphasize that the previous steps provide (to our knowledge) the
shortest path to prove {\em i)} of
Theorem~\ref{theo:adding_incrementing}. However, it is not the only
one.  In Section~\ref{sec:maximin_1d} we provide an alternative to
cover some of the steps mentioned above. These involve the concept of
{\em maximin} sequences and {\em quasi-contractions}.
\subsection{Reduction to an orientation preserving circle map and some properties}\label{sec:properties_circle_maps}
In this section we first show that, under condition {\em i)} of
Theorem~\ref{theo:adding_incrementing}, the piecewise-smooth
map~\eqref{eq:normal_form} can be reduced to a class of orientation
preserving (increasing) discontinuous circle maps.\\
Then, we present results on circle maps which are well known in the
continuous case. However, by making little modifications of the
classical proofs we adapt them to the discontinuous case. We will make
this clear in each particular situation.

\begin{figure}
\begin{center}
\begin{picture}(1,0.5)
\put(0,0.35){
\subfigure[\label{fig:circle_reduction_a}]{
\includegraphics[angle=-90,width=0.5\textwidth]{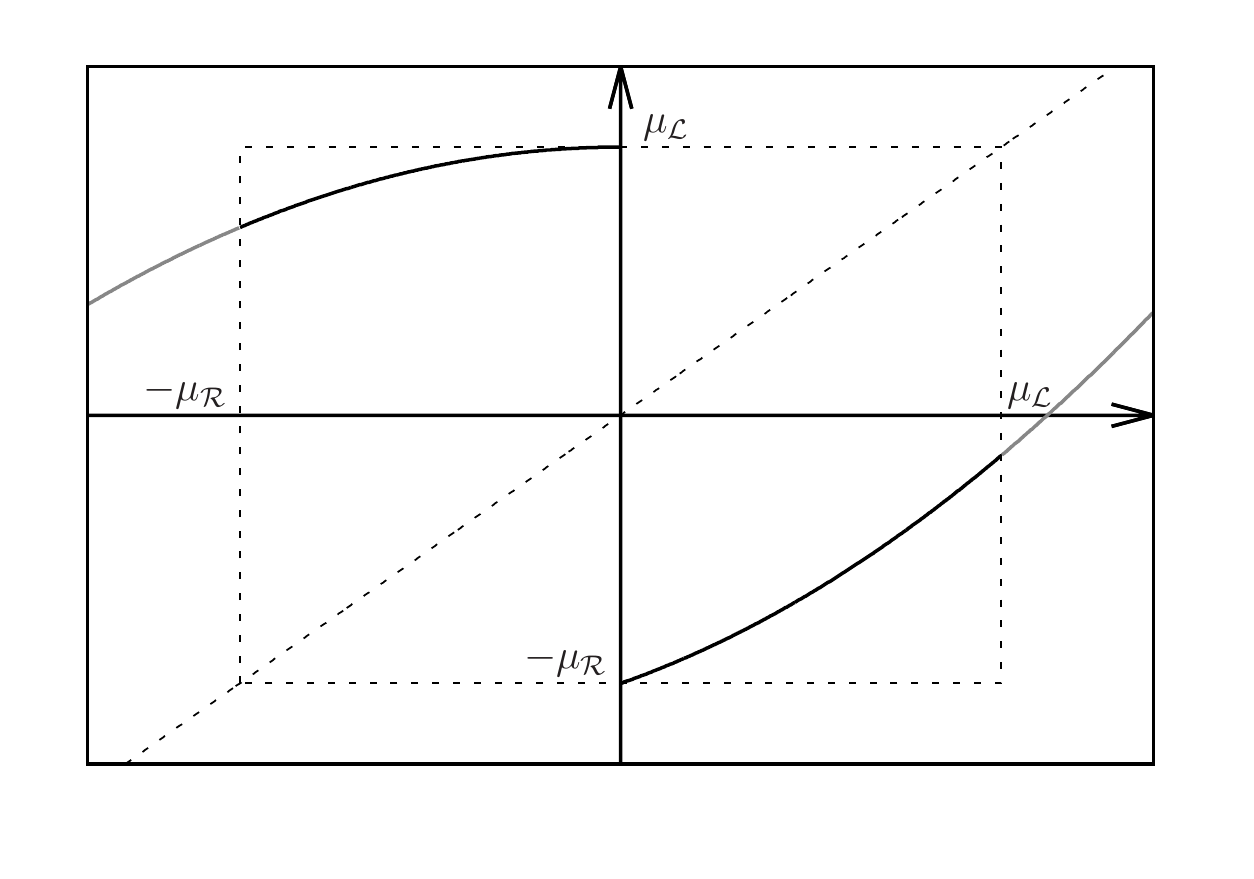}}
}
\put(0.5,0.35){
\subfigure[\label{fig:circle_reduction_b}]{
\includegraphics[angle=-90,width=0.5\textwidth]{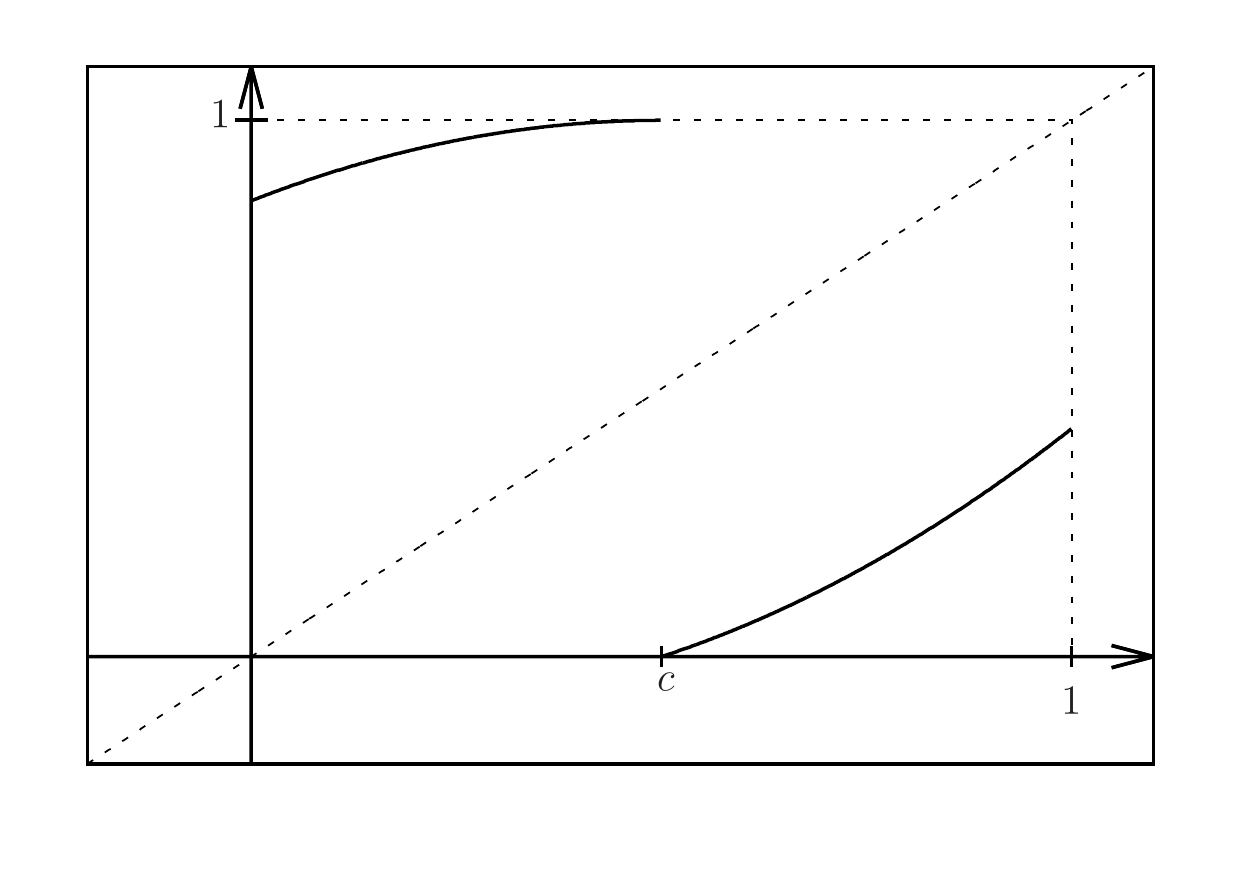}}
}
\end{picture}
\end{center}
\caption{(a) reduction of an increasing-increasing piecewise-smooth
map to an orientation preserving circle map. (b) circle map after the
change of variables given in Equation~\eqref{eq:change_variables}.}
\label{fig:circle_reduction}
\end{figure}

Let us observe that, under condition {\em i)} of
Theorem~\ref{theo:adding_incrementing}, the piecewise-smooth
map~\eqref{eq:normal_form} is increasing on both sides of the
discontinuity $x=0$. Hence, all the dynamics are attracted into the
interval $[f_\LL(0),f_\R(0)]=[-\mu_\R,\mu_\LL]$ (see
Figure~\ref{fig:circle_reduction}). By identifying these two values,
and considering the circle as $\RR/[-\mu_\R,\mu_\LL]$, the map becomes
a circle map which is continuous at $x=0$, but not necessarily at
$x=-\mu_\R\sim\mu_\LL$ (see Figure~\ref{fig:circle_reduction_a}). As
we are interested in varying the parameters $\mu_\R$ and $\mu_\LL$, we
perform the change of variables
\begin{equation}
\begin{array}[]{cccc}
\phi:&[-\mu_R,\mu_\LL]&\longrightarrow&[0,1]\\
&x&\longmapsto&\frac{x+\mu_\R}{\mu_\LL+\mu_R},
\end{array}
\label{eq:change_variables}
\end{equation}
which is a strictly increasing homeomorphism  mapping $-\mu_\R$ to $0$
and $\mu_\LL$ to $1$. Of special interest will be the value
\begin{equation*}
c=\phi(0)=\frac{\mu_\R}{\mu_\R+\mu_\LL},
\end{equation*}
which separates the behaviour given by $\phi\circ f_\LL\circ
\phi^{-1}$ and $\phi\circ f_\R\circ\phi^{-1}$.\\
Hence, by identifying the circle with the interval $[0,1]$,
$S^1=\RR/\Z$, we have reduced the piecewise-smooth map to a class of
circle maps, which consists of increasing (orientation preserving)
circle maps reaching exactly once the value $1$ (degree $1$) and not
necessarily continuous at $x=0\sim1$.  In the following definition we
make precise such class of maps.
\begin{definition}\label{def:orient_preserving}
We say that
\begin{equation}
f:\T\longrightarrow \T,
\label{eq:circle_map}
\end{equation}
$\T=\RR/\Z$, is an orientation preserving circle map (of degree one)
if there exists a unique $c\in[0,1]$ (where $[0,1]$ is identified with
the circle $S^1$) such that
\begin{enumerate}[C.1]
\item $f$ is $C^0$ in $(0,c)$ and $(c,1)$
\item there exists $c\in[0,1]$ such that $f$ is increasing in
$[0,c)$ and $(c,1]$
\item $f(c^-)=1$ and $f(c^+)=0$
\item $f(0^+)\ge f(1^-)$,
\end{enumerate}
\end{definition}
Due to the existence of $c$ fulfilling {\it C.3}, such a circle map is of
degree one, as the image of $\T$ by $f$ twists at most once around
$\T$. When also considering condition {\it C.4} we ensure that such an
orientation preserving map is invertible. However, condition {\it C.4}
allows this class of maps to be not necessary continuous at $x=0$.
Hence, at $x=0$ and $x=1$ one can choose between the image from the
left or right of $x=0$ (or indeed any other value). When convenient,
we will choose both values at $x=0$ and $x=1$ and deal with a bi-valued
function.
\begin{figure}
\begin{center}
\includegraphics[angle=-90,width=0.7\textwidth]{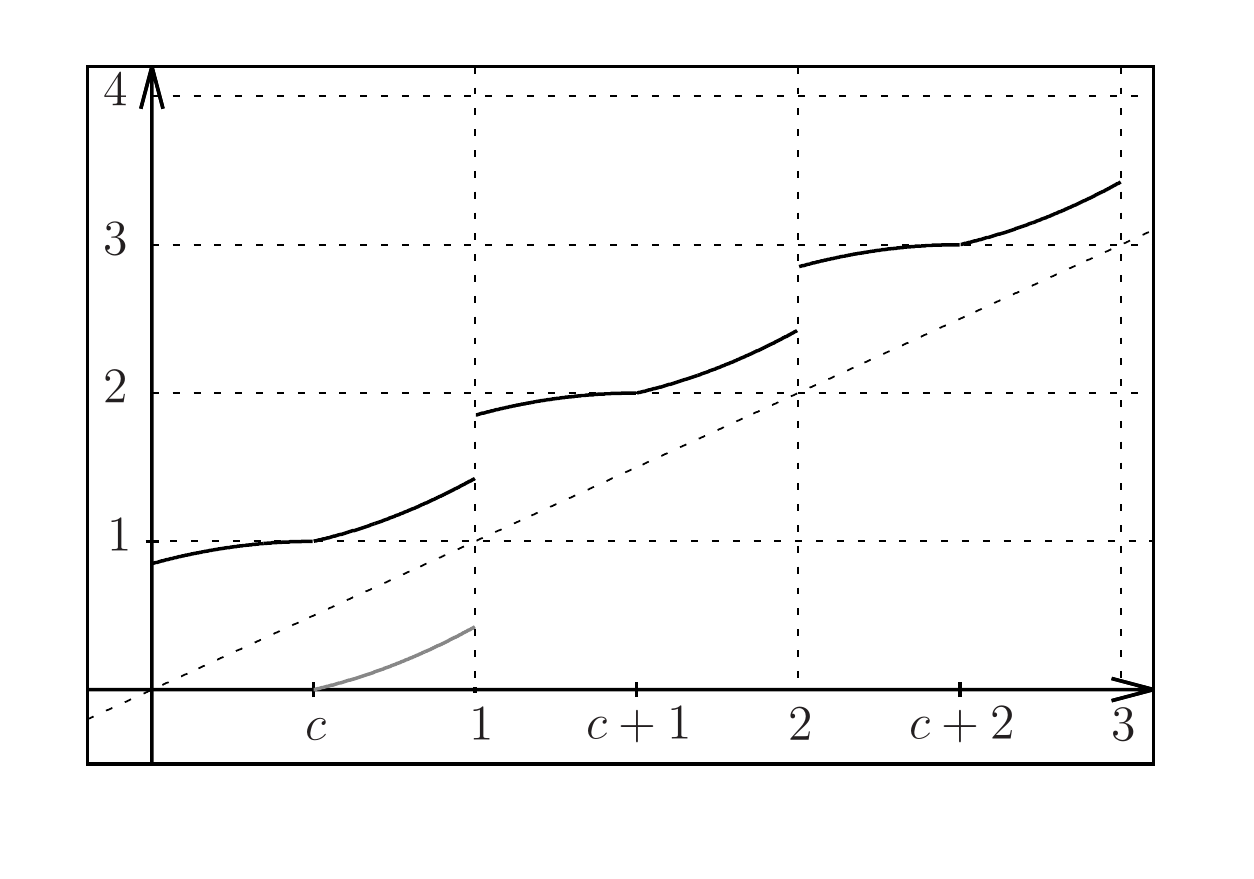}
\end{center}
\caption{Lift of the of the circle map shown in
Figure~\ref{fig:circle_reduction_b}. In gray we show the circle map
for $x\in(c,1)$.}
\label{fig:lift}
\end{figure}

Notice that, at this point, we are not requiring contractiveness, and
all results in this section hold also for expansive maps as long as
conditions~\condsC{} are satisfied.
\begin{definition}\label{def:lift}
Let $f$ be a map satisfying conditions~\condsC{}. We will say that
$F$ is a lift of $f$ of degree $N\ge 0$ if
\begin{equation}
\begin{aligned}
F(x+n)&=f(x)+nN&&\text{if } 0\le x< c\\
F(x+n)&=f(x)+nN+1&&\text{if } c\le x<1\\
\end{aligned}
\label{eq:lift_properties}
\end{equation}
If $N=1$, we will refer to $F$ just as the lift of $f$.
\end{definition}
\begin{remark}
The previous definition of the degree of a circle map
differs from the standard one. When $f$ is a continuous circle map,
the degree of its lift $F$ is the integer number $N$ such that
\begin{equation*}
F(x+1)=F(x)+N,
\end{equation*}
and $F$ is also continuous. In our case, as $f$ is discontinuous at
$x=1$, the lift $F$ cannot be continuous and, hence, its degree can be
somehow chosen.\\
Note that if $f$ is a circle map satisfying~\condsC{} then, if it is
continuous at $x=0\sim 1$ and $F$ is its continuous lift, then
necessarily $N=1$ and $F$ is of degree one in the classical sense.
\end{remark}
From now on, we will restrict to lifts of degree one.
\begin{remark}
Due to condition {\it C.4}, the lift of an orientation preserving map
is an increasing map, possibly discontinuous at integer numbers, where
it undergoes a positive gap.
\end{remark}
\begin{remark}\label{rem:bivalued}
In definition~\ref{def:lift} the value of the lift $F$ at integer
numbers, $n$, is not uniquely defined. When convenient, we will use
the one given by $f(1^-)+n$, $f(0^+)+n$ or both. The latter will lead
to a bi-valued lift.
\end{remark}
The following result is well known and provides the definition of the
rotation number of an orientation preserving circle map
satisfying~\condsC{}. This was introduced by
Poincar\'e~\cite{Poincare1881} for homeomorphisms of the circle of
degree $1$ and later studied and extended to rotation intervals by
many authors (see~\cite{AlsLliMis00} and references therein). For
discontinuous orientation preserving circle maps, this was proven
in~\cite{RhoTho86,Gam87}.  However, if one considers a bi-valued lift
at integer numbers (see Remark~\ref{rem:bivalued}), then the standard
proof holds. 
\begin{proposition}[\cite{Gam87} III.1.1-1,\cite{RhoTho86} Theorem 1]
Let $f$ be a discontinuous orientation preserving map
satisfying~\condsC{}, and let $F$ be its lift. Then, for all
$x\in\RR$, the limit
\begin{equation}
\lim_{x\to\infty}\frac{F^n(x)-x}{n}
\label{eq:rotation_number}
\end{equation}
exists and is independent of $x$.
\label{pro:existence_rotation_number}
\end{proposition}

\begin{proof}
We give the standard proof a slight modification to overcome the
discontinuities at integer numbers.\\
As $F$ is increasing, we get that, for $0\le x\le 1$
\begin{equation*}
\frac{F^n(0^-)}{n}\le\frac{F^n(0^+)}{n}\le
\frac{F^n(x)}{n}\le\frac{F^n(1^-)}{n}\le \frac{F^n(1^+)}{n}.
\end{equation*}
Noting that we take $F$ bi-valued at integer values, we can write this
as
\begin{equation*}
\frac{F^n(0)}{n}\le \frac{F^n(x)}{n}\le \frac{F^n(1)}{n},
\end{equation*}
where $F(0)$ and $F(1)$ can be each of the lateral values.\\
Using that $F(x+1)=F(x)+1$ and applying it recursively to $F^n(1)$ we
have that 
\begin{equation*}
\frac{F^n(0^+)}{n}\le\frac{F^n(x)}{n}\le\frac{F^n(1^+)}{n}=\frac{F^n(0^+)+1}{n},
\end{equation*}
which we can write as
\begin{equation*}
\frac{F^n(0)}{n}\le \frac{F^n(x)}{n}\le \frac{F^n(0)+1}{n},
\end{equation*}
where $F(1)=F(0)+1$ means $F(1^\pm)=F(0^\pm)+1$.\\
Hence, taking limits we get
\begin{equation*}
\lim_{n\to\infty}\frac{F^n(x)-x}{n}=\lim_{n\to\infty}\frac{F^n(x)}{n}=\lim_{n\to\infty}\frac{F^n(0)}{n},
\end{equation*}
and the limit does not depend on $x$. We next show that, indeed, this
limit exists. We apply Proposition~1 of~\cite{RhoTho86}, which states
that, if a sequence $a_n$ satisfies
\begin{equation}
\vert a_{m+n}-a_m-a_n\vert \le A,
\label{eq:prop1_RhoTho}
\end{equation}
for all $n,m\ge1$ and some constant $A$, then there exists some
$\rho$ such that
\begin{equation*}
\vert a_n-n\rho\vert \le A.
\end{equation*}
The sequence $F^n(0)$ satisfies~\eqref{eq:prop1_RhoTho} with
$A=1$. To see this, we use that $F(x+n)=F(x)+n$ to obtain, for
any $x\ge 0$,
\begin{equation*}
F^n(x)=F^n(x\pmod 1)+[x]<F^n(1)+x=F^n(0)+x+1,
\end{equation*}
where $[\cdot]$ denotes the integer part.  Then, taking $x=F^m(0)$, we
get
\begin{equation*}
F^{n+m}(0)<F^n(0)+F^m(0)+1,
\end{equation*}
and~\eqref{eq:prop1_RhoTho} is satisfied with $A=1$. Then, there
exists some $\rho$ such that
\begin{equation*}
\lim_{n\to\infty}\frac{F^n(0)}{n}\le
\lim_{n\to\infty}\frac{1+n\rho}{n}=\rho,
\end{equation*}
and the limit exists.
\end{proof}
Note that this proof differs from the one given in~\cite{AlsLliMis00}.
There, it is first proved that, if $f$ possesses a periodic orbit,
then this limit exists and is rational; then it is shown that it does
not depend on $x$ as above. Subsequently, the definition is extended
to all real numbers using monotonicity and the Dedekind cut
construction. In this approach we show the existence of this limit and
then we discuss the dynamics of the map depending on its value.

The previous result permits one to define the rotation number of a
circle map $f$ fulfilling~\condsC{}. Note that, in general, this
depends on the lift of $F$. However, recall that we have restricted to
lifts of degree one.
\begin{definition}[Rotation number]\label{def:rotation_number}
Given a map $f$ satisfying~\condsC{} and $F$ its lift, we define the
rotation number of $f$ as
\begin{equation*}
\rho(f)=\lim_{n\to\infty}\frac{F^n(x)-x}{n},
\end{equation*}
for any $x\in\RR$.
\end{definition}
\begin{remark}\label{rem:interpretation}
The fact that the limit given in Equation~\eqref{eq:rotation_number} exists
implies that $F^n(x)$ grows linearly with $n$:
\begin{equation*}
F^n(x)-x\sim n\rho(f),\,n>>1.
\end{equation*}
Recalling the properties shown in~\eqref{eq:lift_properties}, the
rotation number $\rho(f)$ can be seen as the average number of
times that the lift $F$ crosses an integer number per iteration.
\end{remark}

The next result is also standard for continuous circle maps
(Proposition 3.7.11 of~\cite{AlsLliMis00}), it provides the existence of a
periodic orbit if the rotation number is rational. Below we prove that
it also holds for discontinuous circle maps.
\begin{proposition}
Let $F$ be the lift (of degree one regarding
Definition~\ref{def:lift}) of a circle map satisfying~\condsC{}. Then,
$\rho(F)=p/q$, $(p,q)=1$, if and only if there exists $x_0$ s.t.
$F^q(x_0)=x_0+p$.
\label{pro:po_rational_rho}
\end{proposition}
Proposition~\ref{pro:po_rational_rho}, whose proof is given below, is
stated assuming that the map $F$ is bi-valued at integer values (see
Remark~\ref{rem:bivalued}). This ensures that one always finds a
periodic orbit if the rotation number is rational. However, if one
considers only one image at integer numbers, given by $f(0^+)$ or
$f(1^-)$, one can lose the existence of a periodic orbit and get a
$\omega$-limit consisting of $q$ points mimicking a periodic orbit.
That is, one recovers a result given in~\cite{RhoTho86}, which we
repeat below for completeness.

\begin{proposition}[\cite{RhoTho86} Th. 2]
Let $F$ be the lift (of degree one regarding
Definition~\ref{def:lift}) of a circle map satisfying~\condsC{} and
$p,q\in\mathbb{Z}$ with $q>0$. If $\rho(F)=p/q$ then exactly one of
the following holds.
\begin{enumerate}[i)]
\item There exists $x_0\in\RR$ such that $F^q(x_0)=x_0+p$.
\item For all $x\in\RR$, $F^q(x)>x+p$, and there exists $x_0\in\RR$ such that 
\begin{equation*}
\lim_{x\to x_0^-}F^q(x)=x_0+p.
\end{equation*}
\item For all $x\in\RR$, $F^q(x)<x+p$, and there exists $x_0\in\RR$ such that 
\begin{equation*}
\lim_{x\to x_0^+}F^q(x)=x_0+p.
\end{equation*}
\end{enumerate}
Conversely, if either i), ii) or iii) holds then $\rho(F)=p/q$.
\label{pro:po_rhodes_thompson}
\end{proposition}

\begin{remark}
The situations ii) or iii) occur when $x_0=c$. That is, when a periodic
orbit bifurcates and the image of $x=0$ by the circle map takes the
value on the ``wrong'' side. More precisely, if $f(0)=f(0^{+})$ then
ii) occurs when a periodic orbits bifurcates when approaching
$x=c$ from the left. Similarly, if $f(0)=f(1^-)$, iii) occurs when
a periodic orbit bifurcates by colliding with $x=c$ from the right.
\end{remark}

For completeness, we will provide a proof of
Proposition~\ref{pro:po_rational_rho} based on~\cite{AlsLliMis00} but
adapting it to the discontinuous case, as in the proof of
Proposition~\ref{pro:po_rational_rho}. It relies on the following two
lemmas.\\
The first lemma is equivalent to Lemma 2 of~\cite{RhoTho86}. We
announce it as in Lemma 3.7.10 of~\cite{AlsLliMis00} but we adapt its
proof to hold also for discontinuous circle maps.
\begin{lemma}
Let $F$ be the lift of a circle map satisfying~\condsC{}, and let
$p\in\mathbb{Z}$. Then,
\begin{enumerate}[i)]
\item If $F^q(x)-p>x$ for all $x\in\RR$, then there exists
$\varepsilon>0$ such that $\rho(F)\ge p/q+\varepsilon$.
\item If $F^q(x)-p<x$ for all $x\in\RR$, then there exists
$\varepsilon>0$ such that $\rho(F)\le p/q-\varepsilon$.
\end{enumerate}
\label{lem:Fq-p_ge_x}
\end{lemma}
Obviously, the proof provided in~\cite{RhoTho86} also holds. However,
by considering that $F$ is bi-valued at integer numbers, we can
proceed as in the proof for the continuous case~\cite{AlsLliMis00},
which we repeat for completeness.\\
\begin{proof}
We show {\em i)}; {\em ii)} can be
proven analogously.  Note that $G(x)=F^q(x)-p-x$
is of degree of zero; i.e.  $G(x+1)=G(x)$. Hence, recalling that $F$
is bi-valued at its discontinuities and hence $F^q(x)-p-x>0$ holds for
each lateral value, it follows that there exists some $\delta>0$ such
that, for every $x$, $F^q(x)-p-x\ge \delta$. Then, for all $k$ we get
\begin{equation*}
F^{qk}(x)-x=\sum_{i=0}^{k-1}\left( F^q\left( F^{qi}(x)
\right)-F^{qi}(x) \right)\ge k(p+\delta).
\end{equation*}
Hence, as by Proposition~\ref{pro:existence_rotation_number} the rotation
number exists and is unique, we have that
\begin{equation*}
\rho(F)=\lim_{k\to\infty}\frac{F^{qk}(x)-x}{qk}\ge
\frac{k(p+\delta)}{kq}=\frac{p}{q}+\frac{\delta}{q},
\end{equation*}
which proves {\em i)} with $\varepsilon=\delta/q$.
\end{proof}
Note that the proof given in~\cite{AlsLliMis00} is slightly different,
as it does not use the uniqueness of the rotation number.\\
The following lemma is trivial for the continuous case.
\begin{lemma}[\cite{RhoTho86} Lemma 3]
Let 
\begin{equation*}
F:\RR\longrightarrow\RR
\end{equation*}
be a (not necessary continuous) non-decreasing map fulfilling
$F(x+1)=x+1$ for all $x\in\RR$.  Assume that $F(x_1)>x_1$ and
$F(x_2)<x_2$ for some $x_1,x_2\in\RR$. Then there exists some $x_0$
such that $F(x_0)=x_0$, and $F$ is continuous on the left at $x_0$.
\label{lem:rhodes_lemma3}
\end{lemma}
We provide  more intuitive proof of this lemma.
\begin{proof}
Assume $F(x)\ne x$ for all $x$ and recall that, if $F$ undergoes a
discontinuity, the jump must be positive.\\
Suppose $F(0)>0$. Then the existence of $x_2$ implies that either
$F$ crosses the diagonal or undergoes a negative jump, either of which
gives a contradictions.\\
If $F(0)<0$, then the existence of $x_1$ implies that $F$ cannot stay
below the diagonal, unless it skips over it by a positive jump. In this
case, as $F(1)=F(0)+1<1$, arguing as before we conclude that
$F$ has to cross the diagonal.
\end{proof}
\begin{remark}
The class of maps considered in Lemma~\ref{lem:rhodes_lemma3} is not
necessary restricted to lifts of circle maps satisfying~\condsC{}.
Note that they may undergo discontinuities with positive jumps between
$0$ and $1$, whereas the lifts of circle maps satisfying~\condsC{} are
continuous in $(0,1)$.
\end{remark}
We now prove Proposition~\ref{pro:po_rational_rho} by, as
before, adapting the proof given in~\cite{AlsLliMis00} to the
discontinuous case.
\begin{proof}[Proof of Proposition~\ref{pro:po_rational_rho}]
Assume that $F^{q}-p$ has a fixed point. Then, we get that
$F^{nq}(x_0)-np=x_0$ for all $n>0$, which implies that
$\rho(F)=p/q$.\\
Assume that $\rho(F)=p/q$ and that $F^{q}(x)-p$ does not have a fixed
point. Then, by Lemma~\ref{lem:rhodes_lemma3}, we get
that either $F^{q}(x)-p<x$ or $F^{q}(x)-p>x$. But then, by
Lemma~\ref{lem:Fq-p_ge_x},  there exists $\varepsilon>0$ such that
either $\rho(F)>p/q+\varepsilon$ or $\rho(F)<p/q-\varepsilon$, which
is a contradiction.
\end{proof}
Note that, besides the fact that we deal with a discontinuous lift,
the previous proof differs from the one given in~\cite{AlsLliMis00} by
the fact that we can use the existence and uniqueness of the rotation
number provided by Proposition~\ref{pro:existence_rotation_number}.
\begin{remark}
Proposition~\ref{pro:po_rational_rho} does not provide the uniqueness
nor stability of periodic orbits. However, if, in addition
to~\condsC{}, one adds contractiveness, i.e. $f'(x)<1$ for
$x\in(0,c)\cup (c,1)$, then one gets that such a periodic orbit is
unique and attracting. Note that $f$ might not be differentiable at
$x=0\sim 1$ and $x=c$.
\label{rem:uiqueness_of_po}
\end{remark}

The next result provides the continuity of the rotation number; i.e.,
if two lifts of orientation preserving maps are ``close'' (using the
uniform norm), so are their rotation numbers. Note that, assuming that
these maps are bi-valued at integer values, one can always choose the
proper images to properly compare them. Hence, its proof for the
discontinuous case becomes the standard one but taking into account
this fact and Lemmas~\ref{lem:Fq-p_ge_x}
and~\ref{lem:rhodes_lemma3}.\\
\begin{proposition}[\cite{AlsLliMis00} Lemma 3.7.12]
The function
\begin{equation*}
F\longmapsto \rho(F)
\end{equation*}
considered in the space of lifts of circle maps satisfying~\condsC{}
is continuous with the norm of uniform convergence.
\label{pro:continuity_rot_num}
\end{proposition}

\begin{proof}
We proceed as in~\cite{AlsLliMis00} by adapting the proof to the fact
that $F$ is bi-valued in order to to overcome the discontinuities at
integer numbers.

Assume $\rho(F)\ne p/q$. Then, the function $G(x)=F^q(x)-p-x$ is away
from zero. By Lemma~\ref{lem:rhodes_lemma3} (applied to $G(x)+x$), we
get that either $G(x)<0$ or $G(x)>0$ for all $x\in\RR$. Then, by
Lemma~\ref{lem:Fq-p_ge_x}, we have that either $\rho(F)<p/q$ or
$\rho(F)>p/q$, respectively.\\
We now note that $G(x)$ has degree $0$; that is,
$G(x+1)=F^q(x+1)-p-(x+1)=F^q(x)-p-x=G(x)$. Hence, we can ensure that,
if $\tilde{G}$ is in  small enough neighbourhood of $G$, then either
$\tilde{G}<0$
or $\tilde{G}>0$, respectively. This implies that, if
$\tilde{F}$ is in a small
enough neighbourhood of $F$, then either $\rho(\tilde{F})<p/q$ or
$\rho(\tilde{F})>p/q$, respectively.\\
Hence, we have shown that, if there exist $p_1/q_1$ and $p_2/q_2$ such
that
\begin{equation*}
\frac{p_1}{q_1}<\rho(F)<\frac{p_1}{q_1},
\end{equation*}
and $\tilde{F}$ is sufficiently close to $F$, then
\begin{equation*}
\frac{p_1}{q_1}<\rho(\tilde{F})<\frac{p_2}{q_2},
\end{equation*}
and hence $F\longmapsto \rho(F)$ is continuous.
\end{proof}
The next Lemma shows that the rotation number is increasing as a
function of $F$.
\begin{lemma}
Let $f$ and $g$ be two circle maps satisfying~\condsC{}, and let $F$
and $G$ be their lifts, respectively. If $F\ge G$ then $\rho(f)\ge
\rho(g)$.
\label{lem:increasing_rotation_number}
\end{lemma}
\begin{proof}
Let $x<c$ and $F$ and $G$ be two lifts of $f$ and $g$ such that $F(x+1)=F(x)+1$
and $G(x+1)=G(x)+1$. As $F$ and $G$ are increasing functions undergoing a
positive gap at $x=k$, $F^n(x)>G^n(x)$, even if, for some $n$, $F^n(x)$ or
$G^n(x)$ reach any of the discontinuities at  $x=k$. Hence,
\begin{equation*}
\frac{F^n(x)-x}{n}\ge \frac{G^n(x)-x}{n}.
\end{equation*}
As the rotation number does not depend on $x$, we get $\rho(f)\ge \rho(g)$.
\end{proof}

We next study properties of the periodic orbits of orientation
preserving circle maps satisfying~\condsC{}. These will be crucial to
show symbolic properties of periodic orbits in
Section~\ref{sec:dyna_orient_preserv}. To this end, we provide the
following definitions.
\begin{definition}[Twist and $p,q$-ordered lifted cycle]\label{def:well-ordered_sequence_of_points}
Let $f$ be a circle map and $x_i\in S^1=\RR/\Z$ such that 
\begin{equation}
0<x_0<x_1<\dots <x_{q-1}<1
\label{eq:seq_qpoints}
\end{equation}
and $f^q(x_i)=x_i$, $0\le i\le q-1$. Consider  the projection
\begin{equation}
\begin{array}{cccc}
\Pi:&\RR&\longrightarrow &\T=\RR/\Z\\
&x&\longmapsto& x\pmod 1
\end{array}
\label{eq:projection_map}
\end{equation}
and the lifted orbit or lifted cycle
\begin{equation*}
S=\left\{ \Pi^{-1}(x_i) \right\},
\end{equation*}
which consists of adding $\Z$ to the orbit~\eqref{eq:seq_qpoints}.
Let $F$ be the lift of $f$. We say that $S$ is a twist lifted cycle
or orbit of $F$ if $F$ restricted to $S$ is increasing.\\
If $S$ is a twist lifted cycle of $f$ and
\begin{equation*}
\lim_{n\to\infty}\frac{F^n(x)-x}{n}=\frac{p}{q}
\end{equation*}
for all $x\in S$, we will say that $S$ is a $p,q$-ordered lifted cycle
of $F$ and $x_0<\dots < x_{q-1}$ is a $p,q$-ordered cycle or periodic
orbit of $f$.
\label{def:twist_lifted_cycle}
\end{definition}
The following result gives us a relation between the spatial and
dynamical ordering of a $p,q$-ordered lifted cycle.
\begin{proposition}\label{pro:pq-ordred}
Let $0<x_0<\dots<x_{q-1}$ be a $p,q$-ordered cycle of the circle map $f$,
let $F$ be its lift and $S$ the corresponding lifted cycle. Assume
that $S$ is given by
\begin{equation*}
S=\left\{ x_0<\dots <x_{q-1}<x_{q}<\dots <x_{2q-1}<\cdots \right\}
\end{equation*}
with $x_j=x_i+n$ if $j=i+nq$, $0\le i\le q-1$. Then, 
\begin{align}
F(x_i)&=x_{i+p}\label{eq:xi_to_xip}\\
F^q(x_i)&=x_i+p.\label{eq:xi_q_to_xip}
\end{align}
\end{proposition}
The proof of this result is as in the continuous case
(\cite{AlsLliMis00} Lemma 3.7.4) but taking into account that the lift
$F$ is bi-valued at integer numbers. We include it here for
completeness.
\begin{proof}
By properly choosing the image of $F$ at integer values, the lift $F$
restricted to the lifted cycle of the periodic orbit is an order
preserving bijection. That is, by iterating the points
$x_0<\dots<x_{q-1}$ one visits all the points of the lifted cycle
exactly once, and the order is preserved ($F(x_j)<F(x_i)$ iff
$x_j<x_i$). Hence, $F(x_i)=x_{i+r}$ for all $i$ and some integer $r$.
If not, then one gets that $F(x_j)=x_{j+r_1}$ and $F(x_i)=x_{i+r_2}$.
If $r_1\neq r_2$, then the number of points of the lifted orbit in
$[x_{j+r_1},x_{i+r_2}]$ and in $[x_{j},x_i]$ does not coincide, and
hence the order cannot be preserved.\\
As the rotation number is $\rho(F)=p/q$, we have that necessarily
$x_{i}+p=F^q(x_i)=x_{i+qr}$. As $x_i$ belongs to a $q$-periodic cycle,
we get $x_{i+qr}=x_i+r$, which is what we wanted to show.
\end{proof}

The next result is also very well known for continuous circle maps.
As in previous results, we provide a standard proof adapted for the
discontinuous case.
\begin{proposition}\label{pro:irrational_rho}
Let $f$ be a circle map satisfying~\condsC{} and assume that
$\rho(f)\in\RR\backslash\Q$.  Then, if $f$ is contracting ($f'(x)<1$,
$x\in (0,c)\cup(c,1)$) and condition {C.4} is a strict inequality, the
$\omega$-limit set of the circle is a Cantor set.
\end{proposition}
\begin{figure}
\begin{center}
\includegraphics[angle=-90,width=0.7\textwidth]{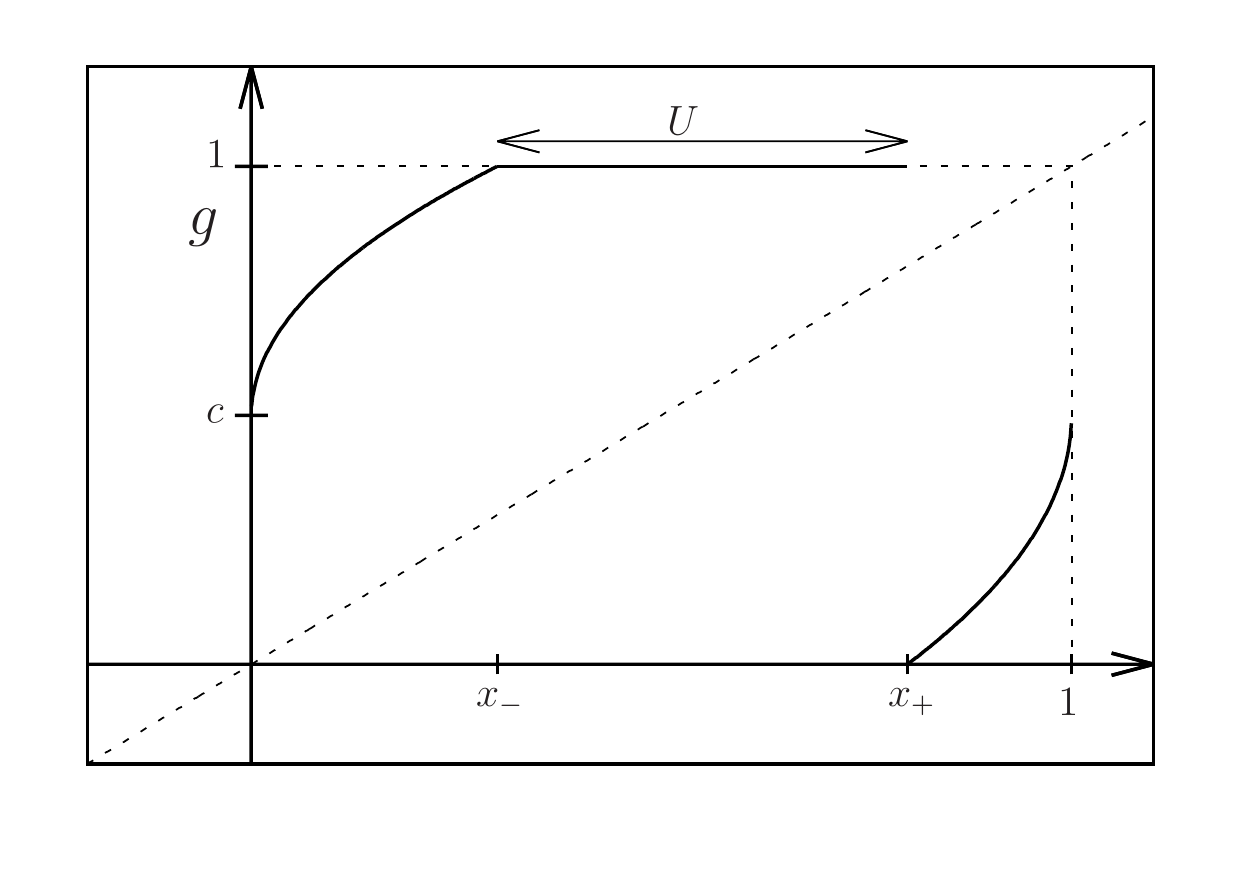}
\end{center}
\caption{Inverse of an orientation preserving circle map
satisfying~\condsC{} with a flat part. See
Equation~\ref{eq:inverse_wflat_part}.}
\label{fig:inverted_circle_map}
\end{figure}
\begin{proof}
Recall that, if $f(0^+)>f(1^-)$, then $f$ is invertible but not
injective. Let $U=(x_-,x_+)$ with $x_-=f(1^-)$ and $x_+=f(0^+)$. We
first show that, if $\rho(f)$ is irrational, then $c\notin f^n(U)$ for
all $n\ge 1$. As
$f$ is invertible in $\T\backslash U$, we consider $f^{-1}$, which is a
function with a ``hole'', as it is not defined in $U$. After filling
this hole with the value $1$,  we obtain the continuous 
function
\begin{equation}
g(y)=\left\{
\begin{aligned}
&f^{-1}(y)&&\text{if }y\in[0,x_-]\\
&1&&\text{if }y\in[x_-,x_+]\\
&f^{-1}(y)&&\text{if }y\in[x_+,1],
\end{aligned}
\right.
\label{eq:inverse_wflat_part}
\end{equation}
(see Figure~\ref{fig:inverted_circle_map}) which has the same rotation
number as $f$ with different sign. Note that
\begin{align*}
g(U)&=1\\
g(1)&=c.
\end{align*}
Hence, if for some $n$, $c\in f^n(U)$, then $g^n(c)\in U$,
$g^{n+1}(c)=1$, $g^{n+2}(c)=c$, and hence $g$ has a periodic orbit of
period $n+2$, which is not compatible with having irrational rotation
number (see Proposition~\ref{pro:po_rational_rho}).\\
Next we show that
\begin{equation*}
\bigcap_{n\ge0}f^n(\T)=\T\backslash \bigcup_{i\ge0}f^i(U)
\end{equation*}
is a Cantor set. As long as $c\notin f^n(U)$, at each iteration,
$f^n(\T)=\T\backslash f^{n-1}(U)$ consists of subtracting a nonempty
interval to the interior of $S^1\backslash f^{n-1}(U)$ (see
Figure~\ref{fig:cantor}). Due to the contraction of $f$, the length of
the subtracted interval tends to $0$.  Moreover, by Corollary~3.3
of~\cite{Vee89}, the total removed amount, $\cup_{i\ge 0}f^i(U)$, is
dense. This comes from the fact that, although the map $g$ is not
differentiable, Denjoy theorem holds and, if the rotation number is
irrational,  $g$ has no ``homtervals''; in particular, $U$ is not a
homterval and the sequence $g^n(U)$ is pairwise disjoint. Therefore,
by construction, $\T\backslash \cup_{i\ge0}f^i(U)$ is a Cantor set.
Moreover, also by construction, the images of $x_-$ and $x_+$ are
dense in this set.  Thus, every point in $\T\backslash
\cup_{i\ge0}f^i(U)$ has a dense orbit and hence this set becomes the
$\omega$-limit of $f$.
\end{proof}
\begin{figure}
\begin{center}
\includegraphics[width=0.5\textwidth]{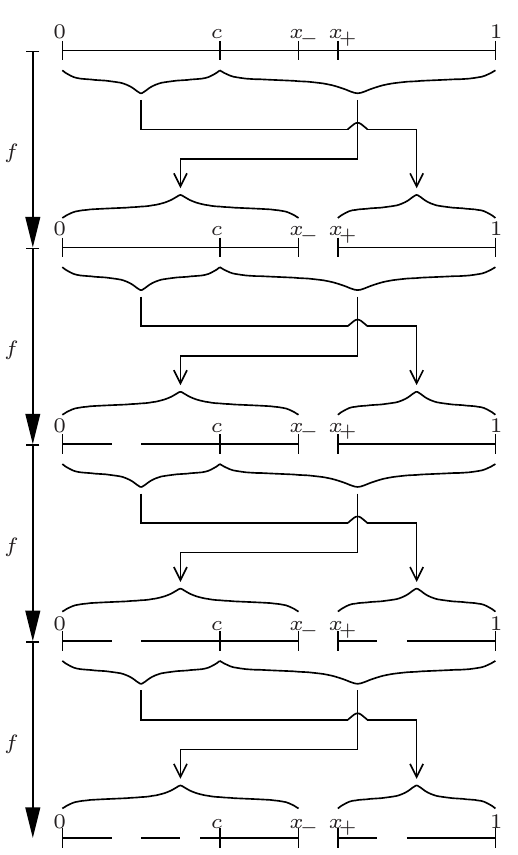}
\end{center}
\caption{Cantor set generated by iterating $U$: new ``holes'' are created
as long as $c\notin f^n(U)$.}
\label{fig:cantor}
\end{figure}

\begin{remark}
If condition {\em C.4} is satisfied by an equality (the map becomes
continuous) and the rotation number is irrational, then the
$\omega$-limit of $f$ may also be a Cantor set or the whole circle. If
$f$ is $C^2$, then Denjoy theorem holds and the latter occurs.
However, if it is $C^1$ or $C^0$, $f$ may become a Denjoy
counterexample (see~\cite{Nit71}), and it's $\omega$-limit may be a
Cantor set.
\end{remark}

To conclude this section, we recover a piecewise-smooth map as defined
in Equation~\eqref{eq:normal_form}. As mentioned above, after applying the
change of variables given in Equation~\eqref{eq:change_variables}, the map
$\tf(x)=\phi\circ f\circ \phi^{-1}(x)$ becomes an
orientation-preserving map satisfying~\condsC{} with
\begin{equation*}
c=\phi(0)=\frac{\mu_\R}{\mu_\R+\mu_\LL}.
\end{equation*}
After applying the reparameterization $\gamma$ given in
Equation~\eqref{eq:1d_scann_curve}, the value
\begin{equation*}
c_\lambda=\frac{\mu_\R(\lambda)}{\mu_\R(\lambda)+\mu_\LL(\lambda)}
\end{equation*}
becomes an strictly decreasing function of $\lambda$ such that
\begin{align*}
\lim_{\lambda\to0^+}c_\lambda&=1\\
\lim_{\lambda\to1^-}c_\lambda&=0.
\end{align*}
Moreover, for $\lambda=0$ and $\lambda=1$, the map $\tf$ possesses
fixed points at $x=1$ and $x=0$, respectively.

Given a piecewise-smooth map $f$ satisfying~\condsHp{}, we will define
its rotation number as the rotation number of the map $\tf=\phi\circ f
\circ\phi^{-1}$ obtained after a reduction to a circle map:
\begin{equation*}
\rho(f)=\rho(\tf).
\end{equation*}
By abusing notation, we will also refer to the lift of $f$ as the lift
of $\tf$.
\begin{remark}\label{rem:rotation-eta_number}
Let $(x_0,\dots,x_{q-1})$ be a periodic orbit of a piecewise-smooth
map whose associated circle map satisfies conditions~\condsC{}, and
let $\x=(\x_0\dots\x_{q-1})\in \left\{ \LL,\R \right\}^q$ be its
associated symbolic sequence regarding the symbolic encoding given
in~\eqref{eq:encoding_LR}:
\begin{equation*}
\x_i=a(f^i(x_i)),\,0\le i\le q-1.
\end{equation*}
Then, recalling Remark~\ref{rem:interpretation} and the fact that the
image of $x$ by the lift of $f$, $F(x)$, crosses an integer number $m$
when $x<c+m<F(x)$, the rotation number of $f$, $\rho(f)$, becomes the
$\eta$ number defined in~\eqref{eq:eta_number}. That is, it becomes
the ratio of the number of symbols $\R$ contained in $\x$ to the length
of the sequence $\x$, $q$.
\end{remark}
\subsection{Symbolic dynamics and families of orientation preserving maps}\label{sec:dyna_orient_preserv}
In this section we will show some dynamical properties of maps
satisfying conditions~\condsC{}, focusing specially on periodic
orbits, their symbolic itineraries and their relation with the
rotation number. In some of the results, we will additionally require
the map to be contractive, however, we emphasize that, when not
specified, the results that we present here do not require
contractiveness.\\
The main result in this section is the following Theorem, which is,
recalling that periodic orbits of orientation preserving circle maps
satisfying~\condsC{} are well ordered, a straightforward consequence
of Proposition~\ref{pro:concatenatetion}.

\begin{theorem}\label{theo:sym_seq_Farey_tree}
The symbolic sequence of the itinerary of a periodic orbit of a circle
map satisfying~\condsC{} with rotation number $P/Q$ is the one in the
Farey tree of symbolic sequences associated with the rational number
$P/Q$.
\end{theorem}

At the end of this section (Lemma~\ref{lem:cantor_set}) we show that,
for a piecewise-smooth map~\eqref{eq:normal_form} satisfying {\em i)}
of Theorem~\ref{theo:adding_incrementing}, the $\eta$-number defined
in Definition~\ref{def:eta-number} follows a devil's staircase. This is a
consequence of Theorem~\ref{theo:Boy85} proved in~\cite{Boy85}.

To show this, we recall that the set $W_{p,q}$ consists of the
$q$-periodic symbolic sequences with $p$ symbols $\R$ (see
Definition~\ref{def:Wpq}). Of special interest will be the {\em well
ordered} symbolic sequences contained in these sets:
\begin{definition}\label{def:well-ordered_symbolic-sequence}
Let $\x\in W_{p,q}$ be a periodic symbolic sequence.  Consider the
(lexicographically) ordered sequence given by the iterates of $\x$ by
$\sigma$
\begin{equation}
\sigma^{i_0}(\x)<\sigma^{i_1}(\x)<\sigma^{i_2}(\x)<\dots<\sigma^{i_{q-1}}(\x).
\label{eq:sequence_of_sequences}
\end{equation}
We say that the sequence $\x$ is a $p,q$-ordered (symbolic) sequence
if
\begin{equation*}
i_j-i_{j-1}=\text{constant}.
\end{equation*}
In other words, $\sigma$ acts on the
sequence~\eqref{eq:sequence_of_sequences} as a cyclic permutation:
there exists some $k\in\N$, $0< k<q$, such that
\begin{equation*}
i_j=i_{j-1}+k\pmod{q}.
\end{equation*}
\end{definition}
The next example illustrates the previous definition.
\begin{example}
The sequence $\x=(\LL^2\R\LL\R)^\infty\in W_{2,5}$ is
$2,5$-ordered, and the sequence $\y=(\LL^3\R^2)^\infty\in
W_{2,5}$ is not. If we consider the four iterates of $\x$ and
$\y$ by $\sigma$ we obtain
\begin{align*}
\sigma(\x)&=(\LL\R\LL\R\LL)^\infty&\sigma(\y)&=(\LL^2\R^2\LL)^\infty\\
\sigma^2(\x)&=(\R\LL\R\LL^2)^\infty&\sigma^2(\y)&=(\LL\R^2\LL^2)^\infty\\
\sigma^3(\x)&=(\LL\R\LL^2\R)^\infty&\sigma^3(\y)&=(\R^2\LL^3)^\infty\\
\sigma^4(\x)&=(\R\LL^2\R\LL)^\infty&\sigma^4(\y)&=(\R\LL^3\R)^\infty.
\end{align*}
Note that $\sigma^5(\x_i)=\x_i$. When we order the iterates by $\sigma$ we
obtain
\begin{align*}
\x&<\sigma^3(\x)<\sigma(\x)<\sigma^4(\x)<\sigma^2(\x)\\
\y&<\sigma(\y)<\sigma^2(\y)<\sigma^4(\y)<\sigma^3(\y),
\end{align*}
and $\x$ is $2,5$-ordered with $k=3$ while $\y$ is not well ordered.
\end{example}

The following result identifies the symbolic sequences of periodic
orbits whose lifted cycles are twist (see
Definition~\ref{def:twist_lifted_cycle}).
\begin{proposition}[\cite{Gam87} Proposition
III.1.1-2]\label{pro:well_ordered_symb-seq} Under the conditions of
Proposition~\ref{pro:pq-ordred}, if the lifted cycle $S$ is
$p,q$-ordered by $F$ (see
Definition~\ref{def:well-ordered_sequence_of_points}) then the itinerary
$I_f(x_i)\in W_{p,q}$ is a $p,q$-ordered symbolic sequence (see
Definition~\ref{def:well-ordered_symbolic-sequence}).
\end{proposition}

\begin{proof}
Let $k$ be such that
\begin{equation}
(k-1)p<q\le kp.
\label{eq:assumption_kp}
\end{equation}
We first note that, for $0\le i,j\le q-1$ we have
\begin{equation*}
I_f(x_i)\le I_f(x_j) \Longleftrightarrow i\le j.
\end{equation*}
We then write
\begin{equation*}
kp=q+r,\,r\ge0.
\end{equation*}
Then, the result comes from the fact that
\begin{equation*}
x_0<x_r\le x_p,
\end{equation*}
which occurs iff $0<r\le p$. Assume that $r>p$. Then $q+r-p>q$ and
hence $(k-1)p>q$, which contradicts~\eqref{eq:assumption_kp}. Letting
$\x=I_f(x_0)$, this implies
\begin{equation*}
\x<\sigma^{k\pod{q}}(\x)<\sigma^{2k\pod{q}}(\x)<\dots<\sigma^{(N-1)k\pod{q}}(\x)=\sigma^{Nk\pod{q}}(\x),
\end{equation*}
where $N$ is the smallest such that $Nk=0\pmod{q}$.
\end{proof}
The next result formalizes what was stated in
Remark~\ref{rem:rotation-eta_number}: the $\eta$-number associated
with the symbolic sequence of a periodic orbit
(Definition~\ref{def:eta-number}) of a piecewise-smooth map satisfying
{\em i)} of Theorem~\ref{theo:adding_incrementing} becomes the
rotation number (Definition~\ref{def:rotation_number}) of the
orientation preserving circle map obtained after the
change~\eqref{eq:change_variables}.
\begin{corollary}\label{cor:eta-number_rotation-number}
Let $f$ be an orientation preserving map, and let $x$ belong to a
$q$-periodic orbit with symbolic sequence $I_f(x)=\x^\infty\in
W_{p,q}$. Then, the rotation number becomes
\begin{equation*}
\rho(f)=\frac{p}{q}=\eta(\x).
\end{equation*}
That is, it is given by the ratio between the number of $\R$ symbols
contained in $\x$ and the period, $q$, of the sequence.
\end{corollary}

Our next step consists of showing that the symbolic sequence
associated with a periodic orbit of an orientation preserving circle map
belongs to the Farey tree of symbolic sequences shown in
Figure~\ref{fig:farey_sequences}. More precisely, we show that such a
symbolic sequence is obtained by the concatenation of the symbolic
sequences associated with the periodic orbits of the Farey parents of
its rotation number. As a consequence of that, one obtains
Theorem~\ref{theo:sym_seq_Farey_tree}, announced above.\\
Note that this result provides an alternative isomorphism between the
Farey tree of rational numbers and the Farey tree of symbolics
sequences (Figures~\ref{fig:farey_rho} and~\ref{fig:farey_sequences},
respectively) by means of the dynamical properties of circle maps.

In Section~\ref{sec:maximin_1d} we will present an alternative
approach using the {\em maximin} properties of these sequences.

\begin{proposition}\label{pro:concatenatetion}
Let $f$ be a circle map satisfying~\condsC{}, and assume it has a
periodic orbit with rotation number $P/Q$, $(P,Q)=1$. Let
$\Delta^\infty\in W_{P,Q}$ be its symbolic sequence, and  assume that
$\Delta$ is minimal:
\begin{equation*}
\Delta=\min_{k\ge0}\sigma^k(\Delta).
\end{equation*}
Let $p$, $q$, $p'$ and $q'$ natural numbers such that
\begin{itemize}
\item $P/Q=(p+p')/(q+q')$
\item $(p,q)=(p',q')=1$
\item $p'q-pq'=1$,
\end{itemize}
that is, $p/q<p'/q'$ are the Farey parents of $P/Q$.\\
Let $\alpha^\infty\in W_{p,q}$ and $\beta^\infty\in W_{p',q'}$ be the symbolic
sequences of the periodic orbits with rotation numbers
$p/q$ and $p'/q'$, respectively. Assume that $\alpha$ and $\beta$ are
minimal. Then $\Delta$ is the concatenation of $\alpha$ and $\beta$:
\begin{equation*}
\Delta=\alpha\beta.
\end{equation*}
\end{proposition}

\begin{proof}
Let
\begin{equation*}
0<z_0<z_1<\dots<z_{Q-1}<1
\end{equation*}
be a periodic orbit with rotation number $P/Q$. Let us consider the
lifted cycle given by $\Pi^{-1}(z_i)$ given in
Equation~\eqref{eq:projection_map},
\begin{equation*}
\begin{aligned}
0<z_0<z_1<\dots<&z_{Q-1}<1<z_Q<z_{Q+1}<\dots <z_{2Q-1}<2<z_{2Q}<\cdots\\
&\cdots<z_{PQ-1}<P<z_{PQ}<\cdots
\end{aligned}
\end{equation*}
Consequently the following identity holds:
\begin{equation}\label{eq:MK0}
z_{j+mQ}=z_j+m,\; \mbox{ where $j,\; m\in\Z$}.
\end{equation}
Let $F$ be the lift of $f$ as in Definition~\ref{def:lift}. Due to
Proposition~\ref{pro:pq-ordred}, the points $z_i$ are $P,Q$-ordered by
$F$:
\begin{align*}
F(z_i)&=z_{i+P}\\
F^Q(z_i)&=z_i+P.
\end{align*}
We now construct
two subsequences, $(x_i)$ and $(y_i)$, as follows (see
Example~\ref{ex:sebsequences}).

The former will consist of the lifting of the set $\left\{
z_0,f(z_0),\dots,f^{q-1}(z_0) \right\}$ (the first $q$ iterates of
$z_0$), given by
\begin{equation*}
\{z_{nP+iQ}\, | \,0\le n\le q-1,\,i\in\Z\}.
\end{equation*}
Since $p$ and $q$ are relatively prime, every integer $j$ is uniquely represented as $j=np+iq$.
Let $(x_j)_{j\in\Z}$, be the sequence defined by
\begin{equation}
x_{np+iq}=z_{nP+iQ},\,0\le n\le q-1,\,i\in\Z.
\label{eq:sequence_x}
\end{equation}
We first prove that $(x_j)_{j\in\Z}$ satisfies
\begin{align}
0<x_0<x_1<\dots<&x_{q-1}<1<x_q<x_{q+1}<\dots
<x_{2q-1}<2<x_{2q}<\cdots\nonumber\\
&\dots<x_{pq-1}<p<x_{pq}<\cdots\label{eq:ordered_sequence_x}
\end{align}
To see that, we show that the sequences $(z_{kP})_{k\in\Z}$ and
$(x_{kp})_{k\in\Z}$ skip integer numbers at the same iterates; that
is, if
\begin{equation*}
z_{(k-1)P}<i<z_{kP}\;\Big(\Longleftrightarrow (k-1)P<iQ<kP\Big),
\end{equation*}
then
\begin{equation*}
x_{(k-1)p}<i<x_{kp}\;\Big(\Longleftrightarrow (k-1)p<iq<kp\Big).
\end{equation*}
This comes from the fact that, recalling {\em ii)} of
Theorem~\ref{the:Farey_properties}, $p/q$ and $P/Q$ are Farey
neighbours, and hence they are in the path of the Farey neighbours
$i/k$ and $i/(k-1)$. Therefore we have
\begin{equation*}
\frac{i}{k}<\frac{p}{q}<\frac{P}{Q}<\frac{p'}{q'}<\frac{i}{k-1},
\end{equation*}
because $P/Q$ is the Farey median of $p/q$ and $p'/q'$.\\
Let now $l=kP\pmod{Q}$ and let $t$ be such that
\begin{equation*}
z_{(t-1)P}\le z_l<z_{tP}.
\end{equation*}
Then we must have
\begin{equation*}
x_{(t-1)p}\le x_{\tilde{l}}<x_{tp},
\end{equation*}
where $\tilde{l}=kp\pmod{q}$. Otherwise, the sequences $z_{iP}$ and
$x_{ip}$ could not skip integer numbers at the same time. This shows
that necessary the sequence $(x_j)_{j\in\Z}$ must
satisfy~\eqref{eq:ordered_sequence_x}. Further, it follows
from~\eqref{eq:MK0} and~\eqref{eq:sequence_x} that
\begin{equation}\label{eq-MK2}
x_{j+qp}=x_j+p,\;j\in (0,\ldots, q-1\}.
\end{equation}
Moreover, for $j=np+iq$, $n\neq q-1$, we have 
\begin{equation}
F(x_j)=F(z_{nP+iQ})=z_{nP+iQ+P}=x_{(n+1)p+iq}=x_{j+p}.
\label{eq:MKO3}
\end{equation}
Note that, for $n=q-1$, Equation~\eqref{eq:MKO3} does not hold, as
$F(z_{(q-1)P+iQ})=z_{qP+iQ}$ does not belong to the sequence
$(x_i)_{i\in\Z}$.\\
Hence, we have that
\begin{equation}\label{eq-MK3}
F(x_j)=x_{j+p},\;j=np+iq,\,n\neq q-1 \pmod{q}.
\end{equation}
As the lifted sequence $(x_j)_{j\in\Z}$ satisfies~\eqref{eq-MK2}
and~\eqref{eq-MK3}, its symbolic sequence, $\alpha\in W_{p,q}$, is the
one associated with a $p,q$-ordered lifted cycle.

We now obtain another subsequence derived from the last $q'$ points of
the first $Q$ iterates of the point $z_0$. The first point of these
$q'$ points is
\begin{equation*}
F^q(z_0)=z_{qP}.
\end{equation*}
Note that, as $p/q<p'/q'$ are Farey neighbours, we get
\begin{equation*}
1=p'q-pq'=p'q+pq-pq-pq'=(p'+p)q-p(q+q')=qP-pQ.
\end{equation*}
Therefore,
\begin{equation*}
qP=1+pQ=1\pmod Q.
\end{equation*}
This implies that, when projecting the point $z_{qP}$ at the circle
$\RR/\Z$ (identified with $[0,1]$), we obtain $\Pi(z_{qP})=z_1$.
Therefore, when lifting the projection of the last $q'$ points of the
first $Q$ iterates of $z_0$, we obtain the sequence given by
\begin{equation}
y_{np'+iq'}=z_{1+nP+iQ},\,0\le n\le q'-1,\,i\in\Z.
\label{eq:sequence_y}
\end{equation}
Arguing as before, this sequence has the form
\begin{align*}
0<y_0<y_1<\dots<&y_{q'-1}<1<y_{q'}<y_{q'+1}<\dots
<y_{2q'-1}<2<y_{2q'}<\cdots\\
&\cdots<y_{p'q'-1}<p'<y_{p'q'}<\cdots,
\end{align*}
and sastisfies
\begin{align*}
F^i(y_0)&=y_{ip'},\,q\le i\le q+q'-1\\
y_{i+q'p'}&=y_i+p'.
\end{align*}
Hence, its symbolic sequence, $\beta\in W_{p',q'}$ is the one
associated with $p',q'$-ordered lifted cycle.\\
Note that, $\alpha$ and $\beta$ are minimal, as they are associated
with the iterates of the lowest positive point of the sequences
$(x_i)$ and $(y_i)$, $x_0=z_0$ and $y_0=z_1$, respectively.

By construction, the symbolic sequence $\Delta\in W_{P,Q}$ is the
concatenation of $\alpha$ and $\beta$, as we wanted to show.
\end{proof}
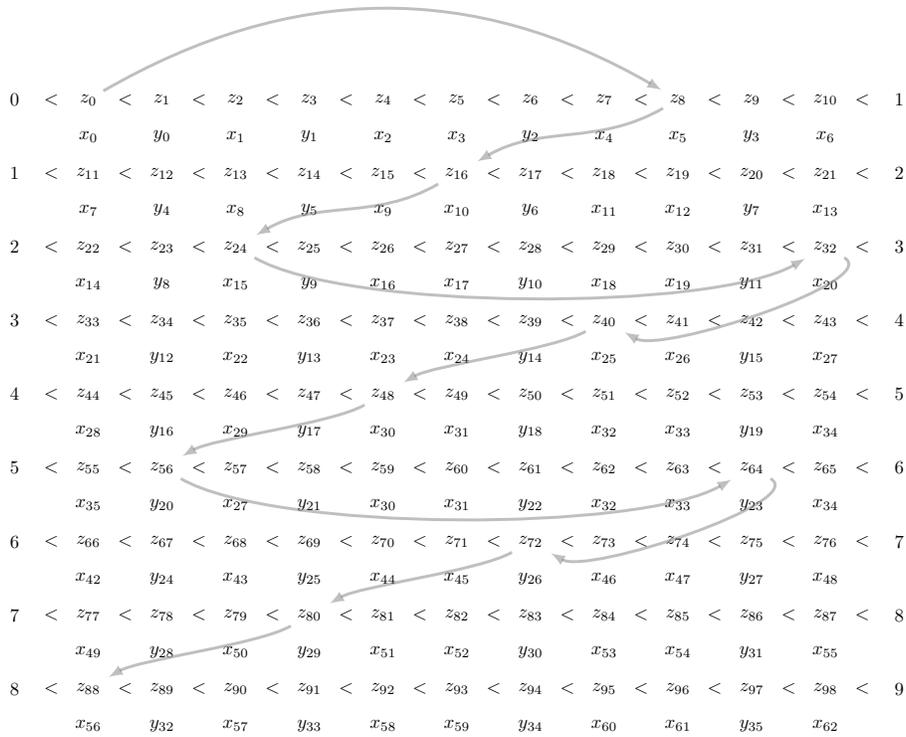
\begin{figure}\centering
\begin{picture}(1,0.7)
\put(-0.01,0.06){
\scalebox{0.7}{
\begin{tikzpicture}[->,>=stealth',shorten >=1pt,auto,node
distance=0.7cm]
\node[] (a0){$0$};
\node[] (a0b)[right of=a0]{$<$};
\node[] (a1)[right of=a0b]{$z_0$};
\node[] (a1b)[right of=a1]{$<$};
\node[] (a2) [right of=a1b] {$z_1$};
\node[] (a2b)[right of=a2]{$<$};
\node[] (a3) [right of=a2b] {$z_2$};
\node[] (a3b)[right of=a3]{$<$};
\node[] (a4) [right of=a3b] {$z_3$};
\node[] (a4b)[right of=a4]{$<$};
\node[] (a5) [right of=a4b] {$z_4$};
\node[] (a5b)[right of=a5]{$<$};
\node[] (a6) [right of=a5b] {$z_5$};
\node[] (a6b)[right of=a6]{$<$};
\node[] (a7) [right of=a6b] {$z_6$};
\node[] (a7b)[right of=a7]{$<$};
\node[] (a8) [right of=a7b] {$z_7$};
\node[] (a8b)[right of=a8]{$<$};
\node[] (a9) [right of=a8b] {$z_8$};
\node[] (a9b)[right of=a9]{$<$};
\node[] (a10) [right of=a9b] {$z_9$};
\node[] (a10b)[right of=a10]{$<$};
\node[] (a11) [right of=a10b] {$z_{10}$};
\node[] (a11b)[right of=a11]{$<$};
\node[] (a12) [right of=a11b] {$1$};
\node[] (ax0)[below of=a0]{};
\node[] (ax0b)[right of=ax0]{};
\node[] (ax1)[right of=ax0b]{$x_0$};
\node[] (ax1b)[right of=ax1]{};
\node[] (ax2) [right of=ax1b] {$y_0$};
\node[] (ax2b)[right of=ax2]{};
\node[] (ax3) [right of=ax2b] {$x_1$};
\node[] (ax3b)[right of=ax3]{};
\node[] (ax4) [right of=ax3b] {$y_1$};
\node[] (ax4b)[right of=ax4]{};
\node[] (ax5) [right of=ax4b] {$x_2$};
\node[] (ax5b)[right of=ax5]{};
\node[] (ax6) [right of=ax5b] {$x_3$};
\node[] (ax6b)[right of=ax6]{};
\node[] (ax7) [right of=ax6b] {$y_2$};
\node[] (ax7b)[right of=ax7]{};
\node[] (ax8) [right of=ax7b] {$x_4$};
\node[] (ax8b)[right of=ax8]{};
\node[] (ax9) [right of=ax8b] {$x_5$};
\node[] (ax9b)[right of=ax9]{};
\node[] (ax10) [right of=ax9b] {$y_{3}$};
\node[] (ax10b)[right of=ax10]{};
\node[] (ax11) [right of=ax10b] {$x_6$};
\node[] (ax11b)[right of=ax11]{};
\node[] (ax12) [right of=ax11b] {};
\node[] (b0)[below of=ax0]{$1$};
\node[] (b0b)[right of=b0]{$<$};
\node[] (b1)[right of=b0b]{$z_{11}$};
\node[] (b1b)[right of=b1]{$<$};
\node[] (b2) [right of=b1b] {$z_{12}$};
\node[] (b2b)[right of=b2]{$<$};
\node[] (b3) [right of=b2b] {$z_{13}$};
\node[] (b3b)[right of=b3]{$<$};
\node[] (b4) [right of=b3b] {$z_{14}$};
\node[] (b4b)[right of=b4]{$<$};
\node[] (b5) [right of=b4b] {$z_{15}$};
\node[] (b5b)[right of=b5]{$<$};
\node[] (b6) [right of=b5b] {$z_{16}$};
\node[] (b6b)[right of=b6]{$<$};
\node[] (b7) [right of=b6b] {$z_{17}$};
\node[] (b7b)[right of=b7]{$<$};
\node[] (b8) [right of=b7b] {$z_{18}$};
\node[] (b8b)[right of=b8]{$<$};
\node[] (b9) [right of=b8b] {$z_{19}$};
\node[] (b9b)[right of=b9]{$<$};
\node[] (b10) [right of=b9b] {$z_{20}$};
\node[] (b10b)[right of=b10]{$<$};
\node[] (b11) [right of=b10b] {$z_{21}$};
\node[] (b11b)[right of=b11]{$<$};
\node[] (b12) [right of=b11b] {$2$};
\node[] (bx0)[below of=b0]{};
\node[] (bx0b)[right of=bx0]{};
\node[] (bx1)[right of=bx0b]{$x_{7}$};
\node[] (bx1b)[right of=bx1]{};
\node[] (bx2) [right of=bx1b] {$y_{4}$};
\node[] (bx2b)[right of=bx2]{};
\node[] (bx3) [right of=bx2b] {$x_{8}$};
\node[] (bx3b)[right of=bx3]{};
\node[] (bx4) [right of=bx3b] {$y_5$};
\node[] (bx4b)[right of=bx4]{};
\node[] (bx5) [right of=bx4b] {$x_9$};
\node[] (bx5b)[right of=bx5]{};
\node[] (bx6) [right of=bx5b] {$x_{10}$};
\node[] (bx6b)[right of=bx6]{};
\node[] (bx7) [right of=bx6b] {$y_{6}$};
\node[] (bx7b)[right of=bx7]{};
\node[] (bx8) [right of=bx7b] {$x_{11}$};
\node[] (bx8b)[right of=bx8]{};
\node[] (bx9) [right of=bx8b] {$x_{12}$};
\node[] (bx9b)[right of=bx9]{};
\node[] (bx10) [right of=bx9b] {$y_{7}$};
\node[] (bx10b)[right of=bx10]{};
\node[] (bx11) [right of=bx10b] {$x_{13}$};
\node[] (bx11b)[right of=bx11]{};
\node[] (bx12) [right of=bx11b] {};
\node[] (c0)[below of=bx0]{$2$};
\node[] (c0b)[right of=c0]{$<$};
\node[] (c1)[right of=c0b]{$z_{22}$};
\node[] (c1b)[right of=c1]{$<$};
\node[] (c2) [right of=c1b] {$z_{23}$};
\node[] (c2b)[right of=c2]{$<$};
\node[] (c3) [right of=c2b] {$z_{24}$};
\node[] (c3b)[right of=c3]{$<$};
\node[] (c4) [right of=c3b] {$z_{25}$};
\node[] (c4b)[right of=c4]{$<$};
\node[] (c5) [right of=c4b] {$z_{26}$};
\node[] (c5b)[right of=c5]{$<$};
\node[] (c6) [right of=c5b] {$z_{27}$};
\node[] (c6b)[right of=c6]{$<$};
\node[] (c7) [right of=c6b] {$z_{28}$};
\node[] (c7b)[right of=c7]{$<$};
\node[] (c8) [right of=c7b] {$z_{29}$};
\node[] (c8b)[right of=c8]{$<$};
\node[] (c9) [right of=c8b] {$z_{30}$};
\node[] (c9b)[right of=c9]{$<$};
\node[] (c10) [right of=c9b] {$z_{31}$};
\node[] (c10b)[right of=c10]{$<$};
\node[] (c11) [right of=c10b] {$z_{32}$};
\node[] (c11b)[right of=c11]{$<$};
\node[] (c12) [right of=c11b] {$3$};
\node[] (cx0)[below of=c0]{};
\node[] (cx0b)[right of=cx0]{};
\node[] (cx1)[right of=cx0b]{$x_{14}$};
\node[] (cx1b)[right of=cx1]{};
\node[] (cx2) [right of=cx1b] {$y_8$};
\node[] (cx2b)[right of=cx2]{};
\node[] (cx3) [right of=cx2b] {$x_{15}$};
\node[] (cx3b)[right of=cx3]{};
\node[] (cx4) [right of=cx3b] {$y_9$};
\node[] (cx4b)[right of=cx4]{};
\node[] (cx5) [right of=cx4b] {$x_{16}$};
\node[] (cx5b)[right of=cx5]{};
\node[] (cx6) [right of=cx5b] {$x_{17}$};
\node[] (cx6b)[right of=cx6]{};
\node[] (cx7) [right of=cx6b] {$y_{10}$};
\node[] (cx7b)[right of=cx7]{};
\node[] (cx8) [right of=cx7b] {$x_{18}$};
\node[] (cx8b)[right of=cx8]{};
\node[] (cx9) [right of=cx8b] {$x_{19}$};
\node[] (cx9b)[right of=cx9]{};
\node[] (cx10) [right of=cx9b] {$y_{11}$};
\node[] (cx10b)[right of=cx10]{};
\node[] (cx11) [right of=cx10b] {$x_{20}$};
\node[] (cx11b)[right of=cx11]{};
\node[] (cx12) [right of=cx11b] {};
\node[] (d0)[below of=cx0]{$3$};
\node[] (d0b)[right of=d0]{$<$};
\node[] (d1)[right of=d0b]{$z_{33}$};
\node[] (d1b)[right of=d1]{$<$};
\node[] (d2) [right of=d1b] {$z_{34}$};
\node[] (d2b)[right of=d2]{$<$};
\node[] (d3) [right of=d2b] {$z_{35}$};
\node[] (d3b)[right of=d3]{$<$};
\node[] (d4) [right of=d3b] {$z_{36}$};
\node[] (d4b)[right of=d4]{$<$};
\node[] (d5) [right of=d4b] {$z_{37}$};
\node[] (d5b)[right of=d5]{$<$};
\node[] (d6) [right of=d5b] {$z_{38}$};
\node[] (d6b)[right of=d6]{$<$};
\node[] (d7) [right of=d6b] {$z_{39}$};
\node[] (d7b)[right of=d7]{$<$};
\node[] (d8) [right of=d7b] {$z_{40}$};
\node[] (d8b)[right of=d8]{$<$};
\node[] (d9) [right of=d8b] {$z_{41}$};
\node[] (d9b)[right of=d9]{$<$};
\node[] (d10) [right of=d9b] {$z_{42}$};
\node[] (d10b)[right of=d10]{$<$};
\node[] (d11) [right of=d10b] {$z_{43}$};
\node[] (d11b)[right of=d11]{$<$};
\node[] (d12) [right of=d11b] {$4$};
\node[] (dx0)[below of=d0]{};
\node[] (dx0b)[right of=dx0]{};
\node[] (dx1)[right of=dx0b]{$x_{21}$};
\node[] (dx1b)[right of=dx1]{};
\node[] (dx2) [right of=dx1b] {$y_{12}$};
\node[] (dx2b)[right of=dx2]{};
\node[] (dx3) [right of=dx2b] {$x_{22}$};
\node[] (dx3b)[right of=dx3]{};
\node[] (dx4) [right of=dx3b] {$y_{13}$};
\node[] (dx4b)[right of=dx4]{};
\node[] (dx5) [right of=dx4b] {$x_{23}$};
\node[] (dx5b)[right of=dx5]{};
\node[] (dx6) [right of=dx5b] {$x_{24}$};
\node[] (dx6b)[right of=dx6]{};
\node[] (dx7) [right of=dx6b] {$y_{14}$};
\node[] (dx7b)[right of=dx7]{};
\node[] (dx8) [right of=dx7b] {$x_{25}$};
\node[] (dx8b)[right of=dx8]{};
\node[] (dx9) [right of=dx8b] {$x_{26}$};
\node[] (dx9b)[right of=dx9]{};
\node[] (dx10) [right of=dx9b] {$y_{15}$};
\node[] (dx10b)[right of=dx10]{};
\node[] (dx11) [right of=dx10b] {$x_{27}$};
\node[] (dx11b)[right of=dx11]{};
\node[] (dx12) [right of=dx11b] {};
\node[] (e0)[below of=dx0]{$4$};
\node[] (e0b)[right of=e0]{$<$};
\node[] (e1)[right of=e0b]{$z_{44}$};
\node[] (e1b)[right of=e1]{$<$};
\node[] (e2) [right of=e1b] {$z_{45}$};
\node[] (e2b)[right of=e2]{$<$};
\node[] (e3) [right of=e2b] {$z_{46}$};
\node[] (e3b)[right of=e3]{$<$};
\node[] (e4) [right of=e3b] {$z_{47}$};
\node[] (e4b)[right of=e4]{$<$};
\node[] (e5) [right of=e4b] {$z_{48}$};
\node[] (e5b)[right of=e5]{$<$};
\node[] (e6) [right of=e5b] {$z_{49}$};
\node[] (e6b)[right of=e6]{$<$};
\node[] (e7) [right of=e6b] {$z_{50}$};
\node[] (e7b)[right of=e7]{$<$};
\node[] (e8) [right of=e7b] {$z_{51}$};
\node[] (e8b)[right of=e8]{$<$};
\node[] (e9) [right of=e8b] {$z_{52}$};
\node[] (e9b)[right of=e9]{$<$};
\node[] (e10) [right of=e9b] {$z_{53}$};
\node[] (e10b)[right of=e10]{$<$};
\node[] (e11) [right of=e10b] {$z_{54}$};
\node[] (e11b)[right of=e11]{$<$};
\node[] (e12) [right of=e11b] {$5$};
\node[] (ex0)[below of=e0]{};
\node[] (ex0b)[right of=ex0]{};
\node[] (ex1)[right of=ex0b]{$x_{28}$};
\node[] (ex1b)[right of=ex1]{};
\node[] (ex2) [right of=ex1b] {$y_{16}$};
\node[] (ex2b)[right of=ex2]{};
\node[] (ex3) [right of=ex2b] {$x_{29}$};
\node[] (ex3b)[right of=ex3]{};
\node[] (ex4) [right of=ex3b] {$y_{17}$};
\node[] (ex4b)[right of=ex4]{};
\node[] (ex5) [right of=ex4b] {$x_{30}$};
\node[] (ex5b)[right of=ex5]{};
\node[] (ex6) [right of=ex5b] {$x_{31}$};
\node[] (ex6b)[right of=ex6]{};
\node[] (ex7) [right of=ex6b] {$y_{18}$};
\node[] (ex7b)[right of=ex7]{};
\node[] (ex8) [right of=ex7b] {$x_{32}$};
\node[] (ex8b)[right of=ex8]{};
\node[] (ex9) [right of=ex8b] {$x_{33}$};
\node[] (ex9b)[right of=ex9]{};
\node[] (ex10) [right of=ex9b] {$y_{19}$};
\node[] (ex10b)[right of=ex10]{};
\node[] (ex11) [right of=ex10b] {$x_{34}$};
\node[] (ex11b)[right of=ex11]{};
\node[] (ex12) [right of=ex11b] {};
\node[] (f0)[below of=ex0]{$5$};
\node[] (f0b)[right of=f0]{$<$};
\node[] (f1)[right of=f0b]{$z_{55}$};
\node[] (f1b)[right of=f1]{$<$};
\node[] (f2) [right of=f1b] {$z_{56}$};
\node[] (f2b)[right of=f2]{$<$};
\node[] (f3) [right of=f2b] {$z_{57}$};
\node[] (f3b)[right of=f3]{$<$};
\node[] (f4) [right of=f3b] {$z_{58}$};
\node[] (f4b)[right of=f4]{$<$};
\node[] (f5) [right of=f4b] {$z_{59}$};
\node[] (f5b)[right of=f5]{$<$};
\node[] (f6) [right of=f5b] {$z_{60}$};
\node[] (f6b)[right of=f6]{$<$};
\node[] (f7) [right of=f6b] {$z_{61}$};
\node[] (f7b)[right of=f7]{$<$};
\node[] (f8) [right of=f7b] {$z_{62}$};
\node[] (f8b)[right of=f8]{$<$};
\node[] (f9) [right of=f8b] {$z_{63}$};
\node[] (f9b)[right of=f9]{$<$};
\node[] (f10) [right of=f9b] {$z_{64}$};
\node[] (f10b)[right of=f10]{$<$};
\node[] (f11) [right of=f10b] {$z_{65}$};
\node[] (f11b)[right of=f11]{$<$};
\node[] (f12) [right of=f11b] {$6$};
\node[] (fx0)[below of=f0]{};
\node[] (fx0b)[right of=fx0]{};
\node[] (fx1)[right of=fx0b]{$x_{35}$};
\node[] (fx1b)[right of=fx1]{};
\node[] (fx2) [right of=fx1b] {$y_{20}$};
\node[] (fx2b)[right of=fx2]{};
\node[] (fx3) [right of=fx2b] {$x_{27}$};
\node[] (fx3b)[right of=fx3]{};
\node[] (fx4) [right of=fx3b] {$y_{21}$};
\node[] (fx4b)[right of=fx4]{};
\node[] (fx5) [right of=fx4b] {$x_{30}$};
\node[] (fx5b)[right of=fx5]{};
\node[] (fx6) [right of=fx5b] {$x_{31}$};
\node[] (fx6b)[right of=fx6]{};
\node[] (fx7) [right of=fx6b] {$y_{22}$};
\node[] (fx7b)[right of=fx7]{};
\node[] (fx8) [right of=fx7b] {$x_{32}$};
\node[] (fx8b)[right of=fx8]{};
\node[] (fx9) [right of=fx8b] {$x_{33}$};
\node[] (fx9b)[right of=fx9]{};
\node[] (fx10) [right of=fx9b] {$y_{23}$};
\node[] (fx10b)[right of=fx10]{};
\node[] (fx11) [right of=fx10b] {$x_{34}$};
\node[] (fx11b)[right of=fx11]{};
\node[] (fx12) [right of=fx11b] {};
\node[] (g0)[below of=fx0]{$6$};
\node[] (g0b)[right of=g0]{$<$};
\node[] (g1)[right of=g0b]{$z_{66}$};
\node[] (g1b)[right of=g1]{$<$};
\node[] (g2) [right of=g1b] {$z_{67}$};
\node[] (g2b)[right of=g2]{$<$};
\node[] (g3) [right of=g2b] {$z_{68}$};
\node[] (g3b)[right of=g3]{$<$};
\node[] (g4) [right of=g3b] {$z_{69}$};
\node[] (g4b)[right of=g4]{$<$};
\node[] (g5) [right of=g4b] {$z_{70}$};
\node[] (g5b)[right of=g5]{$<$};
\node[] (g6) [right of=g5b] {$z_{71}$};
\node[] (g6b)[right of=g6]{$<$};
\node[] (g7) [right of=g6b] {$z_{72}$};
\node[] (g7b)[right of=g7]{$<$};
\node[] (g8) [right of=g7b] {$z_{73}$};
\node[] (g8b)[right of=g8]{$<$};
\node[] (g9) [right of=g8b] {$z_{74}$};
\node[] (g9b)[right of=g9]{$<$};
\node[] (g10) [right of=g9b] {$z_{75}$};
\node[] (g10b)[right of=g10]{$<$};
\node[] (g11) [right of=g10b] {$z_{76}$};
\node[] (g11b)[right of=g11]{$<$};
\node[] (g12) [right of=g11b] {$7$};
\node[] (gx0)[below of=g0]{};
\node[] (gx0b)[right of=gx0]{};
\node[] (gx1)[right of=gx0b]{$x_{42}$};
\node[] (gx1b)[right of=gx1]{};
\node[] (gx2) [right of=gx1b] {$y_{24}$};
\node[] (gx2b)[right of=gx2]{};
\node[] (gx3) [right of=gx2b] {$x_{43}$};
\node[] (gx3b)[right of=gx3]{};
\node[] (gx4) [right of=gx3b] {$y_{25}$};
\node[] (gx4b)[right of=gx4]{};
\node[] (gx5) [right of=gx4b] {$x_{44}$};
\node[] (gx5b)[right of=gx5]{};
\node[] (gx6) [right of=gx5b] {$x_{45}$};
\node[] (gx6b)[right of=gx6]{};
\node[] (gx7) [right of=gx6b] {$y_{26}$};
\node[] (gx7b)[right of=gx7]{};
\node[] (gx8) [right of=gx7b] {$x_{46}$};
\node[] (gx8b)[right of=gx8]{};
\node[] (gx9) [right of=gx8b] {$x_{47}$};
\node[] (gx9b)[right of=gx9]{};
\node[] (gx10) [right of=gx9b] {$y_{27}$};
\node[] (gx10b)[right of=gx10]{};
\node[] (gx11) [right of=gx10b] {$x_{48}$};
\node[] (gx11b)[right of=gx11]{};
\node[] (gx12) [right of=gx11b] {};
\node[] (h0)[below of=gx0]{$7$};
\node[] (h0b)[right of=h0]{$<$};
\node[] (h1)[right of=h0b]{$z_{77}$};
\node[] (h1b)[right of=h1]{$<$};
\node[] (h2) [right of=h1b] {$z_{78}$};
\node[] (h2b)[right of=h2]{$<$};
\node[] (h3) [right of=h2b] {$z_{79}$};
\node[] (h3b)[right of=h3]{$<$};
\node[] (h4) [right of=h3b] {$z_{80}$};
\node[] (h4b)[right of=h4]{$<$};
\node[] (h5) [right of=h4b] {$z_{81}$};
\node[] (h5b)[right of=h5]{$<$};
\node[] (h6) [right of=h5b] {$z_{82}$};
\node[] (h6b)[right of=h6]{$<$};
\node[] (h7) [right of=h6b] {$z_{83}$};
\node[] (h7b)[right of=h7]{$<$};
\node[] (h8) [right of=h7b] {$z_{84}$};
\node[] (h8b)[right of=h8]{$<$};
\node[] (h9) [right of=h8b] {$z_{85}$};
\node[] (h9b)[right of=h9]{$<$};
\node[] (h10) [right of=h9b] {$z_{86}$};
\node[] (h10b)[right of=h10]{$<$};
\node[] (h11) [right of=h10b] {$z_{87}$};
\node[] (h11b)[right of=h11]{$<$};
\node[] (h12) [right of=h11b] {$8$};
\node[] (hx0)[below of=h0]{};
\node[] (hx0b)[right of=hx0]{};
\node[] (hx1)[right of=hx0b]{$x_{49}$};
\node[] (hx1b)[right of=hx1]{};
\node[] (hx2) [right of=hx1b] {$y_{28}$};
\node[] (hx2b)[right of=hx2]{};
\node[] (hx3) [right of=hx2b] {$x_{50}$};
\node[] (hx3b)[right of=hx3]{};
\node[] (hx4) [right of=hx3b] {$y_{29}$};
\node[] (hx4b)[right of=hx4]{};
\node[] (hx5) [right of=hx4b] {$x_{51}$};
\node[] (hx5b)[right of=hx5]{};
\node[] (hx6) [right of=hx5b] {$x_{52}$};
\node[] (hx6b)[right of=hx6]{};
\node[] (hx7) [right of=hx6b] {$y_{30}$};
\node[] (hx7b)[right of=hx7]{};
\node[] (hx8) [right of=hx7b] {$x_{53}$};
\node[] (hx8b)[right of=hx8]{};
\node[] (hx9) [right of=hx8b] {$x_{54}$};
\node[] (hx9b)[right of=hx9]{};
\node[] (hx10) [right of=hx9b] {$y_{31}$};
\node[] (hx10b)[right of=hx10]{};
\node[] (hx11) [right of=hx10b] {$x_{55}$};
\node[] (hx11b)[right of=hx11]{};
\node[] (hx12) [right of=hx11b] {};
\node[] (i0)[below of=hx0]{$8$};
\node[] (i0b)[right of=i0]{$<$};
\node[] (i1)[right of=i0b]{$z_{88}$};
\node[] (i1b)[right of=i1]{$<$};
\node[] (i2) [right of=i1b] {$z_{89}$};
\node[] (i2b)[right of=i2]{$<$};
\node[] (i3) [right of=i2b] {$z_{90}$};
\node[] (i3b)[right of=i3]{$<$};
\node[] (i4) [right of=i3b] {$z_{91}$};
\node[] (i4b)[right of=i4]{$<$};
\node[] (i5) [right of=i4b] {$z_{92}$};
\node[] (i5b)[right of=i5]{$<$};
\node[] (i6) [right of=i5b] {$z_{93}$};
\node[] (i6b)[right of=i6]{$<$};
\node[] (i7) [right of=i6b] {$z_{94}$};
\node[] (i7b)[right of=i7]{$<$};
\node[] (i8) [right of=i7b] {$z_{95}$};
\node[] (i8b)[right of=i8]{$<$};
\node[] (i9) [right of=i8b] {$z_{96}$};
\node[] (i9b)[right of=i9]{$<$};
\node[] (i10) [right of=i9b] {$z_{97}$};
\node[] (i10b)[right of=i10]{$<$};
\node[] (i11) [right of=i10b] {$z_{98}$};
\node[] (i11b)[right of=i11]{$<$};
\node[] (i12) [right of=i11b] {$9$};
\node[] (ix0)[below of=i0]{};
\node[] (ix0b)[right of=ix0]{};
\node[] (ix1)[right of=ix0b]{$x_{56}$};
\node[] (ix1b)[right of=ix1]{};
\node[] (ix2) [right of=ix1b] {$y_{32}$};
\node[] (ix2b)[right of=ix2]{};
\node[] (ix3) [right of=ix2b] {$x_{57}$};
\node[] (ix3b)[right of=ix3]{};
\node[] (ix4) [right of=ix3b] {$y_{33}$};
\node[] (ix4b)[right of=ix4]{};
\node[] (ix5) [right of=ix4b] {$x_{58}$};
\node[] (ix5b)[right of=ix5]{};
\node[] (ix6) [right of=ix5b] {$x_{59}$};
\node[] (ix6b)[right of=ix6]{};
\node[] (ix7) [right of=ix6b] {$y_{34}$};
\node[] (ix7b)[right of=ix7]{};
\node[] (ix8) [right of=ix7b] {$x_{60}$};
\node[] (ix8b)[right of=ix8]{};
\node[] (ix9) [right of=ix8b] {$x_{61}$};
\node[] (ix9b)[right of=ix9]{};
\node[] (ix10) [right of=ix9b] {$y_{35}$};
\node[] (ix10b)[right of=ix10]{};
\node[] (ix11) [right of=ix10b] {$x_{62}$};
\node[] (ix11b)[right of=ix11]{};
\node[] (ix12) [right of=ix11b] {};
\draw[-latex,gray,ultra thick,opacity=0.5](a1) to[out=30,in=150,looseness=1](a9);
\draw[-latex,gray,ultra thick,opacity=0.5](a9) to[out=210,in=30,looseness=1](b6);
\draw[-latex,gray,ultra thick,opacity=0.5](b6) to[out=210,in=30,looseness=1](c3);
\draw[-latex,gray,ultra thick,opacity=0.5](c3) to[out=330,in=210,looseness=0.5](c11);
\draw[-latex,gray,ultra thick,opacity=0.5](c11) to[out=330,in=330,looseness=0.5](d8);
\draw[-latex,gray,ultra thick,opacity=0.5](d8) to[out=210,in=30,looseness=0.5](e5);
\draw[-latex,gray,ultra thick,opacity=0.5](e5) to[out=210,in=30,looseness=0.5](f2);
\draw[-latex,gray,ultra thick,opacity=0.5](f2) to[out=330,in=210,looseness=0.5](f10);
\draw[-latex,gray,ultra thick,opacity=0.5](f10) to[out=330,in=330,looseness=0.5](g7);
\draw[-latex,gray,ultra thick,opacity=0.5](g7) to[out=210,in=30,looseness=0.5](h4);
\draw[-latex,gray,ultra thick,opacity=0.5](h4) to[out=210,in=30,looseness=0.5](i1);
\end{tikzpicture}
}
}
\end{picture}
\caption{Lifted cycle of a $8-11$-periodic orbit ($z_i$) split into
two subsequences, $x_i$ and $y_i$, as in the proof of
Proposition~\ref{pro:concatenatetion} (see
Example~\ref{ex:sebsequences} for text).}
\label{fig:8-11-ordered}
\end{figure}

\begin{example}\label{ex:sebsequences}
Let $P/Q=8/11$, and assume that a circle map $f$ satisfying~\condsC{}
has rotation number $P/Q$. From Propositions~\ref{pro:po_rational_rho}
and~\ref{pro:well_ordered_symb-seq} we know that $f$ has an
$11$-periodic orbit whose lifted cycle is $8,11$-ordered. We wish to
show that, as given by Proposition~\ref{pro:concatenatetion}, its
symbolic sequence is given by the concatenation of the symbolic
sequences associated with $p,q$ and $p',q'$-ordered lifted cycles of
periodic orbits with rotation number $p/q$ and $p'/q'$, respectively,
where $p/q<p'/q'$ are the Farey parents of $8/11$.

In order to find the Farey parents of $8/11$ one can construct the
Farey tree and locate them as the unique Farey neighbours of
$8/11$ at the Farey sequence $\mathcal{F}_{11}$. However, the proof of
Proposition~\ref{pro:concatenatetion} gives us a method to find
$q$ and $q'$: $q$ is the smallest such that
$\Pi\left(F^q(z_0)\right)=z_1$.\\
In Figure~\ref{fig:8-11-ordered} we
show the $8,11$-ordered lifted cycle, $(z_i)$ of the periodic orbit
with rotation number $8/11$. Starting at
$z_0$, the arrows provide the sequence obtained when iterating
$z_0$ by $F$. Note that, at each iterate, one adds $8$ to the
subindex. We obtain that $F^7(z_0)=z_{56}$, and $56=1\pmod{11}$.
Hence, $q=7$ and $q'=11-7=4$. By using that $8=p+p'$ and $p'q-pq'=1$,
we obtain $p=5$ and $p'=3$:
\begin{equation*}
\frac{8}{11}=\frac{5+3}{7+4}.
\end{equation*}
We can now construct the sequence $x_i$. We start with $x_0=z_0$ and
we add $5$ to the subindex of $x$ at each iteration: $x_{5i}=z_{8i}$,
for $0\le i< 7$. This provides the points
$x_0<x_5<x_{10}<\dots<x_{30}$. By lifting these points (adding
multiples of $7$ to their subindex's), we obtain a lifted cycle
defined in Equation~\eqref{eq:sequence_x}, which is $5,3$-ordered.

Next we construct the subsequence $y_i$. Following
Equation~\eqref{eq:sequence_y}, we start with
$y_{pq'}=y_{20}=z_{1+pQ}=z_{56}$. By further iterating, we add $p'$ at
the subindex's of $y_i$: $y_{20+ip'}=z_{56+iP}$, for $0\le i <4$. This
leads to the sequence $y_{20}<y_{23}<\dots<y_{29}$. Finally, by
lifting these points, that is, adding multiples of $4$ to their
subindexes, we obtain the $3,4$-ordered lifted cycle of a periodic
orbit with rotation number $3/4$.

Note that the next iterate of $y_{29}=z_{80}$ becomes $z_{88}$, which
satisfies $\Pi(z_{80})=z_0$. Hence, after following the sequence
$(y_i)$, the lifted cycle $(z_i)$ switches back to the sequence
$(x_i)$, and the symbolic sequence is repeated.
\end{example}

In the rest of this section we study the rotation number for families
of orientation preserving circle maps; that is, under the variation of
the parameter $c$ in condition {\it C.3}, which, by means of the
change of variables given in Equation~\eqref{eq:change_variables}, is
equivalent to the parameter $\lambda$ of parametrizing the curve in
parameter space mentioned in Theorem~\ref{theo:adding_incrementing}.\\

Recalling Proposition~\ref{pro:continuity_rot_num} and
Lemma~\ref{lem:increasing_rotation_number}, we
already have that, when varying $\lambda$ from $0$ to $1$, the
rotation number (and hence the $\eta$-number) is continuous and
monotonically increases from $0$ to $1$.  In order to show that,
moreover, it is a devil's staircase, we need to show that, in
addition, it is constant for all values of $\lambda$ except for a
Cantor set of zero measure. This is will come from the following

\begin{theorem}[\cite{Boy85}, Theorem 1']
Let $f$ be a continuous monotonic non-decreasing map of the circle of degree one
satisfying
\begin{enumerate}[i)]
\item $f$ is constant on an interval $[a,b]$ and of class $C^1$ outside
$[a,b]$
\item $\inf_{y\notin [a,b]}\left( f'(y) \right)>1$
\end{enumerate}
Let $f_t$ be the map defined by $f_t(y)=f(y)+t$, $t\in[0,1]$. Let
\begin{equation*}
E=\left\{ t\,|\, \text{$f_t$ has irrational rotation number} \right\}.
\end{equation*}
Then $m(E)=0$, where $m$ denotes Lebesgue measure, and furthermore $E$ has zero
Hausdorff dimension.
\label{theo:Boy85}
\end{theorem}
Generalizations of the previous theorem can be found
in~\cite{Vee89,Swi89}.

Finally, we show that Theorem~\ref{theo:Boy85} extends to a
piecewise-smooth system of the form~\eqref{eq:normal_form} satisfying
~\condsHp{} and {\em i)} of Theorem~\ref{theo:adding_incrementing} (or
an orientation preserving circle map satisfying~\condsC{}). This will
prove that the $\eta$-number  follows a devil's staircase.
\begin{lemma}
Let $f$ be piecewise-smooth map as in Equation~\eqref{eq:normal_form}
satisfying conditions~\condsHp{} and {i)} of
Theorem~\ref{theo:adding_incrementing}, and let $\tf_\lambda$ be the
orientation preserving circle map obtained after applying the change
of variables~\eqref{eq:change_variables}. Then, the set of values of
$\lambda$ for which the map $\tf_\lambda$ has irrational rotation
number consists of a Cantor set with zero measure.
\label{lem:cantor_set}
\end{lemma}
\begin{figure}
\begin{center}
\begin{picture}(1,0.5)
\put(0,0.35){
\subfigure[\label{fig:invertible_map}]{
\includegraphics[angle=-90,width=0.5\textwidth]{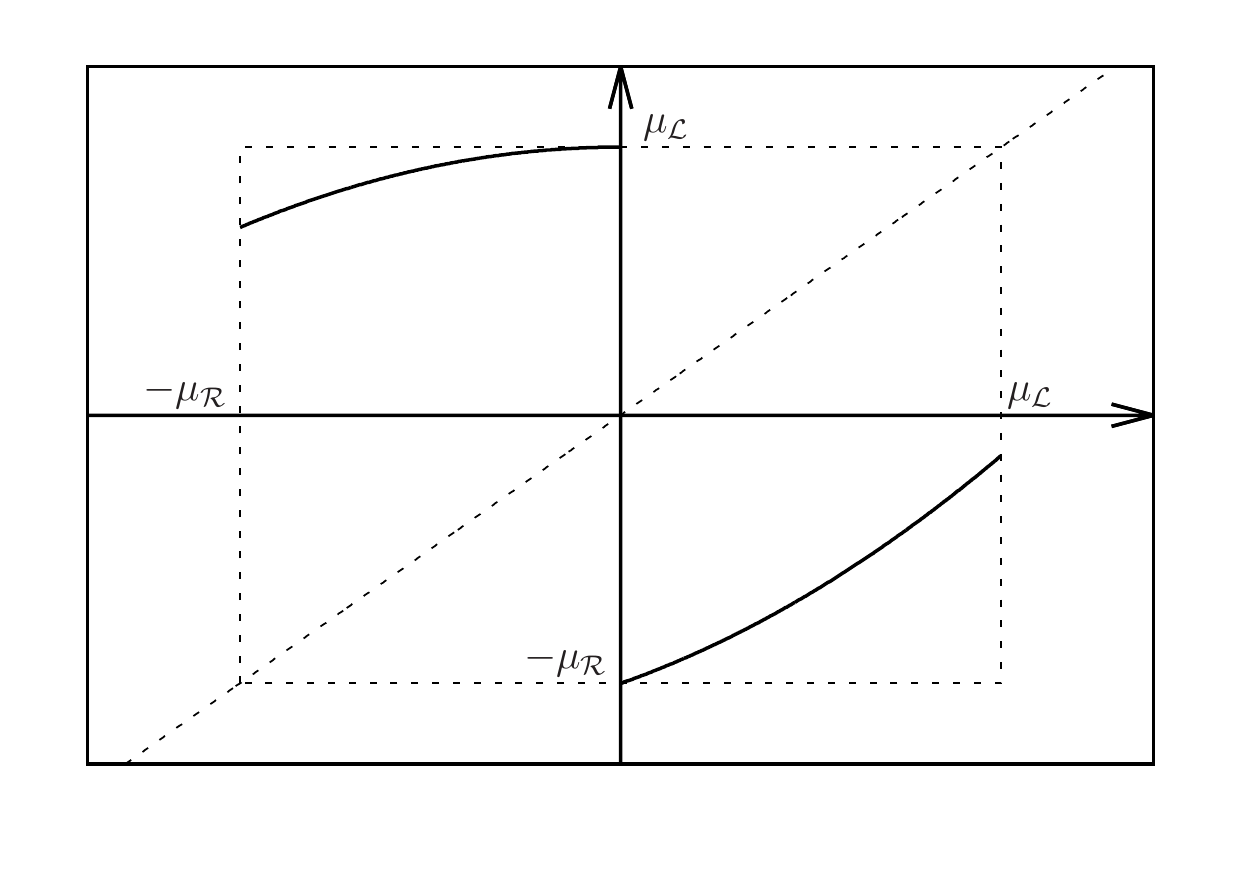}}
}
\put(0.5,0.35){
\subfigure[\label{fig:inverted_map}]{
\includegraphics[angle=-90,width=0.5\textwidth]{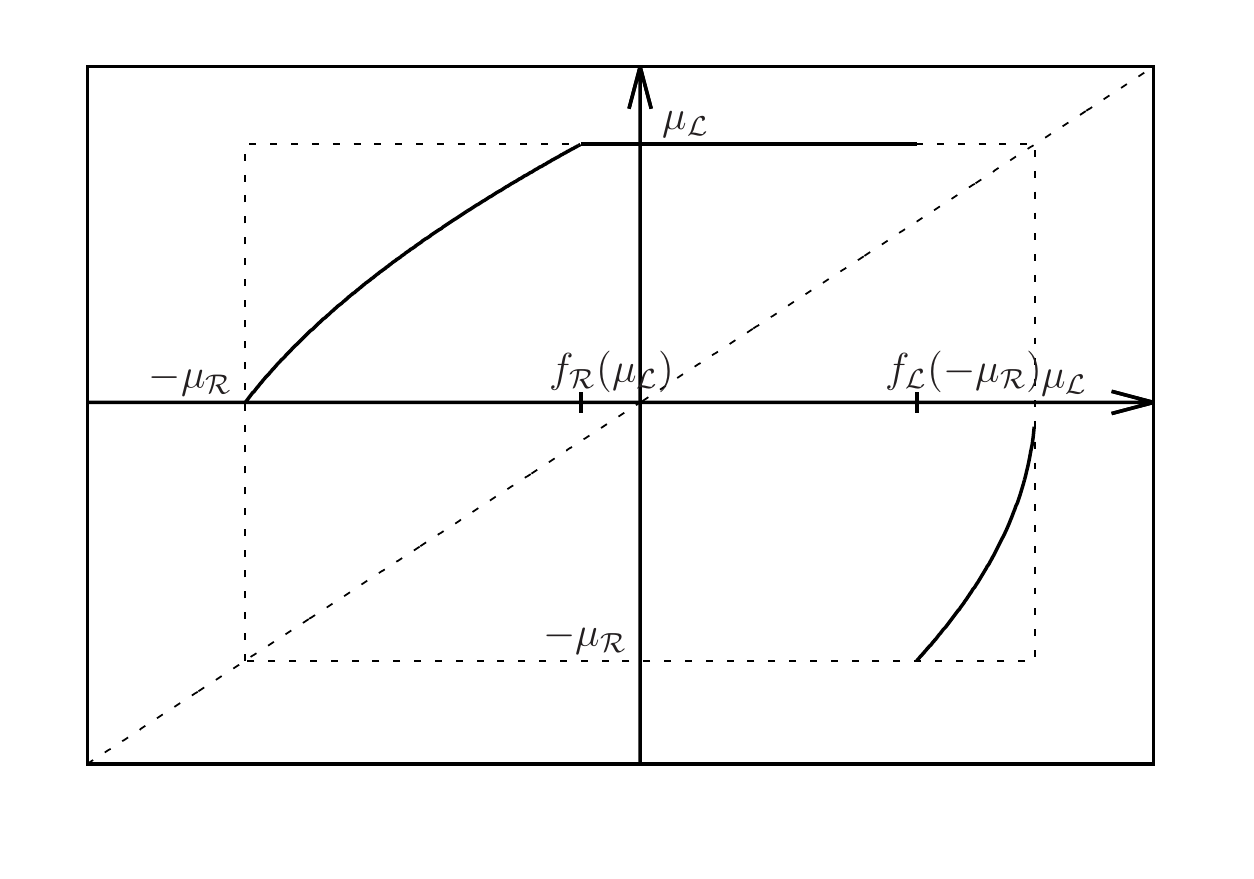}}
}
\end{picture}
\end{center}
\caption{(a) piecewise-smooth map satisfying~\condsC{}. (b) its inverse.}
\label{fig:maps}
\end{figure}
\begin{proof}
We note that the map $\tf_\lambda$ is invertible. Let $\phi$ be as
in~\eqref{eq:change_variables} and define
\begin{align*}
a&:=\tf(1^-)=\phi(f_\R(\mu_\LL))\\
b&:=\tf(0^+)=\phi(f_\LL(-\mu_\R)).
\end{align*}
Then, the inverse $\tf_\lambda^{-1}(y)$ is an increasing expanding map
with a ``hole'' for $y\in[a,b]$ (see Figure~\ref{fig:inverted_map}).  As
the trajectories of any point $x\in[0,1]$ by $\tf$ do not reach the
interval $[a,b]$, we can proceed as the proof of
Proposition~\ref{pro:irrational_rho} and complete the map
$\tf_\lambda^{-1}$ with a horizontal part equal to $1$ for
$y\in[a,b]$. This allows us to consider a map
\begin{equation*}
g(y)=\left\{
\begin{aligned}
&\phi^{-1}\circ f_\R^{-1}\circ \phi(y)&&\text{if }y\in[0,a]\\
&1&&\text{if }y\in[a,b]\\
&\phi^{-1}\circ f_\LL^{-1}\circ \phi(y)&&\text{if }y\in[b,1],
\end{aligned}
\right.
\end{equation*}
which coincides with $\tf_\lambda^{-1}(y)$ for $y\notin(a,b)$.  Let
\begin{equation*}
\begin{array}{cccc}
\gamma:&[0,1]&\longrightarrow&\RR^2\\
&\lambda&\longmapsto &(\mu_\LL(\lambda),\mu_\R(\lambda))
\end{array}
\end{equation*}
be a parameterization satisfying~\condsH{} of
Theorem~\ref{theo:adding_incrementing}. Then, the interval $[a,b]$
smoothly varies from $[0,\phi(f_\LL(-\mu_R(0)))]$ to
$[\phi(f_\R(\mu_\LL(1))),1]$ when $\lambda$ is varied from $\lambda=0$
to $\lambda=1$. Let $g_\lambda(y)$ be $g(y)$ after applying the
reparameterization $\gamma$. Then, we have
\begin{equation*}
g_\lambda(y)=g_0(y)+\lambda,
\end{equation*}
and hence we can apply Theorem~\ref{theo:Boy85} and get the result.
\end{proof}

\section{Non-orientable case}\label{sec:incrementing_proof}
The non-orientable (increasing-decreasing) case occurs under
conditions {\em ii)} of Theorem~\ref{theo:adding_incrementing}. In
this case, the map~\eqref{eq:normal_form} becomes increasing for $x<0$
and decreasing for $x>0$.\\
As mentioned in Section~\ref{sec:introduction}
and Section~\ref{sec:inc-dec_overview}, this situation was discussed
in~\cite{Hom96} Section~3.3 and proven in full detail in~\cite{AvrGraSch11}.
For completeness, we provide in this section an overview of this
proof.

We first observe (see Figure~\ref{fig:abs-inter_preimages}) that a map
$f$ as in Equation~\eqref{eq:normal_form} satisfying conditions~\condsHp{}
and {\em ii)} of Theorem~\ref{theo:adding_incrementing} is a map on
the interval
\begin{equation}
f:[f_\R(\mu_\LL),\mu_\LL]\longrightarrow [f_\R(\mu_\LL),\mu_\LL].
\label{eq:map_restricted_interval}
\end{equation}
\begin{figure}
\begin{center}
\includegraphics[width=0.6\textwidth,angle=-90]{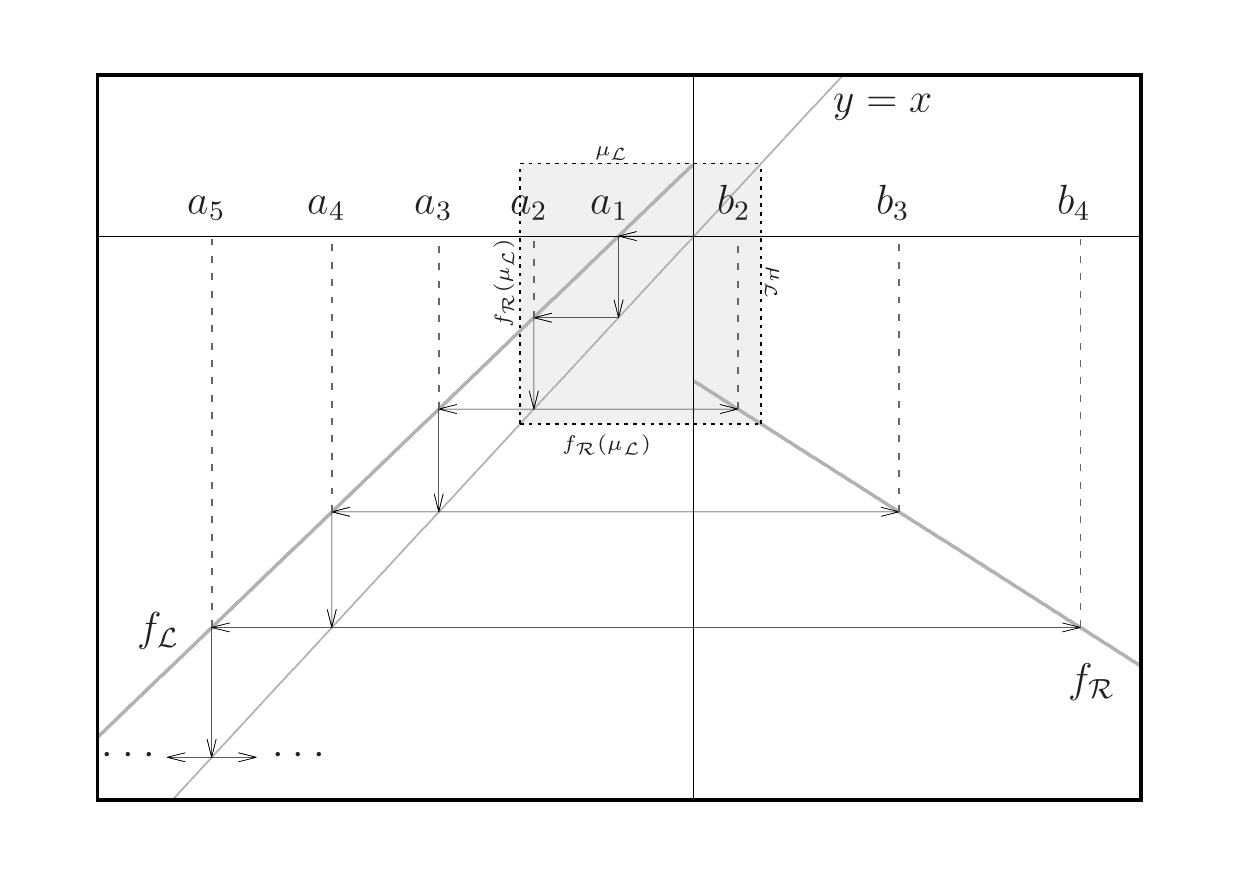}
\end{center}
\caption{Piecewise-smooth map as in Equation~\eqref{eq:normal_form}
satisfying {\em ii)} of Theorem~\ref{theo:adding_incrementing} and the
sequences defined in Equation~\eqref{eq:an}-\eqref{eq:bn} The gray box
represents the absorbing interval~\eqref{eq:map_restricted_interval}.
}
\label{fig:abs-inter_preimages}
\end{figure}

We now consider the symbolic itineraries of periodic orbits for a
map~\eqref{eq:map_restricted_interval} under the  mentioned conditions. Clearly,
when the parameters $\mu_\LL$ and $\mu_\R$ are varied along the
parametrization~\eqref{eq:1d_scann_curve} satisfying~\condsH{}, the
map~\eqref{eq:map_restricted_interval} becomes negative for $x\in
[0,\mu_\LL]$. Hence, periodic orbits cannot have symbolic itineraries
with consecutive $\R$ symbols.\\
In order to show that only possible itineraries for periodic orbits
are of the form $\LL^n\R$, we define the sequences of preimages of
$x=0$ by $f_\LL$ and $f_\R$:
\begin{align}
a_0 = 0, \quad  a_n &= f_\LL^{-1}(a_{n-1})\; 
                 &&\text{if}\,\; n > 0, \label{eq:an}\\
b_n &= f_\R^{-1}(a_n)\; 
                 &&\text{if}\,\; n \ge n_0.\label{eq:bn}
\end{align}
Due to the contractiveness of $f_\R$ and $f_\LL$, we have that
\begin{align}
\frac{m([a_{n+1},a_n])}{m([a_n,a_{n-1}])}>1,\;n>0\label{eq:des1}\\
\frac{m([b_{n},b_{n+1}])}{m([b_{n-1},b_{n}])}>1,\;n> n_0,\nonumber
\end{align}
where $m$ indicates the length of the interval.

Then we have the following
\begin{lemma}\label{lem:unique_aj}
Suppose $f$ is a map of type (\ref{eq:normal_form}) fulfilling
conditions~\condsHp{} and {\em ii)} of
Theorem~\ref{theo:adding_incrementing}.  Then there exists at most one
$a_j$ (equiv. $b_j$) such that $a_j\in f_r\left((0,\mu_\LL]\right)$
(equiv. $b_j\in (0,\mu_\LL]$).
\end{lemma}

\begin{figure}\centering
\includegraphics[width=0.7\textwidth]{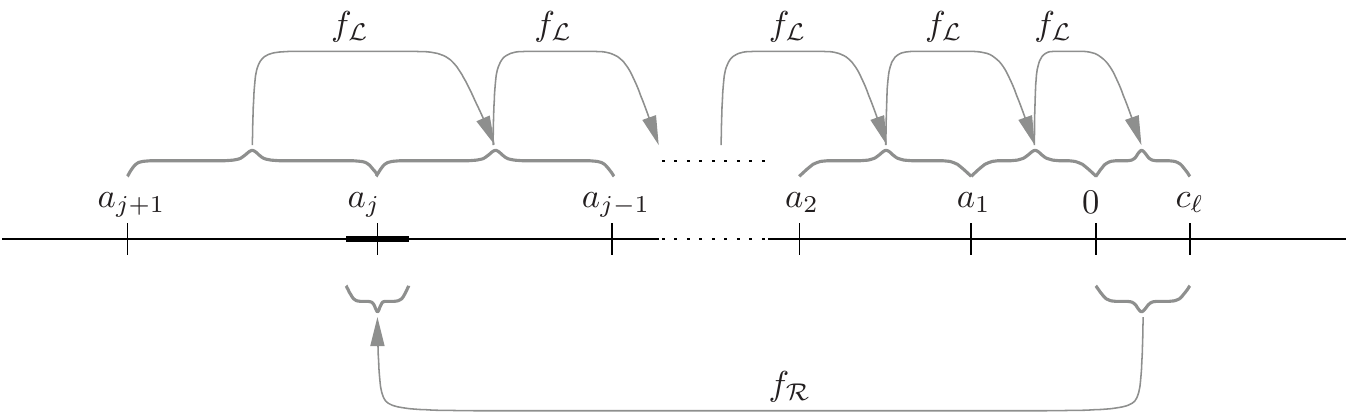}
\caption{Backward and forward iterates of $(0,\mu_\LL]$.
$f_\R((0,\mu_\LL])$ (dark segment) is smaller than
$f_\LL^{-n}((0,\mu_\LL])$ $\forall n$. Therefore, at most one $a_j$
can be reached by $f_\R((0,\mu_\LL])$.}
\label{fig:unique_an}
\end{figure}

\begin{proof}
Recalling that $\mu_\LL=f_\LL(0)$, one has (see
Figure~\ref{fig:unique_an})
\begin{align*}
[a_{n+1},a_n]&=f_\LL^{-1}([a_n,a_{n-1}])\\
\big[a_1,0]&=f_\LL^{-1}([0,\mu_\LL]).
\end{align*}
Using the property shown in Equation~(\ref{eq:des1}) one has
\begin{equation*}
m([0,c_m])<m([a_{n+1},a_n])\;\forall n,
\end{equation*}
and, since $f_\R$ is a contracting one obtains
\begin{equation*}
m (f_r((0,\mu_\LL])<m([a_{n+1},a_n])\;\forall n.
\end{equation*}
Therefore, at most one $a_n$ can be located in $f((0,\mu_\LL])$.
\end{proof}

The next Lemma tells us what symbolic itineraries are possible.
\begin{lemma}[\cite{AvrGraSch11} Lemma~7]\label{lem:lnr_forall_n}
Let $f$ be  a piecewise-smooth map as considered above, and let 
$x_0<\cdots<x_n$, $x_i\in[f_\R(\mu_\LL),\mu_\LL]$ be a periodic orbit. Then,
$I_f(x_0)=(\LL^n\R)^\infty$.
\end{lemma}

\begin{figure}[!h]\centering
\includegraphics[width=0.8\textwidth]{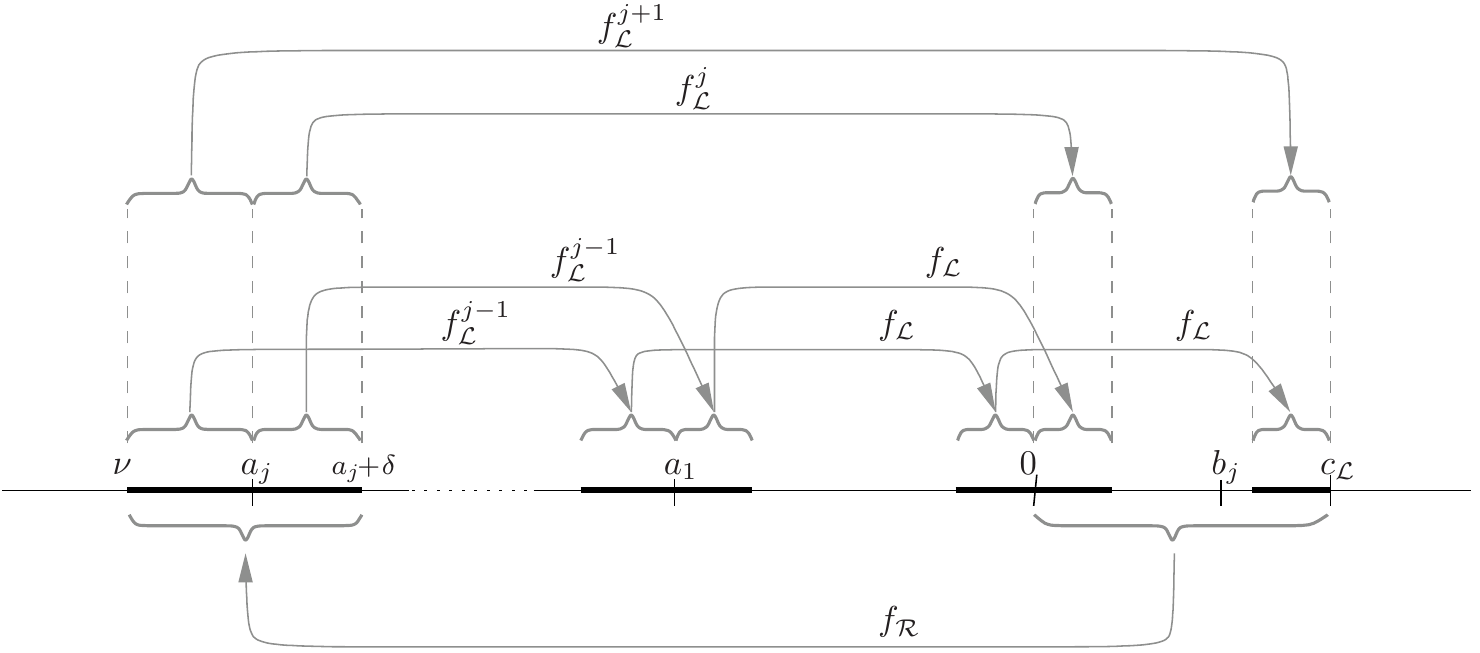}
\caption{The interval $f_\R((0,\mu_\LL])$ is split when it returns to the right
domain.}\label{fig:only_rln}
\end{figure}

The proof of this Lemma is in fact an extension of the arguments presented
in \cite{Hom96} Section~3.3, and was provided in full detail
in~\cite{AvrGraSch11}. We provide it here for completeness.
\begin{proof}
If $f$ has periodic orbit of period $n\ge 2$, we necessary have
\begin{equation*}
f_\R((0,\mu_\LL])\cap [a_{n},a_{n-1}]\ne\emptyset,
\end{equation*}
which can be given due to one of the next three situations (see
Figures~\ref{fig:unique_an} and \ref{fig:only_rln})
\renewcommand{\labelenumi}{S.\arabic{enumi}}
\begin{enumerate}
\item $a_{n-1}\in f_\R((0,\mu_\LL])$
\item $f_\R((0,\mu_\LL])\subset (a_{n},a_{n-1})$
\item $a_{n}\in f_\R((0,\mu_\LL])$
\end{enumerate}
If S.1 holds, $b_{n-1}\in (0,\mu_\LL]$ and
\begin{eqnarray*}
f_\LL^{n}f_\R:\,[b_{n-1},\mu_\LL]\longrightarrow [b_{n-1},\mu_\LL]\\
f_\LL^{n-1}f_\R:\,(0,b_{n-1})\longrightarrow (0,b_{n-1}),
\end{eqnarray*}
are continuous contracting functions which must have a unique (stable) fixed
point. Therefore, two stable periodic orbits of type $\LL^n\R$ and
$\LL^{n-1}\R$ coexist. Note that for $n=2$ this proves also the
existence of a $\LL\R$-periodic orbit.\\
In the second case (S.2), $b_{n-1}\notin (0,\mu_\LL]$
($[0,\mu_\LL]\subset(b_{n-1},b_n)$) and
\begin{equation*}
f_\LL^nf_\R:(0,\mu_\LL]\longrightarrow (0,\mu_\LL]
\end{equation*}
is a continuous contracting function which also must have a unique (stable)
fixed point. In this case, there exists a unique periodic orbit of type
$\LL^n\R$ which is the unique attractor in $(0,\mu_\LL]$.\\
Finally, if S.3 holds, replacing $n$ by $n-1$ and arguing as in S.1, one has
that a stable periodic orbit of type $\LL^n\R$ coexists with a
stable $\LL^{n+1}\R$-periodic one.
\end{proof}

\begin{remark}
By contrast to all periodic orbits of type $\LL^n\R$ with $n\ge2$, the
periodic orbit $\LL\R$ exists not only for $\mu_\R>0$ but
also for $\mu_\R\le0$. In that case, it coexists with the fixed point
$\R$ ($\LL^0\R$) (see Figure~\ref{fig:regions_incrementing}).
\end{remark}

\begin{remark}\label{rem:border_collision}
Note that the transitions between cases S.1, S.2 and S.3 are given by
border collision bifurcations where the respective periodic orbits are
created or destroyed when they collide with the boundary $x=0$. This
defines the border collision bifurcation curves shown in
Figure~\ref{fig:regions_incrementing}.
\end{remark}

\begin{remark}\label{rem:sepparatrix}
As it is known, invariant objects of piecewise-smooth systems do not
necessarily have to be separated by another invariant object. In this
case, the coexistence of stable periodic objects may also be separated
by the discontinuity (and its preimages).
\end{remark}

We focus finally on the bifurcation scenario shown inf
Figures~\ref{fig:regions_incrementing}-\ref{fig:1dscann_incrementing}
when the parameter $\lambda$ of the
parametrization~\eqref{eq:1d_scann_curve} is varied.

Clearly, the successions $a_n$ and $b_n$ are continuous functions of
the parameter $\gamma$. Hence, the transitions between S.1, S.2 and
S.3 described in the proof of Lemma~\ref{lem:lnr_forall_n}.  That is,
assume that for some $\gamma$ S.1 holds; then two periodic orbits of type
$\LL^n\R$ and $\LL^{n-1}\R$ coexists. Due to
conditions~\condsH{}, when decreasing $\gamma$,
$\mu_\LL$ monotonically decreases to zero, whereas $\mu_\R$ increases
towards some bounded value. Hence, $a_{n-1}$ monotonically decreases
to zero and, hence, as the $m([0,\mu_\LL])<m([a_n,a_{n-1}])$, there
exists some value of $\gamma$ for which $f_\R([0,\mu_\LL])\subset
(a_n,a_{n-1})$ and S.2 holds. Hence, the $\LL^{n-1}\R$-periodic orbit
bifurcates and only a $\LL^n\R$-periodic orbit exists. Arguing
similarly, by further decreasing $\gamma$, S.3 holds, a
$\LL^{n+1}\R$-periodic orbit bifurcates and coexists with the
$\LL^n\R$-periodic orbit.\\
As this occurs for all $n$, this argument proves the bifurcation
scenario described in Section~\ref{sec:inc-dec_overview}.

\section{Some remarks on piecewise-smooth expanding maps}\label{sec:remarks_expansive}
Frequently, both in applications and theoretical studies, the
contracting conditions required in the previous sections become too
restrictive and one needs to deal with discontinuous maps exhibiting
expansivenesses. However, such maps are shown to undergo similar
bifurcations as the ones described in Section~\ref{sec:overview}.
Examples are found in the study of homoclinic bifurcations,
such as the Lorenz system
(\cite{GucWil79,Spa82,ProThoTre87,Gle90,GleSpa93,HomKra00})
or in Multiscale dynamics (\cite{Lev90,DesGucKraKueOsiWec12}), but
also in power electronics
(\cite{HamDeaJef92,BanKarYuaYor00,BizStoGar06,Saitoetal07,MatSai08,KapBanPat10,ZhuMosAndMik13,AvrMoseZhuGar15}),
biology \cite{KeeHopRin81,MonDanSelTsiHas11,SigTouVid15} or economy
(\cite{TraWesGar10}), among others. Although the results summarized in
Theorem~\ref{theo:adding_incrementing} cannot be applied when the map
lacks of
contractivness, many of the results presented in
Sections~\ref{sec:adding_proof} and~\ref{sec:incrementing_proof} are
still valid under the assumption of weak enough expansion. In this
section we review possible scenarios led by the loss of
contractiveness, that is, we focus on maps
as in Equation~\eqref{eq:normal_form} satisfying
\begin{enumerate}[{\it h'.1}]
\item $f_\LL(0)=f_\R(0)=0$.
\item $0<(f_\LL(x))'$, $x\in(-\infty,0)$\label{hip:2'}
\item $0<\left|(f_\R(x))'\right|$, $x\in(0,\infty)$,
\end{enumerate}
with $f_\LL$ and $f_\R$ may be expanding in all or part if their
domains.

As before, we now distinguish between the orientable ($(f_\R(x))'>0$)
and non-orientable ($(f_\R(x))'<0$) cases.

\subsection{Orientable case}
Let us assume that conditions~\condsHpp{} hold with $(f_\R(x))'>0$. We
focus on the case $\mu_\LL,\mu_\R>0$.  We first note that, as long as
\begin{equation}
f_\LL(-\mu_\R)\ge f_\R(\mu_\LL),
\label{eq:invertiblity}
\end{equation}
the piecewise-smooth map $f$ is still invertible. Hence, as in the
contractive case, it can 
be reduced to a circle map of type~\eqref{eq:circle_map} satisfying
conditions~\condsC{} by means of the smooth change of
variables~\eqref{eq:change_variables}. Therefore, all the results in
Section~\ref{sec:adding_proof} requiring only that the lift of the map
is strictly increasing (which is given by conditions~\condsC{}) also
hold as long as Equation~\eqref{eq:invertiblity} is fulfilled. Thats
is, recalling that $\lambda$ parametrizes the curves in
Equation~\eqref{eq:1d_scann_curve}, we get that, for any
$\lambda\in[0,1]$ such that~\eqref{eq:invertiblity} is satisfied, we
have that
\begin{enumerate}[\it i)]
\item the rotation number ($\rho$) given in
Definition~\ref{def:rotation_number} is well defined, unique and
increasing as a function of $\lambda$,
\item if $\rho=p/q$, with $p,q$ co-prime, then $f$ possesses a
$p,q$-ordered periodic  orbit whose symbolic sequence belongs to the
Farey tree of symbolic sequences.
\end{enumerate}
Note that {\it ii)} is a consequence of
Proposition~\ref{pro:po_rational_rho}, which requires only
conditions~\condsC{}. However, as noted in
Remark~\ref{rem:uiqueness_of_po}, when the maps $f_\LL$ and $f_\R$
exhibit expansiveness, this periodic may be not only repelling but
also non unique. However, even in this case, all existing periodic
orbits are $p,q$-ordered, as Proposition~\ref{pro:pq-ordred} still
holds, and therefore they all must have the same symbolic dynamics.\\
Regarding the dynamics for $\rho\in\RR\backslash\Q$,
Proposition~\ref{pro:irrational_rho} does not hold if $f_\LL$ or
$f_\R$ are expanding. The main obstacle is given by the fact that
Corollary 3.3 of~\cite{Vee89} cannot be applied to show that
$\omega(f)$ is a Cantor set (see the proof of
Proposition~\ref{pro:irrational_rho}). As a consequence, one needs
additional conditions to state whether $\omega(f)$ is a Cantor set or the
whole circle.

As opposite to the purely contractive case, when $f_\LL$ or $f\R$
are expanding, the rotation number no longer necessary follows a
devil's staircase. Without the contracting assumption
Theorem~\ref{theo:Boy85} does not hold and hence the set of values of
$\lambda$ for which the piecewise-smooth map does not possess any
periodic orbit (the rotation number of the associated circle map is
irrational) does not necessary have zero measure. However,
$\rho(\lambda)$ is still a non-decreasing continuous function.\\

As stated above, the crucial property that keeps most of the results
shown in Section~\ref{sec:adding_proof} valid even when the map is
expansive is the invertibility condition~\eqref{eq:invertiblity}. When
this is lost, then the corresponding lift exhibits negative gaps and
is not an increasing map. As a consequence, the rotation number
becomes non-unique. Instead, when needs to deal with rotation
intervals, coexistence of periodic orbits with different rotation
numbers and symbolic dynamics and positive entropy. Although there
exist many results in the literature (see for example
\cite{AlsLli89,AlsLliMisTre89,AlsLliMis00,AlsFal03,AlsMan90,AlsMan96}),
a precise description of the bifurcation scenarios becomes difficult
under general assumptions might be difficult to state.

\subsection{Non-orientable case}
Let us now consider a class of maps as in~\eqref{eq:normal_form}
satisfying
\begin{enumerate}[{\it h''.1}]
\item $f_\LL(0)=f_\R(0)=0$.
\item $0<(f_\LL(x))'$, $x\in(-\infty,0)$
\item $0>\left|(f_\R(x))'\right|$, $x\in(0,\infty)$,
\end{enumerate}
with $f_\LL$ and $f_\R$ may be expanding in all or part if their
domains.\\
Unlike in the orientable case, such a map fulfilling
conditions~\condsHppp{} is always invertible for $\mu_\LL,\mu_\R>0$.
However, even in the linear case, when the map $f_\LL$ or $f_\R$ is
expanding, the bifurcation scenario may change significantly (see for
example \cite{AvrEckSchSch12}). On one hand, provided that the
existence and pairwise uniqueness of periodic orbits for the
contracting case relies on Brower's fixed point Theorem, it may happen
that no periodic orbits may exist or several unstable ones may
co-exist. On the other hand, whenever they exist, symbolic sequences
associated with periodic orbits may be very different.  Obviously,
provided that, for $\mu_\R>0$, $f_\R(x)\le0$ if $x\ge0$, no symbolic
sequences can contain two consecutive $\R$'s.  Then, symbolic
sequences may contain blocks of the form $\LL^n\R\LL^m$, with $n\ne
m$. Numerical studies have also shown evidence of the existence of
chaotic attractors, whose associated symbolic dynamics are
similar~(\cite{AvrEckSch08a,AvrEckSch08b,AvrEckSch08c}). 

However, if $f_\R$ is expansive but $f_\LL$ is not, it may happen
that, for $\mu_\LL>0$ smaller than a certain quantity (which may
depend on $\mu_\R$), the number of iterations performed in the
negative domain may be large enough to compensate the expanding
dynamics of $f_\R$. This is because the smaller $\mu_\LL>0$ the larger
the number of iterations needed to return to the right domain. Hence,
in such situation, one recovers the period incrementing scenario
described in Section~\ref{sec:inc-dec_overview}, with and attracting
periodic orbits of type $\LL^n\R$ with $n$ larger than a certain
value. As in the attracting case, there will exist parameter values
for which there exists coexistence of periodic orbits of type
$\LL^n\R$ and $\LL^{n+1}\R$.

\section{Maximin itineraries and piecewise-smooth maps in $\RR^n$}\label{sec:maximin_approach}
\subsection{Introduction}\label{sec:maximin_intro}
Often, due to the complexity of real applications, one-dimensional
maps are not enough and one needs to consider maps in higher
dimensions. Examples of two (or higher)-dimensional piecewise-smooth
maps are found as Poincar\'e (or stroboscopic) maps in non-autonomous
mechanical systems (\cite{Bro99,Nordmark01,BerBudChaKow08}), but their
are also found in power electronics
(\cite{BerGarGliVas98,BanGre99,RakAprBan10,GiaBanImr12,ZhuMosAndMik13,AmaCasGraOliHur14}),
control theory (\cite{BerGarIanVas02,ZhuMos03,Zhuetal01,FosGra13}) or
mathematical-neuroscience or biology
(\cite{TouBre08,TouBre09,FreGal11,MenHugRin12,Ton14,SigTouVid15}).
Piecewise-smooth maps in higher dimensions are also obtained when
studying Filippov flows in $\RR^n$ ($n\ge 3$) (\cite{Fil88}); they
appear as half-return Poincar\'e maps, and become discontinuous close to
grazing and Hopf bifurcations
(\cite{FrePonRos05,NorKow06,KBCHHKNP06,PonRosVel13,CarFerFerGarTer12,CarFerFerGarTer14}).\\
As it occurs with homoclinic bifurcations for smooth flows (see
Sections~\ref{sec:introduction},~\ref{sec:inc-inc_overview}
and~\ref{sec:co-dimension_two} for references), symbolic dynamics of
piecewise-smooth discontinuous maps in $\RR^n$ may help to 
understand the number of loops performed by periodic orbits in
different zones of the state space
(\cite{NorKow06,CarFerFerGarTer14}). There exist some general results
for piecewise-smooth continuous maps on the plane
(\cite{SimMei08,Sim09,FoxMei13}) (see~\cite{Sim15} for a recent survey).
Unfortunately, although some efforts have been done by exhaustively
studying particular systems
(\cite{Devaney84,Glen01,BanGre99,Kowalczyk05,DutRouBanAla07,FouChaGar10,FreGal11,AmaCasGraOliHur14})
there exists very few literature providing general results for general
piecewise-smooth discontinuous maps.

\subsection{Maximin/minimax properties of symbolic sequences}
We continue at the symbolic level by defining the {\em
maximin/minimax} properties of symbolic sequences in $W_{p,q}$.
\begin{definition}\label{def:maximin}
We say that a symbolic sequence $\x\in W_{p,q}$ is
\begin{itemize}
\item maximin if
\begin{equation*}
\min_{0\le k\le q}\left( \sigma^k(\x) \right)=\max_{\y\in
W_{p,q}}\left(
\min_{0\le k\le q}\left( \sigma^k(\y) \right)
\right),
\end{equation*}
\item minimax if
\begin{equation*}
\max_{0\le k\le q}\left( \sigma^k(\x) \right)=\min_{\y\in
W_{p,q}}\left(
\max_{0\le k\le q}\left( \sigma^k(\y) \right)
\right).
\end{equation*}
\end{itemize}
\end{definition}
As it was proven in~\cite{Ber82,GamLanTre84} using different
techniques, the maximin and minimax properties are equivalent. That
is, one has the following
\begin{theorem}[\cite{Ber82,GamLanTre84}]
Let $\x\in W_{p,q}$. Then, $\x$ is maximin if and only if it is
minimax.
\label{theo:maximin-minimax}
\end{theorem}

\begin{example}
Let $\eta=2/5$. Up to cyclic permutations, there exist only two
periodic sequences in $W_{2,5}$, which are represented by means of the
minimal and maximal blocks
\begin{align*}
\LL^3\R^2&=\min_{0\le k\le q}\left( \sigma\left( \LL^3\R^2 \right)
\right)\\
\R^2\LL^3=\sigma^2(\LL^3\R^2)&=\max_{0\le k\le q}\left( \sigma\left(\LL^3\R^2 
\right)\right)
\end{align*}
and
\begin{align*}
\LL^2\R\LL\R&=\min_{0\le k\le q}\left( \sigma\left( \LL^2\R\LL\R \right)
\right)\\
\R\LL\R\LL^2=\sigma^2(\LL^2\R\LL\R)&=\max_{0\le k\le q}\left( \sigma\left( \LL^2\R\LL\R
\right)\right).
\end{align*}
Then, as
\begin{align*}
\LL^2\R\LL\R&>\LL^3\R^2\\
\R\LL\R\LL^2&<\R^2\LL^3,
\end{align*}
the sequence $\LL^2\R\LL\R$ is minimax and maximin.
\end{example}

The following result tells us that all the sequences shown in the
Farey tree of symbolic sequences (Figure~\ref{fig:farey_sequences})
(given by consecutive concatenation), are maximin (minimax).

\begin{proposition}[\cite{Gam87} Proposition II.2.4-3]\label{pro:omega-ordered_minimax}
Let $p/q<p'/q'$ be the irreducible form of two Farey neighbours, and let
$\x^\infty\in W_{p,q}$ and $\y^\infty\in W_{p',q'}$ two
maximin sequences, with $\x\in \left\{ \LL,\R \right\}^q$ and
$\y\in\left\{ \LL,\R \right\}^{q'}$ minimal blocks. Then, the
sequences given by the concatenation of these two blocks
$(\x\y)^\infty\in W_{(p+p')/(q+q')}$ is maximin.
\end{proposition}
We provide a sketch of the proof provided in~\cite{Gam87}.
\begin{proof}
One first sees that $\x$ and $\y$ belong to the same domain of some
deflation; i.e, a map which collapses blocks contained in sequences as
follows:
\begin{equation*}
\pi=\left\{
\begin{aligned}
&\LL^{n+1}\R \longrightarrow \R\\
&\LL^n\R\longrightarrow \LL.
\end{aligned}
\right.
\end{equation*}
Then, one proceeds by induction using that $\pi(\x)$ and $\pi(\y)$
are maximin, and also is $\pi^{-1}(\x\y)$. See~\cite{Gam87} for more
details.
\end{proof}

The following result tells us that maximin symbolic sequences are also
well ordered sequenes (see
Definition~\ref{def:well-ordered_symbolic-sequence}).
\begin{theorem}[\cite{GamLanTre84}]
Let $\x\in W_{p,q}$. Then, $\x$ is maximin if and only if it is
$p,q$-ordered.
\label{theo:pq-ordered_iff_maximin}
\end{theorem}
The previous result was also proven in~\cite{Ber82} for the
one-dimensional case using endomorphisms of the circle. Moreover, note
that,  also for the one-dimensional case we already have such result
by using the Farey tree of symbolic sequences and  Proposition
~\ref{pro:concatenatetion}. However, the previous Theorem was proven
using only symbolic properties.

\subsection{Quasi-contractions and piecewise-smooth maps in $\RR^n$}\label{sec:quasi-contractions}
In this section we review some results for piecewise-smooth maps in
$\RR^n$ regarding their symbolic properties.\\
We consider maps defined in some suitable open set $U\subset\RR^n$,
\begin{equation*}
f:U\longrightarrow U,
\end{equation*}
of the following form.  Let 
\begin{equation*}
h:\RR^n\longrightarrow \RR
\end{equation*}
some differentiable function. We then consider a switching manifold
$\Sigma\subset U$ as
\begin{equation*}
\Sigma=h^{-1}(0)\cap U.
\end{equation*}
which splits $U$ in two subsets
\begin{equation*}
E_\LL=\left\{ x\in U,\,|\,h(x)<0 \right\}\quad \text{and}\quad
E_\R=\left\{ x\in U,\,|\,h(x)>0 \right\}.
\end{equation*}
Then we write such maps as
\begin{equation*}
f:E_\LL\cup E_\R\longrightarrow U
\end{equation*}
defined as
\begin{equation}
f(x)=\left\{
\begin{aligned}
&f_\LL(x)&&\text{if }x\in E_\LL\\
&f_\R(x)&&\text{if }x\in E_\R,
\end{aligned}\right.
\label{eq:Rn_pwmap}
\end{equation}
where
\begin{equation*}
f_\LL:\RR^n\longrightarrow \RR^n\quad \text{and}\quad
f_\R:\RR^n\longrightarrow \RR^n
\end{equation*}
are two smooth maps.\\
Note that in Equation~\eqref{eq:Rn_pwmap} does not define $f$ in
$\Sigma$. Similarly as we did for the one-dimensional case, we will
consider that $f$ is bi-valued for $x\in \Sigma$:
\begin{equation*}
f(\Sigma)=f_\LL(\Sigma)\cup f_\R(\Sigma).
\end{equation*}
Given $x\in U$ one can also define its symbolic itinerary by $f$ as in
Equation~\eqref{eq:symbolic_sequence} by considering now the encoding
\begin{equation}
a(x)=
\left\{
\begin{aligned}
&\R&&\text{if }x \in E_\R\\
&\LL&&\text{if }x\in E_\LL.
\end{aligned}
\right.
\label{eq:encoding_LR_Rn}
\end{equation}
Note that, if $\x$ is the symbolic sequence associated with a periodic
orbit of a map of type~\eqref{eq:Rn_pwmap} then it still makes sense
to consider its $\eta$-number, $\eta(\x)$, as defined in
Equation~\eqref{eq:eta_number}. However, we will not provide a definition of
the rotation number for such maps, although the classical definition
through the lift for the one-dimensional case (see
Definition~\ref{def:rotation_number}) could be extended by considering
$f$ as a map onto an $n$-dimensional cylinder (see
Section~\ref{sec:conclusions_future} for a discussion).

Below we will assume now that $f$ is a contraction (or just $f$
contracts), which will mean that there will exist some $0<k<1$ such
that, for any $P,Q\in U$ we have
\begin{equation}\label{eq:contraction}
\lVert f(P)-f(Q)\rVert < k\lVert P-Q\rVert.
\end{equation}
We then recover the result presented by Gambaudo et al.
in~\cite{GamTre88} which states the possible number of periodic orbits
and their possible symbolic properties.
\begin{theorem}[\cite{GamTre88}]\label{theo:maximin_quasi-contraction}
Let $f$ be a piecewise-smooth map as defined above satisfying
\begin{enumerate}
\item $f^n(\Sigma)\cap \Sigma=\emptyset$ for all $n>0$,
\item $f$ contracts.
\end{enumerate}
Then,
\begin{enumerate}[i)]
\item $f$ admits $0$, $1$ or $2$ periodic orbits
\item any periodic orbit of $f$ has an itinerary which is maximin
\item if $f$ has two periodic orbits, then their itineraries belong to
$W_{p,q}$ and $W_{p',q'}$ and $p/q$ and $p'/q'$ are Farey neighbours.
\end{enumerate}
\end{theorem}
Before discussing it, we remark that the previous Theorem was indeed
stated for a more general type of piecewise-defined maps called
\emph{quasi-contractions}. For completeness we provide such definition
in its original form:
\begin{definition}[Quasi-contraction]\label{def:quasi-contraction}
Let $(E_0,d_0)$ and $(E_1,d_1)$ be two metric spaces and $F_0\in E_0$
and $F_1\in E_1$ two points. Then, a map
\begin{equation}
f:\,E_0\cup E_1\longrightarrow E_0\cup E_1
\label{eq:quasi_contraction}
\end{equation}
is a quasi contraction if there exists $0\le k\le 1$ such that for any
$(i,j)\in \left\{ 0,1 \right\}^2$, $\forall P,Q\in f^{-1}(E_j)\cap
E_i$, $\forall R\in f^{-1}(E_{1-j})\cap E_i$ one has
\begin{enumerate}[i)]
\item $d_j\left( f(P),f(Q) \right)\le k d_i\left( P,Q \right)$
\item $d_j\left( f(P),F_j \right)+d_{1-j}\left( f(R),F_{1-j}
\right)\le kd_i(P,R)$
\end{enumerate}
\end{definition}
Despite the arbitrary dimension of the metric spaces $E_i$,
Theorem~\ref{theo:maximin_quasi-contraction} was stated keeping in mind the
one-dimensional case. This is why the two points $F_i$ were
considered instead of manifolds. When considering $E_0=(-\infty,0]$ and
$E_1=[0,\infty)$, then $F_i$ would be chosen to coincide with the boundary:
$F_0=F_1=0$ and we recover the type of one-dimensional maps  as in
Equation~\eqref{eq:normal_form}.\\
Theorem~\ref{theo:maximin_quasi-contraction} was proved
in~\cite{GamGleTre88} (Theorem $A$) for  $f$ a quasi-contraction  with
constant $k$ satisfying  $0\le k\le 1/2$. The version given
in~\cite{GamTre88} not only extends to the case of the
quasi-contractions with $0\le k\le 1$, but it also provides more
information; we give here a mutilated version which is enough for our
purposes. It can be easily seen that both proofs also hold when one
replaces the points $F_i$ by closed sets whose all iterates are kept
connected. This is why we required condition \emph{1.} in
Theorem~\ref{theo:maximin_quasi-contraction} instead of considering
general quasi-contractions. Finally, note that if $f$ is a contraction
in terms of \eqref{eq:contraction} then it is automatically a
quasi-contraction when choosing $E_0=E_\LL$, $E_1=E_\R$ and the
switching manifold $\Sigma$ instead of the points $F_i$, and using
\begin{equation*}
d(f(P),\Sigma)=\min_{x\in\Sigma}d\left(x,f(P)\right)
\end{equation*}
(similarly for $f(Q)$). Hence, conditions \emph{1.} and \emph{2.} in
Theorem~\ref{theo:maximin_quasi-contraction} replace the quasi-contracting
condition for the particular case of the type of piecewise-smooth maps
that we are considering here.\\
Note that, obviously, Theorem~\ref{theo:maximin_quasi-contraction}
also holds when considering $f$ a quasi-contraction for some
$F_0=F_1\in\Sigma$. However, this becomes much more resctrictive than
conditions \emph{1.} and \emph{2.}.\\

We now analyze Theorem~\ref{theo:maximin_quasi-contraction} keeping an
analogy with the results summarized in Section~\ref{sec:overview} for
the one-dimensional case.

We first note that the case in which only a periodic orbits exists is
the analogous to the orientation-preserving case for one-dimensional
maps. Although this is not stated in the result itself, this comes
from the proof provided in~\cite{GamGleTre88,GamTre88}. By
Proposition~\ref{pro:omega-ordered_minimax}, such periodic orbit must
have a symbolic sequence which belongs to the Farey tree of symbolic
sequences (see Figure~\ref{fig:farey_sequences}).\\
As it also occurs for the one-dimensional case, when the map is
non-orientable one finds the possibility of coexistence between two
periodic orbits. As before, this comes from the proof provided
in~\cite{GamGleTre88,GamTre88}. Let $\mathbf{\beta}$ and $\Delta$ be
the symbolic sequences of those periodic orbits.  From \emph{iii)}
$(\beta)^\infty \in W_{p,q}$ and $(\Delta)^\infty\in W_{p',q'}$ with
$p/q$ and $p'/q'$ Farey neighbours at some Farey sequence
$\mathcal{F}_n$ (see Definition~\ref{def:farey_sequence}). Assume that
$q'>q$, then they are neighbours at the Farey sequence of order $q'$
and $p/q$ is a Farey parent of $p'/q'$. Then, if $p/q>p'/q'$, by
Proposition~\ref{pro:concatenatetion} we know that
$\Delta=\alpha\beta$, where $\alpha$ is the Farey sequence of the
other Farey parent of $p'/q'$. Then, by going down through the Farey
tree of symbolic sequences we find two sequences, $\sigma$ and
$\gamma$ such that $\beta=\sigma^n\gamma$ and
$\Delta=\sigma^{n+1}\gamma$, for some $n\ge0$.\\
Then, by considering the iterates of the maps $f_\LL$ and $f_\R$ given
by the sequences $\sigma$ and $\gamma$ (see the example bellow), the
periodic orbits of the original map are collapsed to periodic orbits
of the form $\LL^{n+1}\R$ and $\LL^n\R$ (or $\LL\R^{n+1}$ and
$\LL\R^n$) for the composite map, which corresponds to the coexistence
of periodic orbits given at the period incrementing structure.
\begin{example}\label{exa:quasi-contract_coexis_inc}
Assume that we are in the situation described in iii) with two
periodic orbits with symbolic sequences in $W_{3,8}$ and $W_{2,5}$
coexisting. From Section~\ref{sec:dyna_orient_preserv} we know that,
with the notation from above, these symbolic sequences are
$\Delta=\left( \LL^2\R \right)^2\LL\R$ and $\beta=\LL^2\R\LL\R$. By
going down trough the tree, we get $\alpha=\sigma=\LL^2\R$ and
$\gamma=\LL\R$. Then, for two properly chosen sets $\tilde{E}_\LL$ and
$\tilde{E}_\R$ (see below for more details), the map
\begin{equation*}
\tilde{f}(x)=\left\{
\begin{aligned}
&\tilde{f}_\LL(x):=f_\LL^2\circ f_\R(x)&&\text{if }x\in \tilde{E}_\LL\\
&\tilde{f}_\R(x):=f_\LL\circ f_\R(x)&&\text{if }x\in \tilde{E}_\R,
\end{aligned}\right.
\end{equation*}
possesses two periodic orbits with symbolic sequences $\LL^2\R$ and
$\LL\R$, which corresponds to the predicted ones by the period
incrementing bifurcation structure.\\
Note that the sets $\tilde{E}_\LL$ and $\tilde{E}_\R$ need to be
properly found. This can be done by noting that the iterates by
$\tilde{f}_\LL$ or $\tilde{f}_\R$ provide the boundaries of the
domains of attraction between the points of the periodic orbits.
See~\cite{AvrSchGar10b} and \cite{AvrSchGar10} for explicit examples.
\end{example}

Finally, the third case for which $f$ does not have any periodic
orbit corresponds to quasi-periodic dynamics. In this case $f$
possesses an attracting Cantor set (see~\cite{GamTre88}).

\begin{remark}\label{rem:bif_stru_Rn}
Theorem~\ref{theo:maximin_quasi-contraction} does not provide
information about the bifurcation structures that may appear when
adding parameters to $f$. As it will be discussed in
Section~\ref{sec:conclusions_future}, similar structures to the period
adding may appear, although it is not guaranteed that all periodic
orbits in the tree may exist for some parameter.
\end{remark}

\subsection{Maximin approach for the one-dimensional case}\label{sec:maximin_1d}
In this section we present an alternative path to prove the period
adding structure for one-dimensional piecewise-smooth maps. More
precisely we show that one can use
Theorem~\ref{theo:maximin_quasi-contraction} after Gambaudo et al. to
identify sequences in the Farey tree of symbolic sequences with the
itineraries of one-dimensional orientation preserving circle maps. By
using the concept of maximin sequences instead of $p,q$-ordered
sequences (see Definition~\ref{def:well-ordered_symbolic-sequence}),
this approach provides stronger results than the one presented in
Section~\ref{sec:dyna_orient_preserv}, as they give more information
about the symbolic sequences. In particular, given the rotation
number, $p/q$, the maximin property permits to obtain the proper
symbolic itinerary of a periodic orbit without constructing the whole
Farey tree of symbolic sequences up to level $q$. Instead, one needs
to find in $W_{p,q}$ which is the symbolic sequence that verifies the
maximin/minimax condition (see Definition~\ref{def:maximin}).\\
However, one of the counterparts to these advantages consists of the
requirement of contraction, whereas, as remarked in
Section~\ref{sec:remarks_expansive}, many of the results
through circle maps (shown in
Sections~\ref{sec:properties_circle_maps}
and~\ref{sec:dyna_orient_preserv}) also hold for expansive maps.
Moreover, the proof of Theorem~\ref{theo:maximin_quasi-contraction},
specially for $1/2\le k\le 1$ (see~\cite{GamTre88}), is significantly
more difficult than the ones of in
Section~\ref{sec:dyna_orient_preserv}.

We now specify the alternative path of this alternative proof, which
is as follows. From Proposition~\ref{pro:po_rational_rho} we get that,
under conditions~\condsC, if the rotation number is rational, $p/q$,
then a periodic orbit exists. When adding the assumption of
contraction, Theorem~\ref{theo:maximin_quasi-contraction} holds and,
due to uniqueness of the rotation number, tells us that a unique
periodic orbit must exist; moreover, its symbolic itinerary is
maximin. Then, from Proposition~\ref{pro:omega-ordered_minimax}, we
get that its symbolic sequence belongs to the Farey tree of symbolic
sequences.  Once we get that the symbolic itineraries of orientation
preserving circle maps belong to the Farey tree of symbolic sequences,
one proceeds as in Section~\ref{sec:dyna_orient_preserv} to study the
bifurcation structure when the parameters $c$ or $\lambda$ are varied
in order to prove {\em i)} in
Theorem~\ref{theo:adding_incrementing}.

\section{Applications}\label{sec:examples}
In Section~\ref{sec:examples_1d} we present examples for which
Theorem~\ref{theo:adding_incrementing} provides a full description of
the dynamics that one can find.\\
In the first example (Section~\ref{sec:co-dimension_two}), we will
focus on the codimension-two bifurcation given by the simultaneous
collision of two periodic orbits of one-dimensional discontinuous map
with the boundary. By applying Theorem~\ref{theo:adding_incrementing}
we generalize in this example the results obtained
in~\cite{GamLanTre84,GamGleTre88} by providing a precise description
of the bifurcation scenario.\\
In the second example (Section~\ref{sec:relay}) we deal with a system
of control theory. We consider a first order system with a stable
equilibrium point. In order to stabilize this system at another value,
we perform a non-smooth control action based on relays, which is well
extended technique in engineering control, and known as {\em
sliding-mode} control.  Theorem~\ref{theo:adding_incrementing} holds
and describes the dynamics of this system.\\
In Section~\ref{sec:IF} we consider a generic {\em integrate-and-fire}
system submitted to a periodic forcing. Such a system consists of a
hybrid model widely used in neuroscience. We show how
Theorem~\ref{theo:adding_incrementing} can be applied to describe, not
only its dynamics, but relevant biological properties as the {\em
firing-rate} under parameter variation. There exist in the literature
other examples in neuron models from which one derives a
piecewise-smooth discontinuous map such that
Theorem~\ref{theo:adding_incrementing} {\em ii)} holds; see for
example~\cite{JimMihBroNieRub13,TouBre09}.

Section~\ref{sec:examples_highdim} is devoted to illustrate the
results revisited in Section~\ref{sec:maximin_approach} for
higher-dimensional piecewise-smooth maps.\\
We first consider a higher-order system controlled with relays
(Section~\ref{sec:high-order_relays}). This represents an extension of
the example shown in Section~\ref{sec:relay}.\\
In Section~\ref{sec:boost_zad} we consider a ZAD-controlled DC-DC
boost converter. We show how symbolic dynamics given by the maximin
periodic orbits describe how the saturation of a Pulse Width
Modulation process is distributed along periodic cycles.
\subsection{One-dimensional examples}\label{sec:examples_1d}
\subsubsection{Codimension-two border collision
bifurcations}\label{sec:co-dimension_two} Our first example consists
of a generic dynamical system given by a piecewise-smooth map for
which, under variation of two parameters, two periodic orbits
transversally collide with the boundary. In this two-dimensional
parameter space, this is given by the crossing of two border collision
bifurcation curves, and is hence a codimension-two bifurcation point.
Below we will show that the map~\eqref{eq:normal_form} is a normal
form for such a bifurcation and hence the dynamics are given by
Theorem~\ref{theo:adding_incrementing}.\\
This bifurcation
codimension-two bifurcation was reported as {\em gluing bifurcation}
(\cite{GamLanTre84,GamGleTre88}) when studying {\em figure of eight}
or {\em butterfly} homoclinic bifurcations for three-dimensional
flows. Assuming the existence of a strong stable direction it is
possible to reduce the dynamics near a homoclinic bifurcation to a
two-dimensional system. Then, one considers a first-return Poincar\'e
map, which turns out to be a one-dimensional discontinuous map of the
form given in Equation~\eqref{eq:normal_form}
(see~\cite{ArnCouTre81,Spa82} for more details). Depending on the
orientability of the stable and unstable manifolds, this map exhibits
different configurations regarding its monotonicity near the
discontinuity (see~\cite{GhrHol94,Hom96}), leading to the cases {\em
i)-iv)} of Theorem~\ref{the:Farey_properties} and
Remark~\ref{rem:conds_iii-iv}.  

Later on, in the context of simulations of piecewise-smooth dynamical
systems, this co-dimension two bifurcation was rediscovered and called
{\em big bang} bifurcation (\cite{AvrSch06}) as, depending on the
monotonicity properties of the map, it may involve the emergence of an
infinite number of bifurcation curves from the codimension-two
bifurcation point at the parameter space. Such points organize the
dynamics around them and have been highly reported in simulations of
piecewise-smooth
maps~\cite{AvrSch04,AvrSch05,AvrSchBan06,AvrSchBan07,FutAvrSch12,AvrSchSch10a,AvrSus13}.\\
Also from the piecewise-smooth perspective, the bifurcation scenarios
that appear around such a codimension-two bifurcation point have been
recently reconsidered in~\cite{GarAvrSus14}. There the authors
distinguish between the cases {\em i)} and {\em ii)} of
Theorem~\ref{theo:adding_incrementing}. For the period adding
structure (case {\em i)}) the authors use renormalization arguments to
demonstrate the existence of periodic orbits nested between regions in
the parameter space. However, the proof is not complete, as it is not
shown that this occurs only for regions containing periodic orbits
with Farey neighbours rotation numbers, which is the main result.
Regarding the period incrementing (case {\em ii)}), the authors
of~\cite{GarAvrSus14} repeat the proof provided in~\cite{AvrGraSch11}
and in Section~\ref{sec:incrementing_proof}.\\

We now start with the example. We will show that, under non-degeneracy
conditions, after collapsing the colliding periodic orbits to fixed
points and performing a reparametrization, such a bifurcation can be
reduced to the study of a map of the form of
Equation~\eqref{eq:normal_form} under variation of the parameters
$\mu_\LL$ and $\mu_\R$.

Assume that we have a dynamical system given by
\begin{equation}
x_{n+1}=f(x_n),
\label{eq:codim2_system}
\end{equation}
with
\begin{equation}
f(x)=\left\{
\begin{aligned}
f_\LL(x;a,b)&&\text{if }x<0\\
f_\R(x;a,b)&&\text{if }x>0,
\end{aligned}
\right.
\label{eq:codimtwo_map}
\end{equation}
where $a,b\in\RR$ are two parameters.\\
Assume that system~\eqref{eq:codim2_system} possesses two periodic
orbits $(x_0,\dots,x_{n-1})$ and\linebreak $(y_0,\dots,y_{m-1})$,
$x_i,y_i\in\RR$, satisfying
\begin{align*}
f(x_i)&=x_{i+1},\;0\le i<n-1\\
f(x_{n-1})&=x_0
\end{align*}
and
\begin{align*}
f(y_i)&=y_{i+1},\;0\le i<m-1\\
f(y_{m-1})&=y_0.
\end{align*}
Note that the periodic orbits are indexed by dynamical and not
spatial order, which implies that we do not necessarily have $x_i<x_{i+1}$
and $y_i<y_{i+1}$.\\
Assume that these periodic orbits undergo right and left border
collision bifurcations at two curves, $b=\xi_x(a)$ and $b=\xi_y(a)$ in
the parameter space $(a, b)$, respectively; i.e, there exist unique
integers $k_x\in\left\{ 0,\dots, n-1 \right\}$ and $k_y\in \left\{
0,\dots, m-1 \right\}$ such that, for any $a\in \RR$,
\begin{align*}
\lim_{b\to \xi_x(a)^+}x_{k_x}&=0^-\\
\lim_{b\to \xi_y(a)^-}y_{k_y}&=0^+.
\end{align*}
Assume that these two curves intersect transversally at some point
$(a^*,b^*)$.  In order to write the map~\eqref{eq:codimtwo_map} into
the form of Equation~\eqref{eq:normal_form} for parameter values near
$(a^*,b^*)$ we first need to consider a higher iterate of $f$ so that
these periodic orbits become fixed points.  To this end, let
$\x=(\x_1,\dots,\x_n)\in \left\{ \LL,\R \right\}^n$ and
$\y=(\y_1,\dots,\y_m)\in\left\{ \LL,\R \right\}^m$ be
symbolic sequences such that
\begin{align*}
I_f(x_{k_x})&=\x^\infty\\
I_f(y_{k_y})&=\y^\infty.
\end{align*}
Then, consider the iterates
\begin{align*}
\tf_\LL&=f_{\x_1}\circ f_{\x_2}\circ\dots\circ f_{\x_n}(x)\\
\tf_\R&=f_{\y_1}\circ f_{\y_2}\circ\dots\circ f_{\y_m}(x),
\end{align*}
and the map
\begin{equation*}
\tf(x;a,b) =\left\{
\begin{aligned}
\tf_\LL(x;a,b)&&\text{if }x<0\\
\tf_\R(x;a,b)&&\text{if }x>0.
\end{aligned}\right.
\end{equation*}
Note that one can always find proper shifts of the sequences $\x$ and
$\y$ such that the domains of $\tf_\LL$ and $\tf_\R$ become $x<0$ and
$x>0$, respectively. Indeed, this occurs when the sequences $\x$ and
$\y$ start with the points $x_{k_x}$ and $y_{k_y}$, respectively.
See~\cite{AvrSchGar10b} and \cite{AvrSchGar10} for explicit
examples.\\
These compositions collapse the periodic orbits of the maps $f_\LL$ and
$f_\R$ to fixed points of $\tf_\LL$ and $\tf_\R$, respectively. There
exists then a neighbourhood $(a^*,b^*)\in\U\subset \RR^2$ which is
split in four subsets for which $\tf$ has two fixed points, a unique
positive fixed point, a unique negative fixed point and no fixed
points. These fixed points undergo simultaneous border collision
bifurcations at $(a,b)=(a^*,b^*)$ in such a way that the map $\tf$
becomes continuous at $x=0$ for these parameter values.

We now perform a reparametrization along the curves $\xi_x$ and
$\xi_y$. The non-degeneracy conditions given by their transversal
intersection will give us that, in first order, the new parameters
will be equivalent to the offset parameters $\mu_\LL$ and $\mu_\R$ of
the general map given in Equation~\eqref{eq:normal_form}.\\
This reparametrization is given  by
\begin{equation*}
\begin{array}[]{cccc}
\phi:&\RR^2&\longrightarrow&\RR^2\\
&(\mu_\LL,\mu_\R)&\longmapsto&\left(
a^*+\mu_\LL+\mu_\R,\xi_x(a^*+\mu_\LL)+\xi_y(a^*+\mu_\R) \right),
\end{array}
\end{equation*}
which maps the axis $\mu_\LL=0$ and $\mu_\R=0$ to the bifurcation
curves $(a,\xi_x(a))$ and $(a,\xi_y(a))$, respectively. Expanding in
powers of $x$, $\mu_\LL$ and $\mu_\R$ around
$(x,\mu_\LL,\mu_\R)=(0,0,0)$ we get
\begin{align*}
\tf_\LL(x;\phi(\mu_\LL,\mu_\R))&=\tfL(x;\phi(0,0))\\
&+D_{\mu_\LL,\mu_\R}\tfL(x;\phi(\mu_\LL,\mu_\R))_{
\left|
\tiny
\begin{array}[]{l}
x=0\\\mu_\LL=0\\\mu_\R=0
\end{array}\right.
}
\cdot
\left(
\begin{array}[]{c}
\mu_\LL\\\mu_\R
\end{array}
 \right)\\
&+O(\mu_\LL x,\mu_\R x,,\mu_\LL\mu_\R,x^2,\mu_\LL^2,\mu_\R^2)
\end{align*}

\begin{align*}
D_{\mu_\LL,\mu_\R}\left(\tfL(x;\phi(\mu_\LL,\mu_\R))\right)
_{
\left|
\tiny
\begin{array}[]{l}
x=0\\\mu_\LL=0\\\mu_\R=0
\end{array}\right.
}
=&D_{\mu_\LL,\mu_\R}\left(
f(0;\phi(\mu_\LL,\mu_\R))
\right)_{\left|\tiny
\begin{array}[]{c}
\mu_\LL=0\\\mu_\R=0
\end{array}\right.
}\\
=&\Bigg( \frac{\partial}{\partial
\mu_\LL}\tfL(0;\phi(\mu_\LL,\mu_\R)),\\
&\frac{\partial}{\partial \mu_\R}\tfL(0;\phi(\mu_\LL,\mu_\R)) \Bigg)
_{\left|\tiny
\begin{array}[]{c}
\mu_\LL=0\\\mu_\R=0
\end{array}\right.
}
\end{align*}

\begin{align*}
\frac{\partial}{\partial\mu_\LL}\tf(0;\phi(\mu_\LL,\mu_\R))
_{\left|\tiny
\begin{array}[]{c}
\mu_\LL=0\\\mu_\R=0
\end{array}\right.
}
=&\frac{\partial}{\partial
\mu_\LL}\tfL(0;\phi(\mu_\LL,0))_{\mu_\LL=0}\cdot\mu_\LL\\
+&O(\mu_\LL^2,\mu_\R^2,\mu_\LL\mu_\R)
\end{align*}

\begin{align*}
\frac{\partial}{\partial\mu_\R}\tf(0;\phi(\mu_\LL,\mu_\R))
_{\left|\tiny
\begin{array}[]{c}
\mu_\LL=0\\\mu_\R=0
\end{array}\right.
}
&=\left(\frac{\partial}{\partial
\mu_\R}\underbrace{\tfL(0;\phi(0,\mu_\R))}_{=0}\right)_{|\mu_\R=0}\cdot \mu_\R\\
&+O(\mu_\LL^2,\mu_\R^2,\mu_\LL\mu_\R).
\end{align*}
Proceeding similarly for $\tfR$ we get that there exist some constants
$\delta_\LL,\delta_\R\in\RR$ such that the map $\tf$ can be written, close to
$(x,\mu_\LL,\mu_\R)=(0,0,0)$, as
\begin{equation}
\tf(x;\mu_\LL,\mu_\R)=\left\{
\begin{aligned}
\delta_\LL \mu_\LL +\tfL(x;a^*,b^*)+h.o.t.\\
\delta_\R \mu_\R +\tfR(x;a^*,b^*)+h.o.t.\\
\end{aligned}
\right.
\label{eq:new_map}
\end{equation}
where $h.o.t.=O(\mu_\LL x,\mu_\R x,\mu_\LL\mu_\R,\mu_\LL^2,\mu_\R^2)$.
Hence, taking $\delta_\LL\mu_\LL$ and $\delta_\R\mu_\R$ as new
parameters, the first order terms of the map $\tf$ given in
Equation~\eqref{eq:new_map} are of the form as in
map~\eqref{eq:normal_form} and satisfy condition {\em h.1}.

Finally, if the two periodic orbits $(x_0,\dots, x_{n-1})$ and
$(y_0,\dots,y_{m-1})$ are attracting, the maps $\tfL$ and $\tfR$ become
contracting near $x=0$. This implies that conditions~\condsH{} are
satisfied and Theorem~\ref{theo:adding_incrementing} can be applied to
system~\eqref{eq:new_map} for $\delta_\LL\mu_LL$ and $\delta_\R\mu_\R$
small. Hence, depending on the sign of $(\tf_\LL)'(0^-;a^*,b^*)$ and
$(\tf_\R)'(0^+;a^*,b^*)$ we will get situations {\em i)}-{\em iv)} of
Theorem~\ref{theo:adding_incrementing} and
Remark~\ref{rem:conds_iii-iv}.\\
The periodic orbits of original map $f$ given in
Equation~\eqref{eq:codimtwo_map} are obained from the ones of $\tf$
given by Theorem~\ref{theo:adding_incrementing} by replacing each
point $x_i$ by
\begin{align*}
x_i&\longrightarrow
(x_i,f_{\x_1}(x_i),f_{\x_2}\circ
f_{\x_1}(x_i),\dots,f_{\x_n}\circ\cdots\circ
f_{\x_1}(x_i))&&\text{if }x_i<0\\
x_i&\longrightarrow
(x_i,f_{\y_1}(x_i),f_{\y_2}\circ
f_{\y_1}(x_i),\dots,f_{\y_m}\circ\cdots\circ
f_{\y_1}(x_i))&&\text{if }x_i>0.
\end{align*}
The corresponding symbolic for each periodic orbit of the map
$f$ sequence will be obtained by replacing
\begin{align*}
\LL&\longrightarrow \x\\
\R&\longrightarrow\y.
\end{align*}
\subsubsection{First order sliding-mode controlled system with relays}\label{sec:relay}
In our next example we consider a first order one-dimensional
system\footnote{In this section we abuse notation and refer with
$f$ to a field rather than a map.},
\begin{equation}
\dot{y}=f(y),
\label{eq:open_loop}
\end{equation}
such that $f$ is monotonically decreasing
and has a simple zero at the origin, $f(0)=0$, which is a stable equilibrium point
of system~\eqref{eq:open_loop}.  We generalize the results presented
in~\cite{FosGra11} for a linear system and in~\cite{FosGra13} for
second order linear system.\\
We wish to stabilize system~\eqref{eq:open_loop} at new equilibrium
point $y^*$ by performing a {\em control} action $u(t,y)$:
\begin{equation}
\dot{y}=f(y)+u(t,y).
\label{eq:controlled_system}
\end{equation}
Obviously, as $f$ is monotonically decreasing, one can always choose
$u$ constant ({\em open loop} control) such that $f(y^*)+u=0$.
However, if $u$ is chosen to depend on the current value of the
system, $y$, ({\em closed loop} control) then such control action will
be robust to small perturbations and inaccuracies on the modeling. The
type of control that we will consider is based on {\em sliding-mode}.
This consists on performing some action when $y>y^*$ and another one
whenever otherwise so that solutions are ``pushed'' towards $y^*$. In
a hardware setting this is implemented by means of relays.  The term
{\em sliding} comes from to the generalization to higher dimensions,
$y\in\RR^n$. In this case, one defines  a surface $\sigma(y)=0$ and
the desired behaviour consists of sliding motion along this surface.
To achieve this,  one acts similarly depending on whether
$\sigma(y)>0$ or $\sigma(y)<0$.\\
In addition, we will also assume that $u$ depends on $t$, as we will
perform a periodic sampling in order to discretize the system and use
a digital control.

All together, we will show that the dynamics of such system are given
by a  piecewise-smooth map of the form~\eqref{eq:normal_form}
satisfying~\condsHp{}. Hence, Theorem~\ref{theo:adding_incrementing}
can be applied to accurately describe the dynamics of
system~\eqref{eq:controlled_system}.\\

We now describe how the control action $u$ is defined and how we
derive the piecewise-smooth map.\\
The control is schematically illustrated in
Figure~\ref{fig:system_with_relay}.
\begin{figure}
\begin{center}
\includegraphics[width=0.8\textwidth]{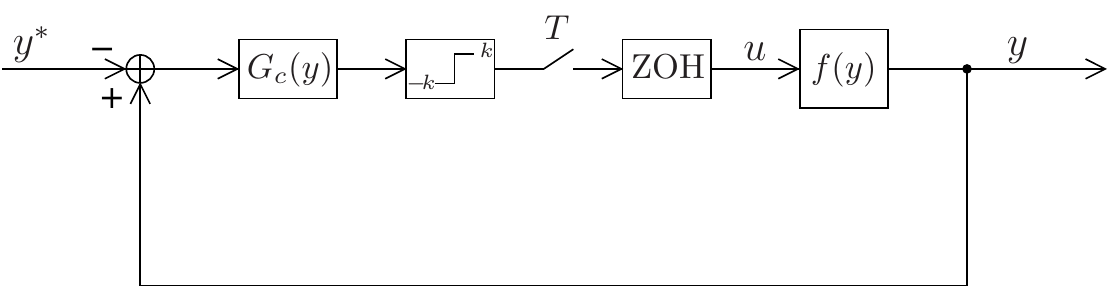}
\end{center}
\caption{Schematic representation of the control strategy based on
sliding and relays.}
\label{fig:system_with_relay}
\end{figure}
As one can see, this is a closed loop system, whose input is the
desired new equilibrium $y^*$. The error $y-y^*$ becomes the input of
a controller represented by the module $G_c$.  In this work we will
assume a proportional control only. Indeed, it will not be a loss of
generality if we set its gain equal to $1$ (it does not perform any
action), provided that the next module, the relay, will also have some
tunable gain.\\
The output of the controller is pushed to a relay of gain $k$. This
means that, depending on the sign of its input, the relay will set its
output to $k$ or $-k$ depending on whether $y-y^*>0$ or $y-y^*$,
respectively.\\
The parameter $T$ is the sampling period, which digitalizes the
system. At every time $T$ the switch closes and provides a new sample
to the $ZOH$ (Zero Order Holder) block. This block holds the received
value until it is changed at the next sampling. The input of the
module $f(y)$ becomes hence a piecewise-constant function:
\begin{equation}
u(t)=\left\{
\begin{aligned}
&-k&&\text{if }y(nT)-y^*<0\\
&k&&\text{if }y(nT)-y^*>0,
\end{aligned}
\right.
\label{eq:u}
\end{equation}
for $t\in[nT,(n+1)T)$.\\
In order to obtain sliding motion we require $f(y^*)+k<0$ and
$f(y^*)-k>0$; that is, we require the
field~\eqref{eq:controlled_system} to point towards $y^*$. This
condition can be generalized to higher dimensions in terms of the
so-called {\em equivalent control} and Lie derivatives
(see~\cite{Utk77} for details). In our case, this implies that
necessary $k<0$, although note that this is not a sufficient
condition.\\
The fact that the variable $y$ is periodically sampled suggests that
the dynamics of system~\eqref{eq:controlled_system} with $u$ given by
Equation~\eqref{eq:u} can be better represented by a map
\begin{equation*}
y_{n+1}=P(y_n),
\end{equation*}
where $y_k=y(nT)$ is the value of $y$ at the $n$th sample, and $P$ is
the stroboscopic (time-$T$ return map) of
system~\eqref{eq:controlled_system},\eqref{eq:u}. It becomes the
piecewise-smooth map
\begin{equation*}
P(y)=\left\{
\begin{aligned}
P_\LL(y)&:=\varphi(T;y;-k)&&\text{if }y<y^*\\
P_\R(y)&:=\varphi(T;y;k)&&\text{if }y>y^*,
\end{aligned}
\right.
\end{equation*}
where $\varphi(t;y;k)$, satisfying
$\varphi(0;y;k)=y$, is the flow associated with the
system $\dot{y}=f(y)+ k$. Note that $P(y)$ is a
smooth map as regular as the flow $\varphi$ if $y>y^*$ or $y<y^*$, and
it is discontinuous at $y=y^*$ if $k\ne 0$.\\
Due to the monotonicity of $f$ the maps $P_\LL$ and $P_\R$ are
contracting for $y\in\RR$.  Moreover, they are both increasing (they
preserve orientation), as they are given by integration of an
autonomous differential equation. Finally, if $k<0$ is such that the
sliding condition is satisfied, we get that $P_\LL(y^*)>y^*$ and
$P_\R(y^*)<y^*$, and $P$ undergoes a negative gap at $y=y^*$. Due to
the monotonicity of $f$, this will occur for any $k<0$ such that $|k|$
is large enough.
Let us now fix $k<0$ and we study the dynamics of the system under
variation of the parameter $y^*$, the desired output. If $|k|$ is
large enough such that the sliding condition is satisfied, then the
attracting fixed points $y_\LL$ and $y_\R$  of $P_\LL$ and $P_\R$,
respectively, are virtual:
\begin{equation*}
y_\R<y^*<y_\LL.
\end{equation*}
Therefore, when $y^*$ is varied, these fixed points may undergo border
collision bifurcations when
\begin{align*}
y_\R(k_\R)&=y^*\\
y_\LL(k_\LL)&=y^*.
\end{align*}
After applying the change of variables $z\longmapsto y-y^*$, the
one-parameter family of maps $\tilde{P}_{y^*}(z):=P(z+y^*)-y^*$ is in
the class of maps given by applying the
reparametrization~\eqref{eq:1d_scann_curve} to the piecewise-smooth
map~\eqref{eq:normal_form} satisfying {\em i)} of
Theorem~\ref{theo:adding_incrementing}. Hence, when varying $y^*$ from
$y_\R$ to $y_\LL$, one observes periodic orbits following  the
period adding structure.\\
As this occurs for any $k<0$, if $y^{*}$ is close enough to the
equilibrium point of system~\eqref{eq:open_loop}, $y=0$, the origin of
the parameter space $(y^*, k)$ represents a co-dimension two
bifurcation point.  This is shown in Figure~\ref{fig:bbb_yc-b} for the
linear system
\begin{equation*}
f(y)=ay
\end{equation*}
with $a<0$. The bifurcation curves shown in
Figure~\ref{fig:bbb_region} are straight lines due to the linearity of
the system.

\begin{figure}[h]\centering
\begin{picture}(1,0.5)
\put(0,0.4){
\subfigure[\label{fig:bbb_region}]{\includegraphics[width=0.35\textwidth,angle=-90]{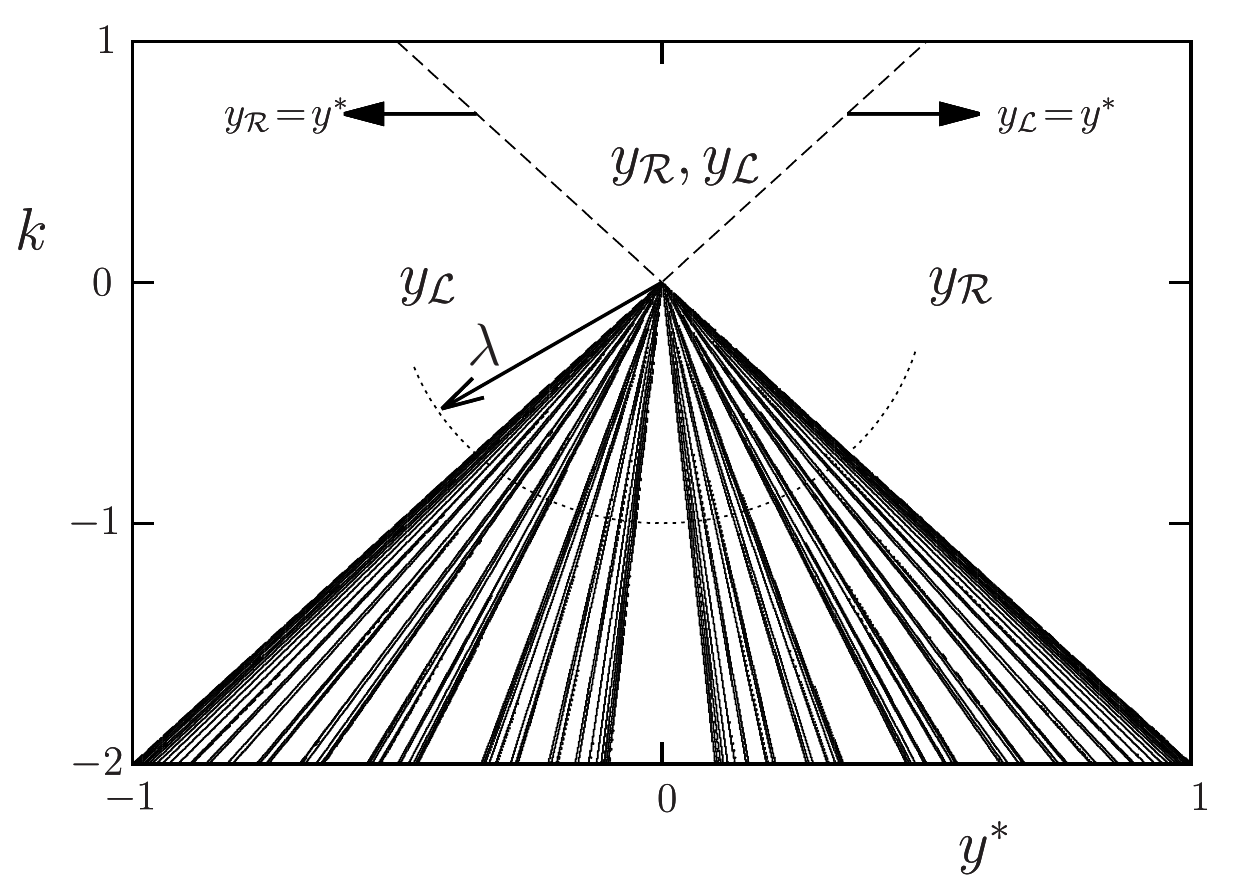}}
}
\put(0.5,0.4){
\subfigure[\label{fig:bbb_periods}]{\includegraphics[width=0.35\textwidth,angle=-90]{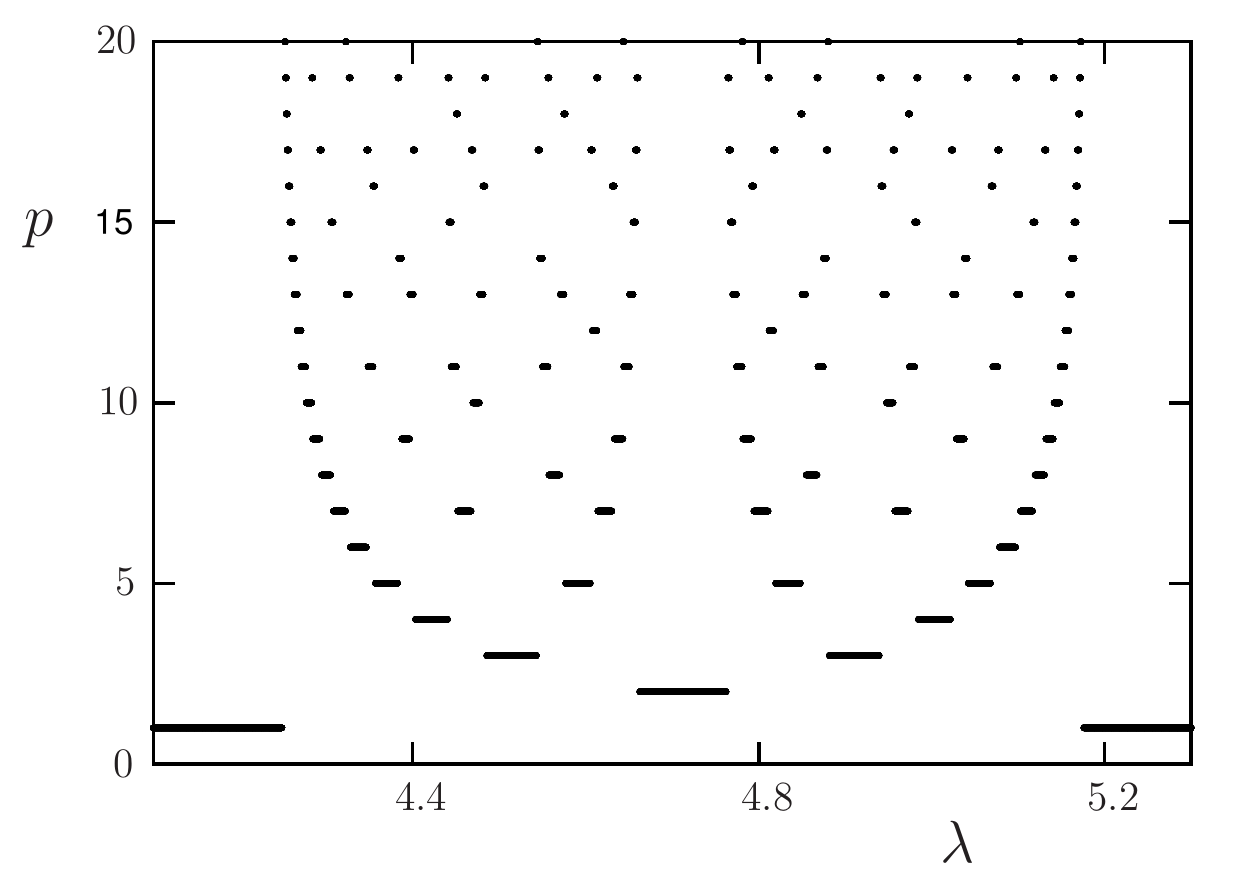}}
}
\end{picture}
\caption{(a) Co-dimension two bifurcation point of the adding type for
a the system $f(y)=-0.2y$ and $T=0.1$.  The fixed points $y_i$ are
labeled in the regions where they are feasible, and, as dashed lines,
the curves where they undergo border collision bifurcation.  In (b) we
show the periods $p$ of the periodic orbits found along the pointed
curve in (a), which is parametrized by the angle $\lambda$.}
\label{fig:bbb_yc-b}
\end{figure}%
In the parameter region where one finds the period adding
structure is the one for which the system performs the proper
tracking.  Its asymptotic dynamics consists of periodic orbits
bouncing around the desired equilibrium point and whose amplitude can
be made arbitrarily small by choosing $T$ small enough. For any chosen
$y^*$, there exists a maximal value of $k$ for which the systems
performs tracking. However, although for larger values of $k$ the
system oscillates between the desired valued $y^*$, not only the
amplitude of the periodic motion grows, but also its symbolic
dynamics, which can be modified by properly tunning $k$ such that a
certain behaviour is achieved (or avoided). For example, one may wish to minimize
the time for which the system's output exceeds $y^*$, and hence one may
choose symbolic sequences with low rotation number.

\subsubsection{Hybrid systems in biology}\label{sec:IF}
In this example we consider a generalization of an integrate-and-fire
system, widely used neuron model. It consists of a hybrid system given
by\footnote{As in Section~\ref{sec:relay}, in this section we
abuse notation and refer with $f$ to a field rather than a map.}
\begin{equation}
\dot{x}=f(x),
\label{eq:IF}
\end{equation}
$f\in C^\infty(\RR)$, submitted to the reset condition
\begin{equation}
x=\theta \longrightarrow x=0;
\label{eq:reset}
\end{equation}
that is, whenever the variable $x$ reaches a certain threshold
$\theta$, it is reset to a certain value, which we assume to be $x=0$.
This emulates  {\em spike} of a neuron (action potential).\\
Typically, the dynamics of system~\eqref{eq:IF}-\eqref{eq:reset}
are studied by means of the so-called {\em firing} map
(\cite{KeeHopRin81}), which is a Poincar\'e map onto the threshold
$x=\theta$. However, when system~\eqref{eq:IF} is periodically forced,
this map is not optimal for obtaining general results, as
one has to explicitly compute the time when spikes occur.\\
In this example we present the results reported in~\cite{GraKruCle14},
where it was shown that, when a periodic forcing is considered, it is
more convenient to study the system by means of the stroboscopic
map (time-$T$ return map, being $T$ the period of the forcing).\\
Following~\cite{GraKruCle14}, in this example we periodically force
system~\eqref{eq:IF} by means of a square wave function, which models a
pulsatile stimulus of a neuron.  That is, we consider the system
\begin{equation}
\dot{x}=f(x)+I(t),\,x\in\RR
\label{eq:general_system}
\end{equation}
with $f(x)\in C^\infty(\RR)$ and $I(t)$ the $T$-periodic function
\begin{equation}
I(t)=\left\{
\begin{aligned}
&A&&\text{if }t\in\left(nT,nT+dT\right]\\
&0&&\text{if }t\in(nT+dT,(n+1)T],\\
\end{aligned}\right.
\label{eq:forcing}
\end{equation}
where $A>0$ and $0\le d\le 1$ is the so-called {\em duty} cycle. We are
interested in the bifurcation structure in the parameter space given
by the amplitude of the pulse, $A$, and its duty cycle, $d$. 
We refer to~\cite{GraKru15} for the study of the bifurcation
structures in this parameter space under frequency variation of the
input, $1/T$.

Let us assume that system~\eqref{eq:IF} satisfies
\begin{enumerate}[{A}.1)]
\item it possesses an attracting equilibrium point
\begin{equation*}
0<\bx<\theta,
\end{equation*}
\item $f(x)$ is a monotonic decreasing function in $[0,\theta]$:
\begin{equation*}
f'(x)<0,\;x\in[0,\theta].
\end{equation*}
\end{enumerate}
Note that, conditions \condsA{} guarantee that spikes can only occur
when the pulse is active ($I=A$) with $A$ large enough; otherwise,
when the pulse is off, trajectories are attracted by the equilibrium
point $\bx$.

As mentioned above, by contrast to typical approaches by means of the
firing map, we use the stroboscopic map:
\begin{equation}
\begin{array}{cccc}
\s:&[0,\theta)&\longrightarrow &[0,\theta)\\
&x_0&\longmapsto&\phi(T;x_0),
\end{array}
\label{eq:strobo_map}
\end{equation}
where $\phi(t;x_0)$ is the solution of system~\SYSTEMWR{} with initial
condition $\phi(0;x_0)=x_0$. Note that the flow $\phi$ is well
defined, although it is discontinuous when spikes occur. The
stroboscopic map $\s$ is a piecewise-smooth map. To see this, we
define the sets
\begin{equation}
\begin{aligned}
S_n=\Big\{ x_0\in [0,\theta)\;\text{s.t. }&\phi(t;x_0)\text{ reaches
 the threshold }\left\{ x=\theta\right\}\\
&n\text{ times for }0\le t\le T \Big\},\, n\ge0.
\end{aligned}
\label{eq:sets}
\end{equation}
When restricted to $S_n$,  the map $\s$ becomes a concatenation of
maps consisting of integrating the system $\dot{x}=f(x)+A$ and
performing resets, for $t\in [0,dT]$, and then integrating the system
$\dot{x}=f(x)$ for $t\in (dT,T]$. Recall that, as mentioned above, no
resets can occur for $t\in(dT,T)$.  Hence, $\s$ is smooth in the
interior of $S_n$, as it is given by a certain composition of smooth
maps; however, $\s$ is discontinuous at the boundaries of $S_n$. This
comes from the following fact. Let us define $\Sigma_n\in S_n$ as the
initial condition that leads to a trajectory exhibiting its $n$th
spike precisely when the pulse is deactivated:
\begin{equation*}
\phi(dT;\Sigma_n)=\theta.
\end{equation*}
Note that, as the flow $\phi(t;x_0)$ of system~\SYSTEMWR{} is well
defined if it exists, such initial condition is unique and can only
exist for certain $n$.
Notice also that, if there exists $\Sigma_n\in(0,\theta)$, then points
above $\Sigma_n$ exhibit $n$ spikes, whereas points below exhibit
$n-1$ spikes (see Figure~\ref{fig:boundary}).\\
As a consequence, the map $\s$ undergoes at most one discontinuity, at
$x=\Sigma_n$, and points to its left and right exhibit $n-1$ and $n$
spikes, respectively, for $t\in[0,dT]$. As shown
in~\cite{GraKruCle14}, for any $n$ the discontinuity $\Sigma_n$
satisfies
\begin{equation}
\frac{d\Sigma_n}{dA}<0,
\label{eq:A_monotonically_dec}
\end{equation}
and hence it decreases montonically with $A$. By increasing this
parameter, the discontinuity $\Sigma_n$ collides with $x=0$,
disappears and a new discontinuity, $\Sigma_{n+1}$, appears at
$x=\theta$. At this moment, the system changes from exhibiting $n-1$
and $n$ spikes to exhibiting $n$ and $n+1$ (see~\cite{GraKruCle14} for
more details).

If we compute the lateral images of $\s$ at these values are given by
\begin{align*}
\s(\Sigma_n^-)&=\varphi(T-dT;\theta;0)\\
\s(\Sigma_n^+)&=\varphi(T-dT;0;0),
\end{align*}
where $\varphi(t;x_0;A)$ is the flow associated with the system
$\dot{x}=f(x)+A$, with initial condition $\varphi(0;x_0;A)=x_0$.
\begin{figure}
\begin{center}
\begin{picture}(1,0.5)
\put(0,0.35){
\subfigure[]
{\includegraphics[angle=-90,width=0.5\textwidth]{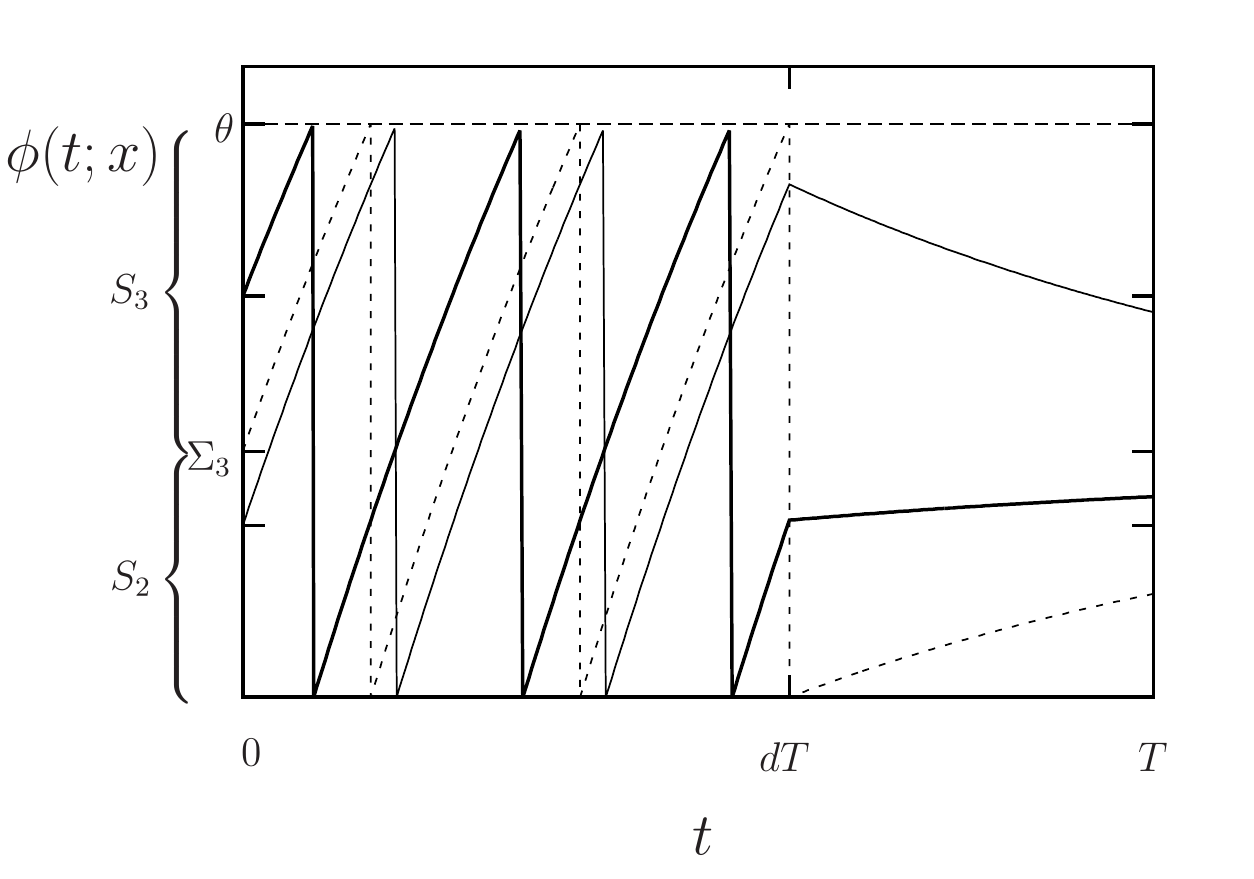}}
}
\put(0.5,0.35){
\subfigure[\label{fig:boundary_map}]
{\includegraphics[angle=-90,width=0.5\textwidth]{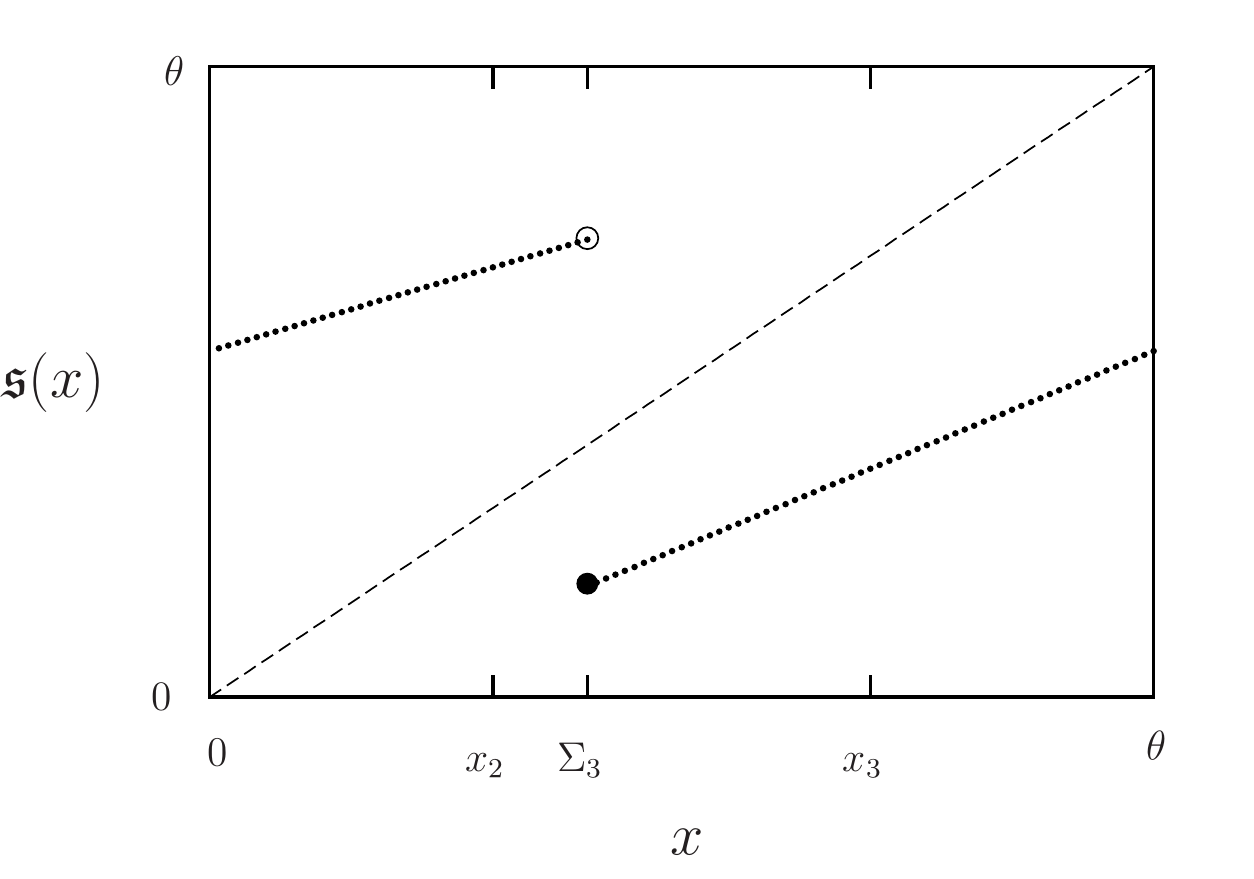}}
}
\end{picture}
\end{center}
\caption{In (a) the trajectories of systems~\SYSTEMWR{}. Dashed line: trajectory
with $\Sigma_3$ as initial condition. Thick line: trajectory with $x_3>\Sigma_3$
as initial condition, which spikes $3$ times. Normal line: trajectory with
$x_2<\Sigma_3$ as initial condition, which spikes $2$ times. In (b) the
stroboscopic map, with a discontinuity at $x=\Sigma_3$.}
\label{fig:boundary}
\end{figure}
Note that $\s(\Sigma_n^-)>\s(\Sigma_n^+)$ and hence $\s$ undergoes a
negative gap at $x=\Sigma_n$ (see Figure~\ref{fig:boundary}). Clearly,
as the reset action and the integration of the flow provide orientation
preserving maps, $\s$ is an increasing map.\\
The stroboscopic map $\s$ is always contracting in $[0,\Sigma_n)$.
However, it is contractive in $[\Sigma_n,\theta)$ only if $n$ is large
enough (\cite{GraKruCle14} Lemma~3.4). 

One can see (\cite{GraKruCle14}) that, for any $n\ge0$ and
$d\in(0,1)$, there exists a range of values of $A$,
$(A_n^\R,A_n^\LL)$, for which $\s$ possesses a unique fixed point,
$\bx_n\in S_n\subset(0,\theta)$.  These fixed points undergo border
collision bifurcations at $A=A_n^\R$ and $A=A_n^\LL$ when colliding
with the boundaries $\Sigma_n$ and $\Sigma_{n-1}$, respectively:
\begin{align*}
A&\longrightarrow A_n^\LL(d)\Longrightarrow\bx_n\longrightarrow\left( \Sigma_n
\right)^-,\;n\ge 0\\
A&\longrightarrow A_{n}^\R(d)\Longrightarrow\bx_n\longrightarrow \left( \Sigma_{n-1}
\right)^+,\;n\ge 1,
\end{align*}
(see Figure~\ref{fig:bif_T-per} for $n=2$). 
\begin{figure}
\begin{center}
\begin{picture}(1,0.8)
\put(0,0.8){
\subfigure[\label{fig:map_T_R-bif}]{\includegraphics[angle=-90,width=0.5\textwidth]
{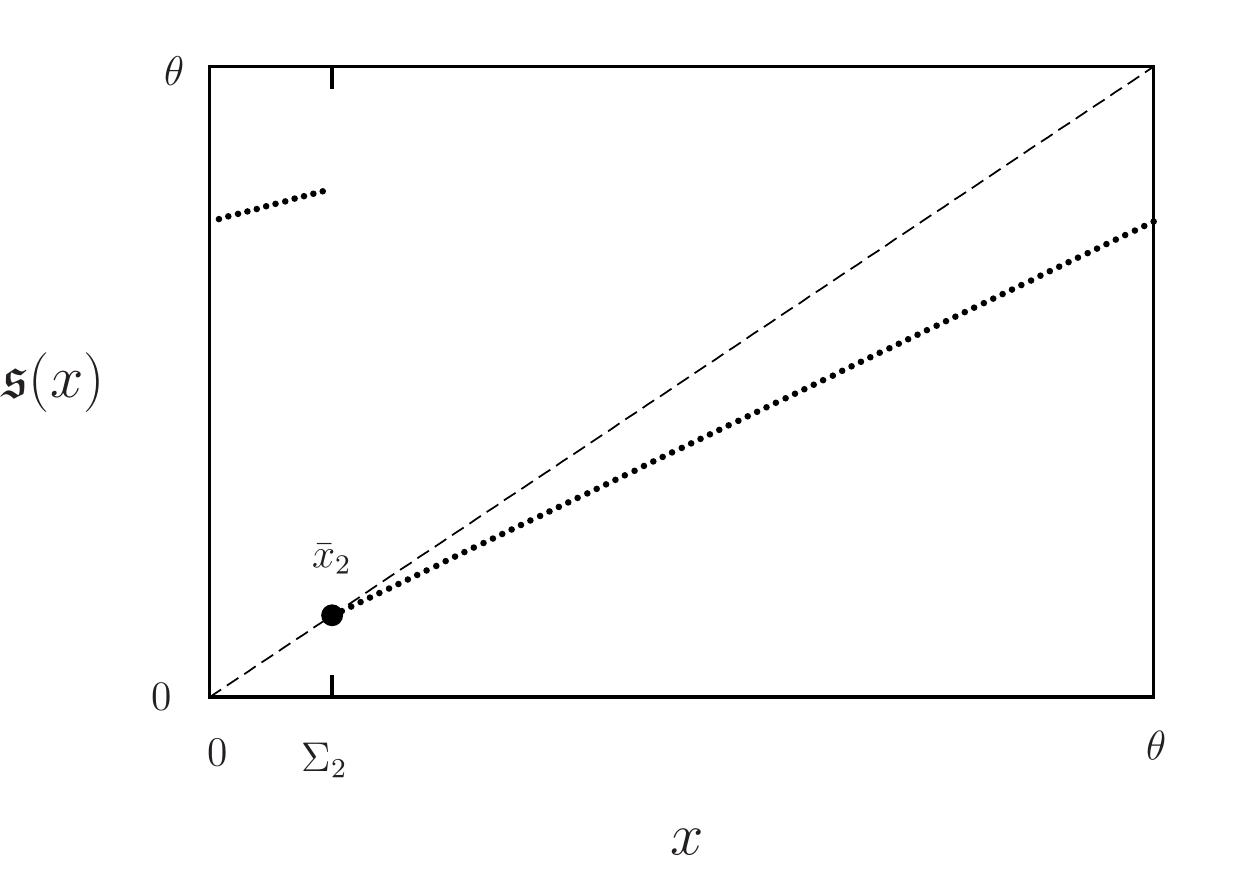}}
}
\put(0.5,0.8){
\subfigure[\label{fig:T-per_R-bif}]
{\includegraphics[angle=-90,width=0.5\textwidth]{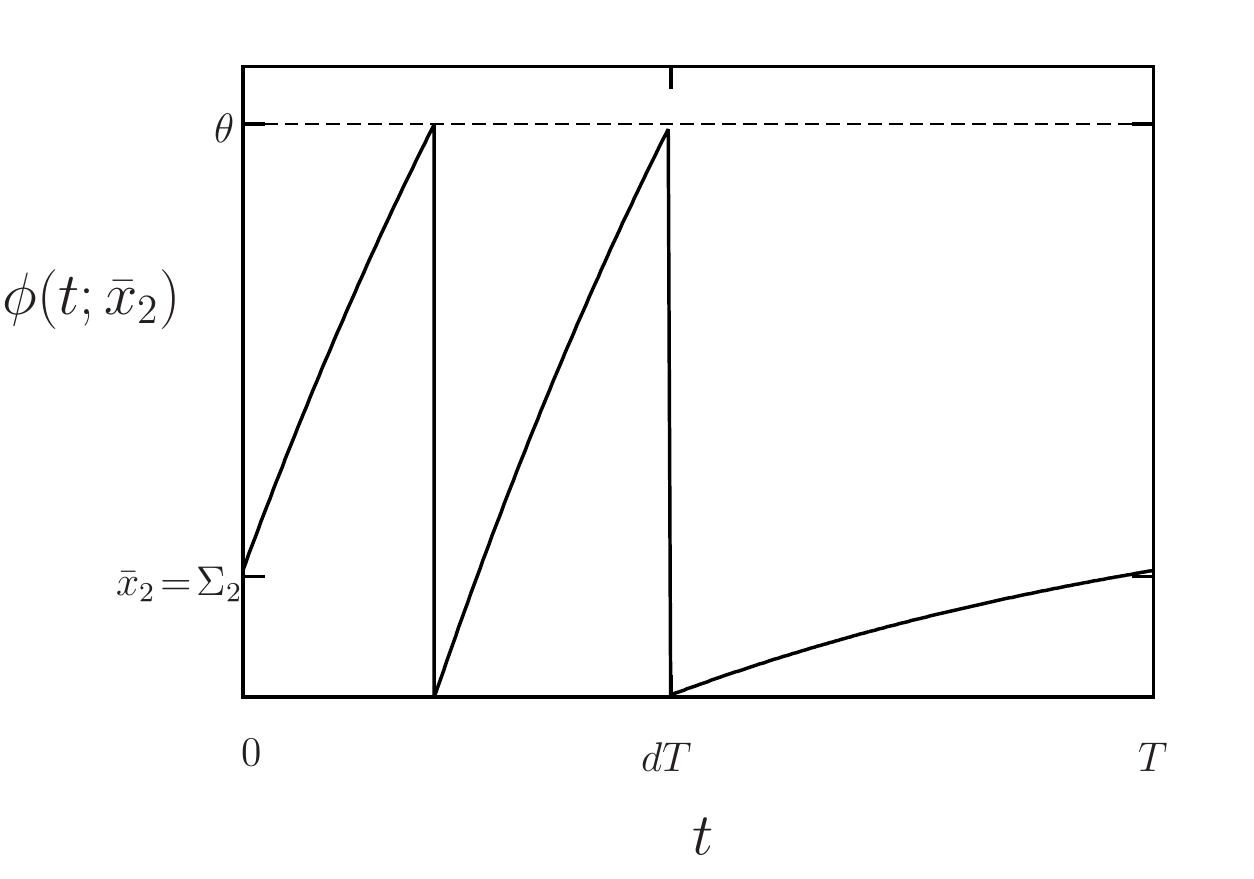}}
}
\put(0,0.4){
\subfigure[\label{fig:map_T_L-bif}]{\includegraphics[angle=-90,width=0.5\textwidth]
{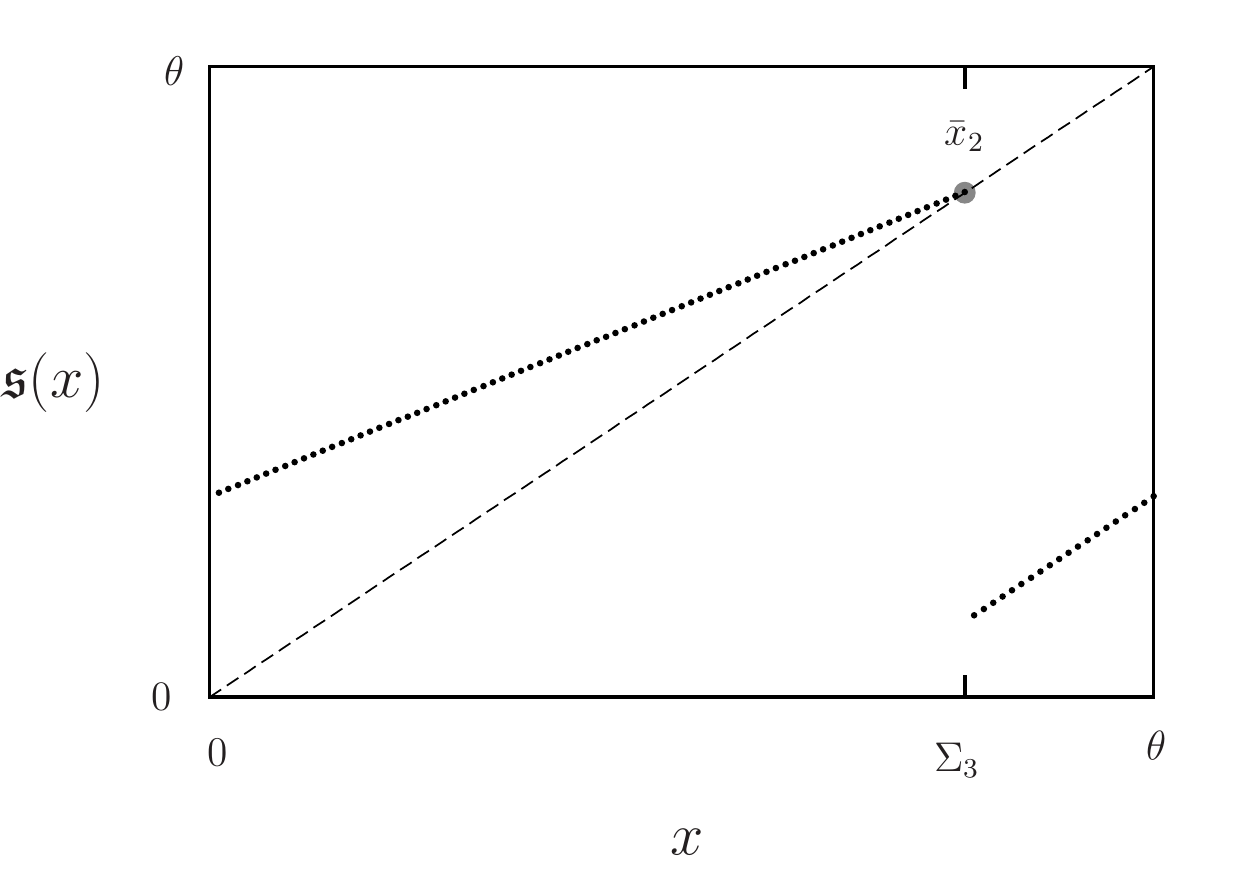}}
}
\put(0.5,0.4){
\subfigure[\label{fig:T-per_L-bif}]
{\includegraphics[angle=-90,width=0.5\textwidth]{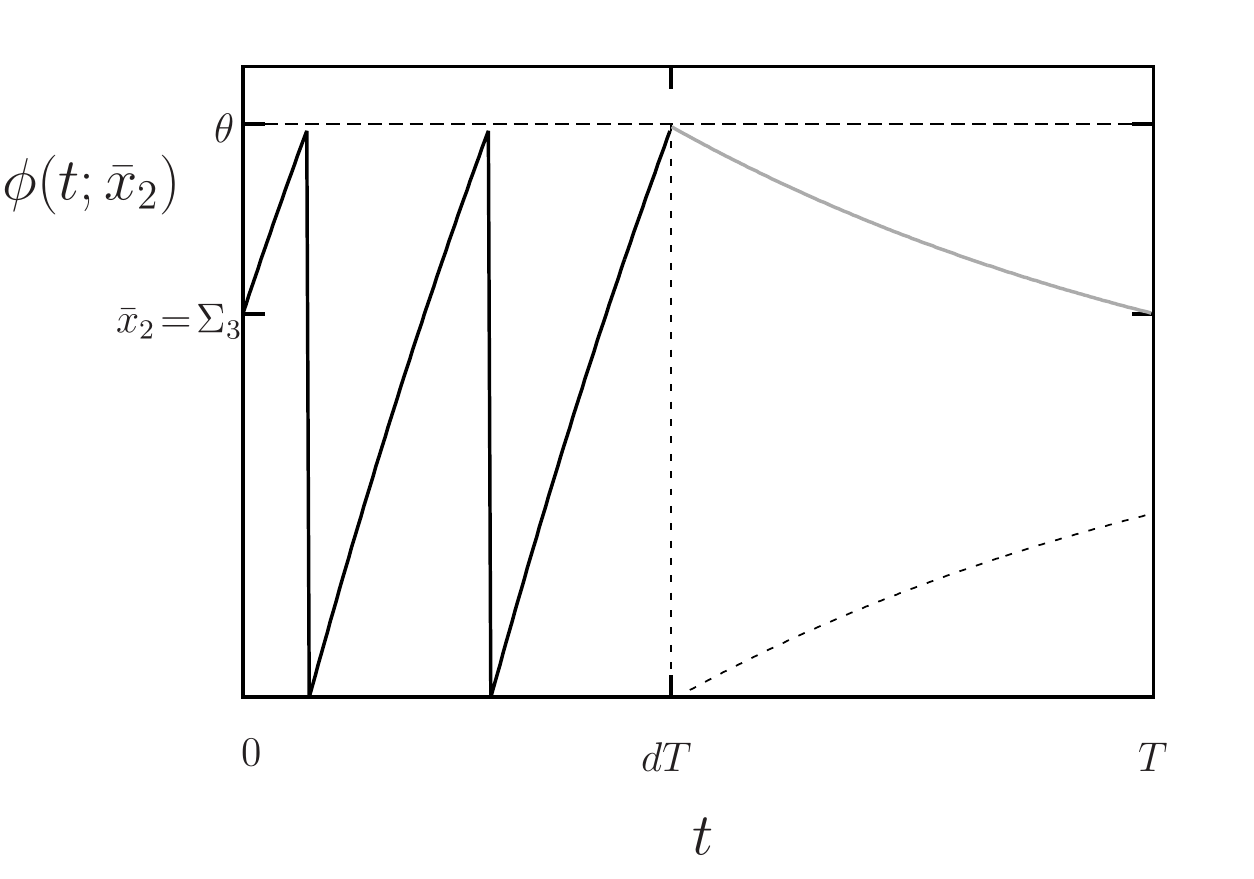}}
}
\end{picture}
\end{center}
\caption{(a) Fixed point $\bx_2$ colliding with the boundary
$\Sigma_2$ for $A=A_2^\R$ (b) $T$-periodic orbit for $A=A_2^\R$. (c)
fixed point $\bx_2$ colliding with the boundary $\Sigma_3$ for
$A=A_3^\LL$. (d) $T$-periodic orbit for $A=A_3^\LL$. Note that the
periodic orbit in (d) should exhibit a reset at $t=dT$; this is why it
is plot in  gray color.}
\label{fig:bif_T-per}
\end{figure}

As given in Equation~\eqref{eq:A_monotonically_dec}, $\Sigma_n$
monotonically decreases with $A$. Hence, for any $n\ge0$, one can
apply a reparametrization such that $\s$ can be written as in
Equation~\eqref{eq:normal_form} and $A$ becomes equivalent to
the parameter $\lambda$ in Equation~\eqref{eq:1d_scann_curve}.
Therefore, if $n$ is large enough such that $\s$ is contracting,
conditions \condsH{} and {\em i)} of
Theorem~\ref{theo:adding_incrementing} are satisfied, and $\s$
exhibits a period adding bifurcation structure under variation of
$A\in [A_n^\LL,A_{n+1}^\R]$ (see~\cite{GraKruCle14} Proposition 3.4).
As this occurs for any $d\in(0,1)$, the parameter space $(d,1/A)$
possesses an infinite number of period adding structures.
\begin{figure}
\begin{center}
\begin{picture}(1,0.5)
\put(0,0.35){
\subfigure[\label{fig:d-invA_T1d9}]
{\includegraphics[angle=-90,width=0.5\textwidth]{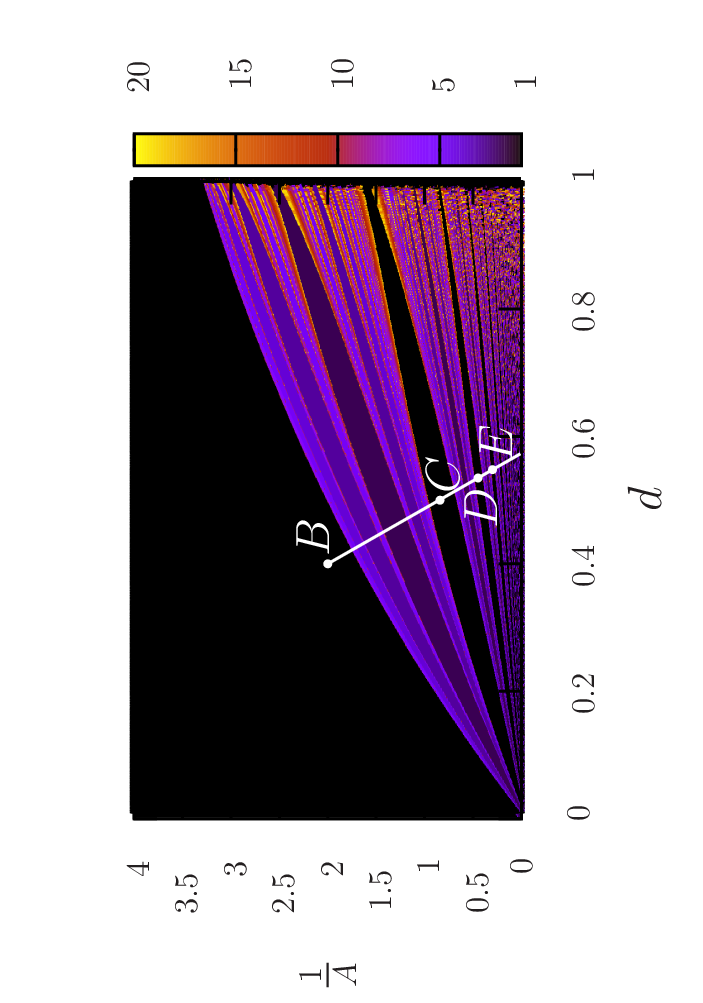}}
}
\put(0.5,0.35){
\subfigure[\label{fig:1dscann_T1d9}]
{\includegraphics[angle=-90,width=0.5\textwidth]{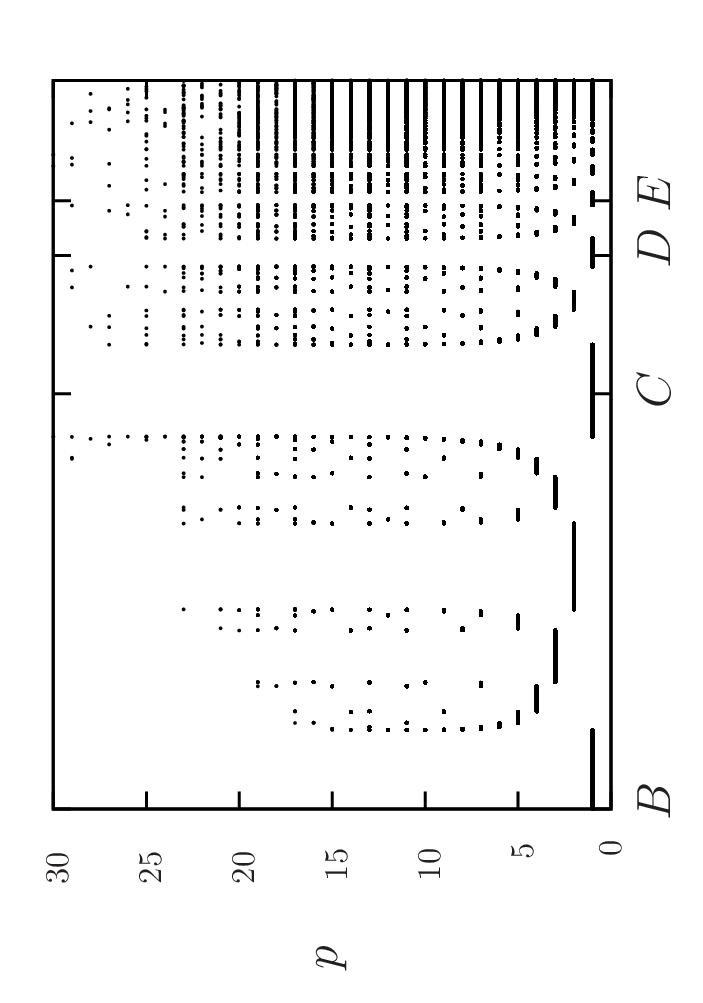}}
}
\end{picture}
\end{center}
\caption{(a) Bifurcation scenario in the $(d,1/A)$ parameter
space for the map $\s$ with $f(x)=-0.5x+0.2$,
$\theta=1$ and $T=1.9$.  The colors refer to the periods of the
periodic orbits found by simulating the system. (b) period of the
periodic orbits found along the the line shown in (a).}
\label{fig:T1d9}
\end{figure}
These are shown Figure~\ref{fig:T1d9} for a linear system for $f(x)$.
The black regions shown in Figure~\ref{fig:d-invA_T1d9} correspond to
fixed points; for example, for parameter values at points $B$, $C$,
$D$ and $E$ there exist the fixed points $\bx_0\in S_0$, $\bx_1\in
S_1$, $\bx_2\in S_2$ and $\bx_3\in S_3$, respectively. As one can see
from the periods of the periodic orbits shown in
Figure~\ref{fig:1dscann_T1d9}, there exist period adding structures
nested between the regions of existence of fixed points, as predicted
by Theorem~\ref{theo:adding_incrementing}. Note however that the
adding structure between $B$ and $C$ (involving the discontinuity
$\Sigma_1$) is not complete. This is due to the fact that $n=1$ is not
large enough to guarantee the contractiveness of $\s$ at the interval
$[\Sigma_1,\theta)$. As the map $\s$ (after proper reparametrization
and change of variables) satisfies conditions~\condsC{},
Proposition~\ref{pro:po_rational_rho} holds to provide the existence
of a periodic orbit. However, due to the expansiveness,
periodic orbits may be unstable and not unique, and
hence they are not easy to detect by direct simulation.\\

One of the most interesting properties that one can derive from the
period adding structures comes from the symbolic dynamics and the
rotation numbers associated with these periodic orbits.  Specifically,
consider parameter values such that $\Sigma_n\in(0,\theta)$ and
suppose that the stroboscopioc map has a periodic orbit. As explained
above, on each iteration of the stroboscopioc map the flow can perform
$n-1$ or $n$ spikes. We assign a symbolic sequence to this periodic
orbit by letting  its $k$th symbol to be $\LL$ if this number is $n-1$
and $\R$ if it is $n$.  Hence, when divided by the period of the
periodic orbit, this becomes the so-called {\em firing number}, which
is the average number of spikes per iteration of the stroboscopic map.
Therefore, after recalling Definition~\ref{def:eta-number} and
Corollary~\ref{cor:eta-number_rotation-number}, it is easy to see that
the rotation number ($\rho)$ and the firing number ($\eta$) satisfy
the relation
\begin{equation*}
\eta=n+\rho,
\end{equation*}
where $n$ is such that the periodic orbit steps around  the
discontinuity $\Sigma_{n+1}$, $n\ge0$. Therefore, by
Theorem~\ref{theo:adding_incrementing}, under parameter variation, the
firing number follows a devil's staircase growing from $0$ to infinity
(see Figure~\ref{fig:fr-fn}).
\begin{figure}
\begin{center}
\begin{picture}(1,0.5)
\put(0,0.35){
\subfigure[\label{fig:firing_number}]
{\includegraphics[angle=-90,width=0.5\textwidth]{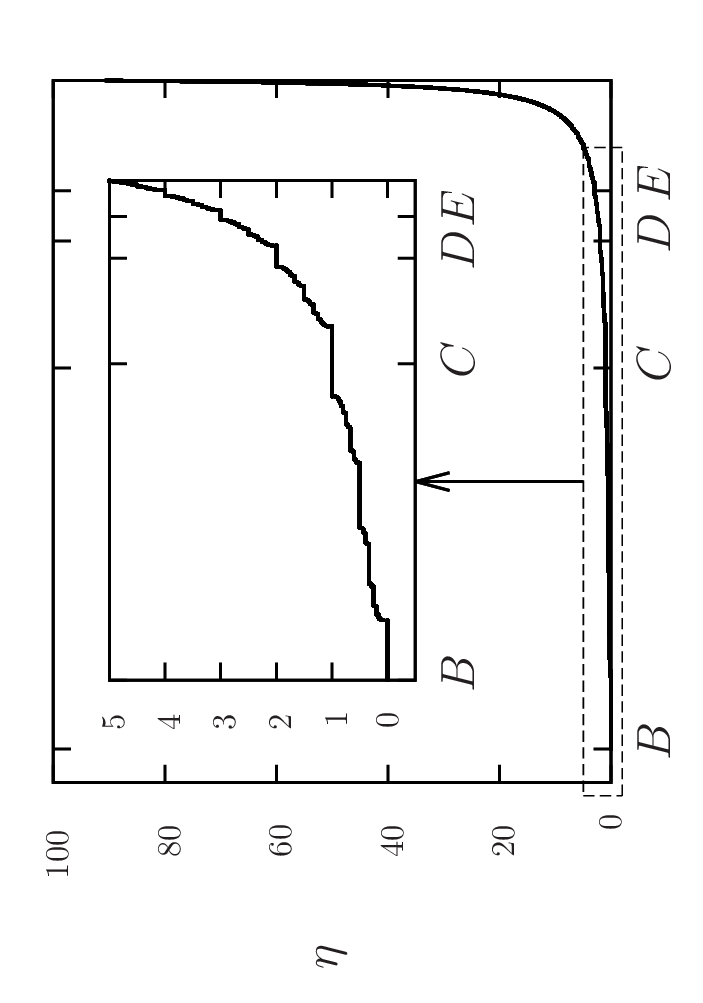}}
}
\put(0.5,0.35){
\subfigure[\label{fig:firing_rate}]
{\includegraphics[angle=-90,width=0.5\textwidth]{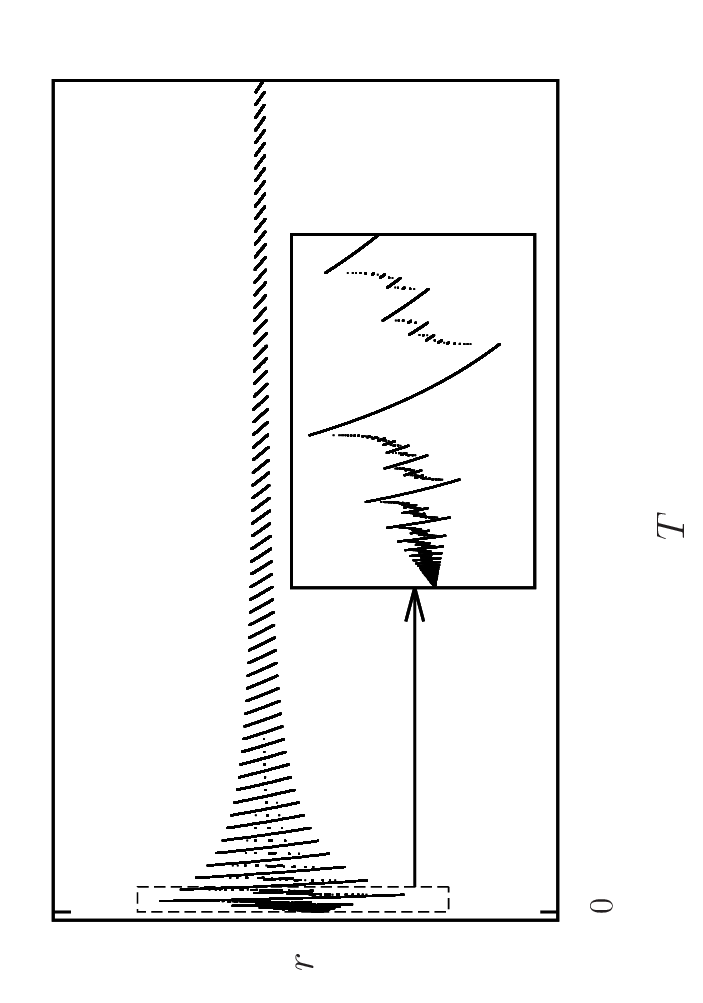}}
}
\end{picture}
\end{center}
\caption{(a) Firing number along the line shown in
Figure~\ref{fig:d-invA_T1d9}, for the same system. (b) Firing rate
under variation of $T$.}
\label{fig:fr-fn}
\end{figure}
When divided by $T$ (period of the periodic forcing $I(t)$), the
firing number becomes the so-called {\em firing rate}, which is the
asymptotic average number of spikes per unit time. Hence, it also
follows a devil's staircase when varying parameters $A$ and $d$ along
lines as the one shown in Figure~\ref{fig:d-invA_T1d9}. However, as
most of the bifurcation structures shown in
Figure~\ref{fig:d-invA_T1d9} do not qualitatively depend on $T$, when
this parameter is varied the firing rate follows a devil's staircase
with pieces of hyperbolas as steps. This is shown in
Figure~\ref{fig:firing_rate}, and more details on this frequency
analysis can be found in~\cite{GraKru15}.

\subsection{Examples in higher dimensions}\label{sec:examples_highdim}
\subsubsection{Higher order sliding-mode controlled system with relays}\label{sec:high-order_relays}
In this example we recover the results presented in~\cite{FosGra13},
which extend the example shown in Section~\ref{sec:relay} to higher
order linear systems.\\
Assume that system~\eqref{eq:open_loop} is a linear system of order
$n$ with an equilibrium point at the origin.  In this case, due to the
linearity, the block representing the open-loop system in
Figure~\ref{fig:system_with_relay} can be replaced by its Laplace
transform
\begin{equation*}
G_s(s)=\frac{b}{U(s)},
\end{equation*}
where $U(s)$ is a polynomial of the form
\begin{equation}
U(s)=s^n+a_{n-1}s^{n-1}+\cdots +a_0,
\label{eq:polynomial}
\end{equation}
which we assume to have real negative roots only.\\
The system can be written in terms of a differential equation as
\begin{equation}
y^{n)}+a_{n-1}y^{n-1)}+\cdots+a_0y=bu(t),
\label{eq:n-order_ode}
\end{equation}
with $\,y(t),a_i,b\in\RR$, $y^{i)}=d^iy/dt^i$ and $u(t)\in\RR$ is the input of the system.
We wish to design a control such that the output of the system,
$y(t)\in\RR$, is stabilized around a new equilibrium point $y^*$; that
is, $y(t)\simeq y^*$ and $y^{i)}\simeq 0$ for $1\le i\le n$. To
achieve this, we consider a two-step control. A first action is taken
by a classic controller given by the block $G_c(s)$, which we assume
to be of the form
\begin{equation*}
G_c(s)=1+c_1s+\ldots +c_{n-1}s^{n-1}.
\end{equation*}
The second part will be given by a relay of gain $k$. Then, the
desired motion becomes
\begin{equation*}
y-y^*+c_1y^{1)}+\ldots
+c_{n-1}y^{n-1)}=0.
\end{equation*}
Equivalently, system~\eqref{eq:n-order_ode} can be written as a first
order $n$-dimensional system as
\begin{equation}
\dot{\by}=A\by+Bu,
\label{eq:n-order_matricial}
\end{equation}
where
\begin{equation*}
\by=(y,y^{1)},\dots,y^{n-1)})\in \RR^n,
\end{equation*}
and $A$ and $B$ become
\begin{equation*}
A=\left( \begin{array}{cccccc}
0&1&0&\ldots&&0\\
0&0&1&&&0\\
\vdots&\vdots&&\ddots\\
0&0&\ldots&0&1&0\\
-a_0&-a_1&\ldots&&-a_{n-2}&-a_{n-1}
\end{array} \right),\,B=
\left( \begin{array}{c}
0\\ \vdots\\ 0\\b
\end{array} \right),
\end{equation*}
and we wish to stabilize the system around the point
\begin{equation*}
\by=
\left( \begin{array}{c}
y^*\\0\\ \vdots\\ 0
\end{array} \right)
\end{equation*}
After the $T$-periodic sampling, the input $u$ becomes constant
at the intervals
$t\in[iT,(i+1)T)$:
\begin{equation}
u=\left\{
\begin{aligned}
&-k &&\text{if }\sigma(\by)<0\\
& k &&\text{if }\sigma(\by)>0
\end{aligned}\right.
\label{eq:input_n-order}
\end{equation}
where
\begin{equation}
\sigma(\by)=y-y^*+c_1y^{1)}+\ldots
+c_{n-1}y^{n-1)}.
\label{eq:sliding_surf}
\end{equation}
We wish to stabilize the system close to the switching surface
\begin{equation*}
\sigma(\by)=0.
\end{equation*}
The fact that the polynomial $U(s)$ possesses real negative roots only
ensures us that, for $k=0$, system~\eqref{eq:n-order_matricial}
possesses an attracting node at the origin. In order to find proper
values of $k$ such that there exists dynamics satisfying $\sigma(\by(t))=0$, we
impose sliding motion on the surface given by $\sigma=0$, which occurs
when the vector fields $F^{\pm}=A\by\pm Bk$, obtained by replacing
$u=\pm k$, point both to the surface $\sigma$.  Since $F^{\pm}$ are
smooth everywhere,  this can be checked through
\begin{equation}
\left({L_{F^{+}}\sigma}\right)\left({L_{F^{-}}\sigma}\right)<0.
\label{eq:non_degeneracy_Lie}
\end{equation}
Let us define $u_{eq}=-\frac{(\nabla\sigma)A\by}{c_{n-1}b}$. Then, the
previous inequality meets on the subset of $\sigma$ defined by
\begin{equation}
-\vert k\vert <u_{eq}<\vert k\vert
\label{eq:non_degeneracy_linear}
\end{equation}
(see \cite{Utk77} for details). In turn, this result can be read as
\textit{for $k$ properly selected (both in sign and in absolute
value), there is sliding motion on $\sigma$}.

After the discretization performed by the $T$-periodic sampling, the
system is better understood by the time-$T$ return (stroboscopic) map
\begin{equation}
P(y)=\left\{
\begin{aligned}
&P_\LL(\by):=\rho \by+\mu_\LL&&\text{if }\sigma(\by)<0\\
&P_\R(\by):=\rho \by+\mu_\R&&\text{if }\sigma(\by)>0,
\end{aligned}\right.
\label{eq:norder_sliding_map}
\end{equation}
where $\rho$, $\mu_\LL$ and $\mu_\R$ are the matrices
\begin{equation*}
\rho=e^{AT},\quad \mu_\R=k(\rho-Id)(A^{-1}B),\quad
\mu_\LL=-k(\rho-Id)(A^{-1}B).
\end{equation*}
The maps $P_\LL$ and $P_\R$  are the time-$T$ return maps
(stroboscopic) maps associated with the fields $F^{+}$ and $F^{-}$,
respectively.\\
Using that system~\eqref{eq:n-order_matricial} is linear and that, for
$k=0$, it possesses an attracting node,  the matrix $\rho$ possesses
real positive eigenvalues with modulus less than $1$ only.\\
Each branch of the map~(\ref{eq:norder_sliding_map}), $P_r$ and
$P_\LL$, has a fixed point
\begin{equation}
\byr=-(\rho-Id)^{-1}\mu_\R,\quad\byl=-(\rho-Id)^{-1}\mu_\LL,
\label{eq:fixed_points_general_map}
\end{equation}
which may be \emph{feasible} or \emph{virtual} depending on whether
they belong to the domain of their respective map or not. When
feasible, these fixed points become attracting nodes, and they undergo
border collision bifurcations  when they collide with boundary
$\sigma(\by)=0$.

Regarding the possible dynamics, we distinguish between three
different situations.\\
If both fixed points are feasible ($\sigma(\byr)>0$ and
$\sigma(\byl)<0$) they also become attracting fixed points of the
map~(\ref{eq:norder_sliding_map}).  Their
domains of attraction are formed by the values of $\by\in\RR^n$ such
that $\sigma(\by)>0$ and $\sigma(\by)<0$, respectively.\\
If only one of both fixed points is feasible ($\sigma(\byr)<0$ and
$\sigma(\byl)<0$ or vice-versa), then it becomes the unique fixed
point of the map~(\ref{eq:norder_sliding_map}). As it is attracting,
all trajectories tend towards it, and now its domain of attraction
becomes $\RR^n$.\\
Note that, in these two previous cases, the control specification is
not fulfilled as $\sigma(\by)=0$ is not flow invariant by the vector
fields $F^\pm$.

The third situation occurs when both fixed points are virtual
($\sigma(\byr)<0$ and $\sigma(\byl)>0$). This occurs when the sliding
condition~\eqref{eq:non_degeneracy_Lie} is fulfilled, guaranteeing that
the original time-continuous system possesses sliding motion on
$\sigma(\by)=0$. Provided that the fixed points are attracting, the
map~\eqref{eq:norder_sliding_map} satisfies condition {\em 2.} of
Theorem~\ref{theo:maximin_quasi-contraction}. Letting
\begin{equation*}
\Sigma=\left\{ \by\in\RR^n,\,|\,\sigma(\by)=0 \right\}\cap U,
\end{equation*}
condition {\em 1.} is also satisfied by taking the set $U$ small
enough but containing both virtual fixed points. As a consequence and
noting that both $P_\LL$ and $P_\R$ preserve orientation as they are
obtained from the integration of a flow, the map $P$ admits $0$ or
$1$ periodic orbit. When it exists, such periodic orbit must have 
a symbolic itinerary contained in the Farey tree of symbolic
sequences shown in Figure~\ref{fig:farey_sequences}.

As noted in Remark~\ref{rem:bif_stru_Rn}, the results given in
Section~\ref{sec:maximin_approach} do not provide information about
bifurcation structures. In Figure~\ref{fig:bbb_2d} we show the
bifurcation structure obtained for the planar case when the parameters
$k$ and $y^*$ are varied. As one can see in
Figure~\ref{fig:bbb_periods_2d}, the obtained structure resembles the
period adding described in Section~\ref{sec:inc-inc_overview} for the
one-dimensional case, although it is not known whether for the planar
case the gluing of orbits is also fractal. A similar physical
interpretation of the bifurcation structure as in the first-order case
(Section~\ref{sec:relay}) holds also also in this case.
\begin{figure}
\begin{center}
\begin{picture}(1,0.5)
\put(0,0.4){
\subfigure[\label{fig:bbb_region_2d}]
{\includegraphics[width=0.35\textwidth,angle=-90]
{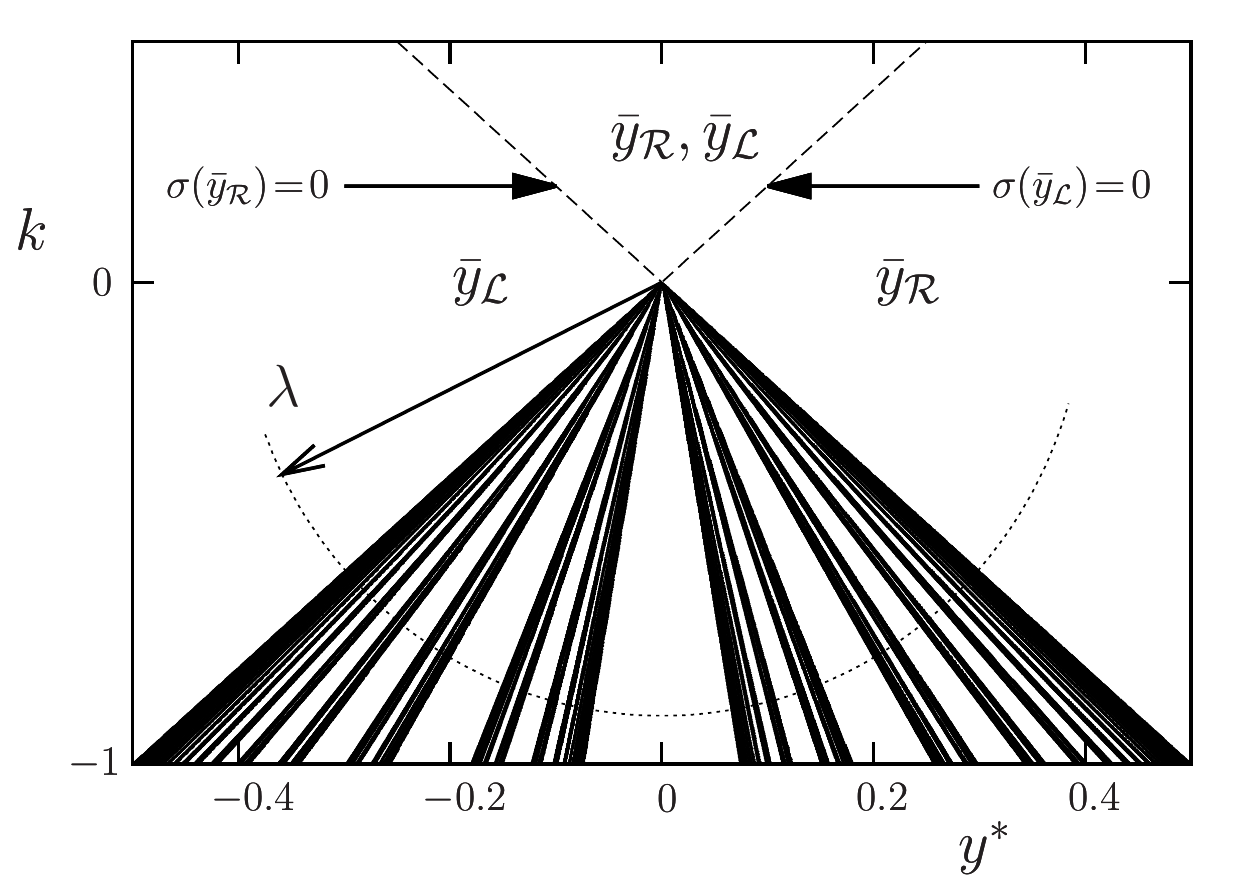}}
}
\put(0.5,0.4){
\subfigure[\label{fig:bbb_periods_2d}]
{\includegraphics[width=0.35\textwidth,angle=-90]
{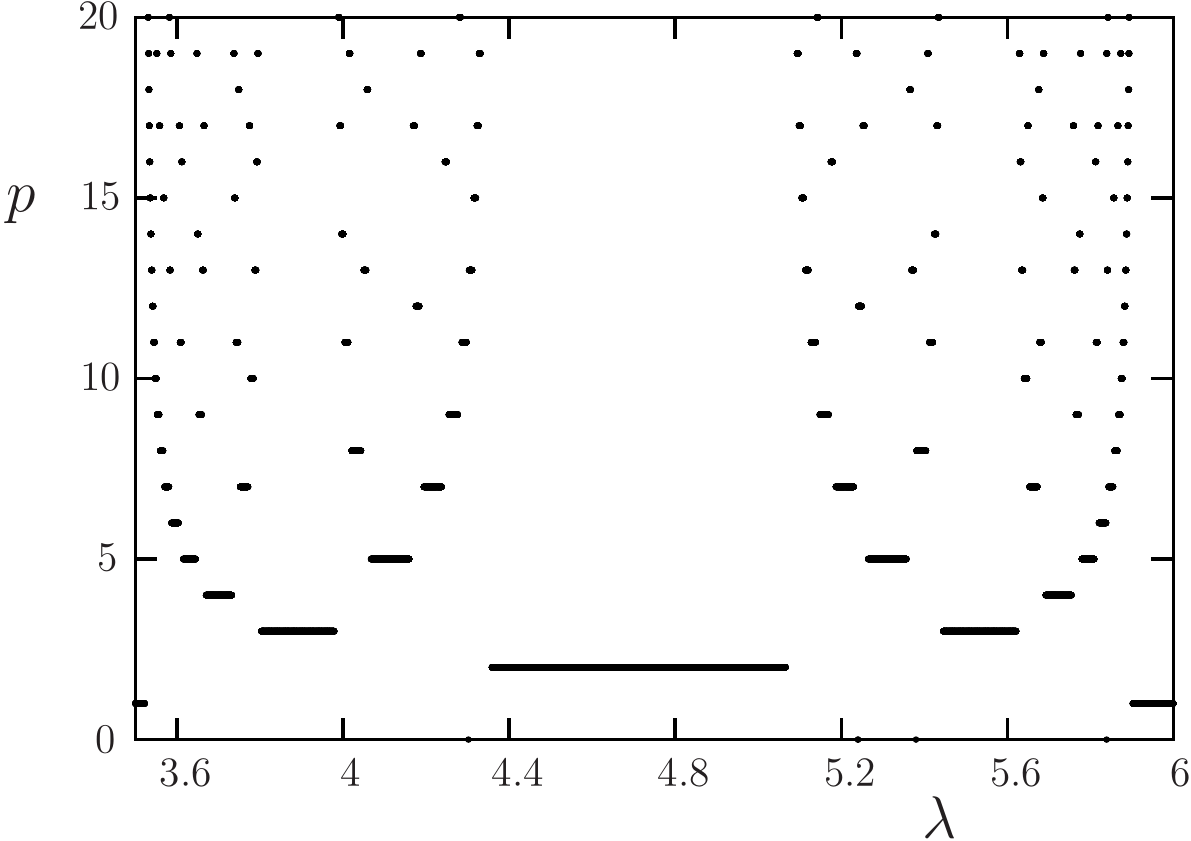}}
}
\end{picture}
\end{center}
\caption{Bifurcation scenario in the $(y^*,k)$ parameter space for the
planar case, $n=2$. The rest of parameters are fixed to $a_0=-2$,
$a_1=-5$, $b=1$, $c_1=1.5$ and $T=0.1$. In~(a) we show the border collision
bifurcation curves separating existence regions of periodic orbits. In~(b) the
periods of the periodic orbits found along the pointed curve in~(a) parametrized
by $\lambda$.}
\label{fig:bbb_2d}
\end{figure}

\subsubsection{ZAD-controlled DC-DC boost converter}\label{sec:boost_zad}
In this applied example we recover the results shown
in~\cite{AmaCasGraOliHur14} and state them in terms of the
quasi-contractions of Section~\ref{sec:maximin_approach}.  We consider
a DC-DC boost converter, which is aimed to convert a given constant DC
voltage, $v_i$, into a desired lower one, $v_o$, also DC. Its circuit
is schematically shown in Figure~\ref{fig:boost}.
\begin{figure}
\begin{center}
\includegraphics[width=0.6\textwidth]{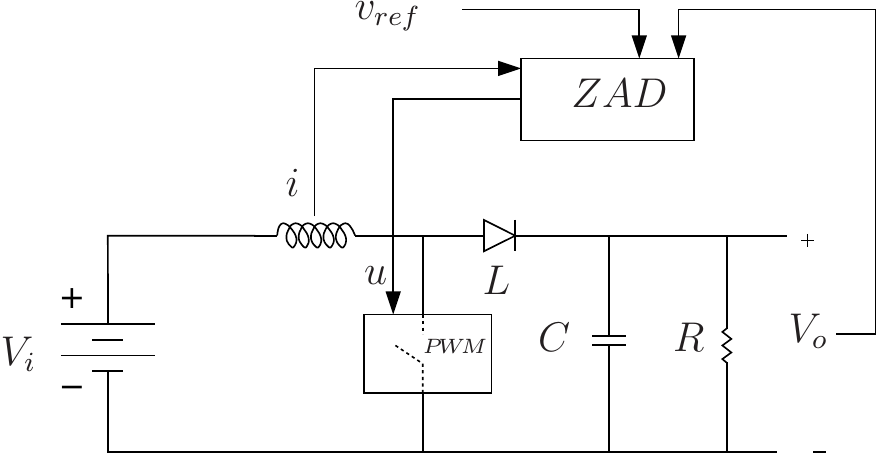}
\end{center}
\caption{Schematic representation of a $ZAD$-controlled boost
converter.}
\label{fig:boost}
\end{figure}
To achieve this conversion in a robust way (to guarantee a
certain stability of $v_o$ under possible
fluctuations of $v_i$), the transistor shown in Figure~\ref{fig:boost} 
needs to be properly controlled. There exist different approaches to design
such a control; in this example we consider the so-called Zero
Average Dynamics (ZAD) strategy, which is well extended in the
community~(\cite{FosGriBie00,AngFosOli05,AngOliBer08,AngOliTab08,FosHogSea09,AvrFosGraSch11,AmaCasGraOliHur14}).\\
Letting $v$ be the voltage at the capacitor, $i$ the current through the
solenoid, calling
\begin{equation*}
x=\frac{v}{v_i}\quad y=\sqrt{\frac{L}{C}}\frac{i}{v_i}
\end{equation*}
and rescaling time by a factor of $\sqrt{LC}$, the non-dimensional
equations that model the system become
\begin{equation}
\begin{aligned}
\dot{x}&=-\gamma x+y(1-u)\\
\dot{y}&=-x(1-u)+1,
\end{aligned}
\label{eq:boost_nondim_equations}
\end{equation}
with $\gamma=\sqrt{\frac{L}{R^2C}}$. The function $u$ takes the values
$0$ or $1$ depending on whether the transistor is opened or closed,
respectively. It is given by the output of 
a Pulse Width Modulation (PWM) process, and can be written as
\begin{equation*}
u(t)=\left\{
\begin{aligned}
&1&&\text{if }kT\le t< kT+1/2Td_k\\
&0 &&\text{if }kT+1/2Td_k\le t< (k+1)T-1/2Td_k\\
&1&& \text{if }(k+1)T-1/2Td_k\le t<(k+1)T
\end{aligned}
\right.,
\end{equation*}
where $d_k\in[0,1]$ is the so-called duty cycle. It is computed at
each sampling moment, $t_k=kT$, and is kept constant until the next
sampling period, $t_{k+1}$. Note that, if $d_k=1$ or $d_k=0$, then $u$
becomes constant in the time interval $[kT,(k+1)T)$, equal to $1$ or
$0$, respectively.\\
The ZAD strategy consists of computing the value of the duty cycle at
each $T$-time interval, $d_k$, by imposing
\begin{equation}
\int_{kT}^{(k+1)T}s(t)dt=0,\forall k\in \mathbb{Z},
\label{eq:zad}
\end{equation}
where $s(t)$ is the error surface
\begin{equation*}
s(x(t),y(t))=k_1\left( x(t)-x_{ref} \right)+k_2\left(
y(t)-y_{ref} \right),
\end{equation*}
with
\begin{equation*}
x_{ref}=\frac{v_{ref}}{v_i},\quad
y_{ref}=\left( x_{ref}
\right)^2\gamma.
\end{equation*}
and $k_1$ and $k_2$ two control constants.\\
Equation~\eqref{eq:zad} becomes transcendental in $d_k$. However, if
the solution of system~\eqref{eq:boost_nondim_equations}  is
approximated by piecewise-linear functions, the duty cycle can be
approximated by the closed expression
\begin{equation}
d_k=\frac{2s_0+T\dot{s}_2}{\dot{s}_2-\dot{s}_1},
\label{eq:duty_cycle_approx}
\end{equation}
where
\begin{align*}
\dot{s}_1&=-\gamma k_1 x(kT)+k_2\\
\dot{s}_2&=k_1\left( -\gamma x(kT)+y(kT) \right)+k_2\left( 1-x(kT)
\right)\\
s_0&=k_1\left( x(kT)-x_{ref} \right)+k_2\left( y(kT)-y_{ref} \right),
\end{align*}
(see~\cite{AmaCasGraOliHur14} for more details). Provided that the
duty cycle must be in the interval $[0,1]$, it is set to
$1$ or $0$ depending on whether the result of
expression~\eqref{eq:duty_cycle_approx} is greater than $1$ or less
than $0$, respectively. That is, one imposes the saturation condition
\begin{equation}
d_k=\left\{
\begin{aligned}
&1&&\text{if }\frac{2s_0+T\dot{s}_2}{\dot{s}_2-\dot{s}_2}\ge1\\
&0&&\text{if }\frac{2s_0+T\dot{s}_2}{\dot{s}_2-\dot{s}_2}\le 0.
\end{aligned}
\right.
\label{eq:saturation_condition}
\end{equation}

Let us now consider the time-$T$ return (stroboscopic) map  of
system~\eqref{eq:boost_nondim_equations}. It becomes the composition
of three stroboscopic maps consisting of flowing
system~\eqref{eq:boost_nondim_equations} for the time intervals
$[kT,kT+1/2Td_k)$, $[kT+1/2Td_k,(k+1)T-1/2Td_k)$ and
$[(k+1)T-1/2Td_k,(k+1)T)$ setting $u=1$, $u=0$ and $u=1$, respectively
(see~\cite{AmaCasGraOliHur14} for more details). As these three maps
are smooth maps, their composition is also a smooth map, as long as
$d_k$ given by Equation~\eqref{eq:duty_cycle_approx} lies in the
interval $(0,1)$. However, due to the saturation
condition~\eqref{eq:saturation_condition}, the stroboscopic map
becomes indeed a piecewise-defined map, with two switching manifolds,
\begin{align*}
\Sigma_1&=\left\{ (x,y),\,|\,d_k(x,y)=1 \right\}\\
\Sigma_0&=\left\{ (x,y),\,|\,d_k(x,y)=0 \right\}.
\end{align*}
Hence, the stroboscopic map is indeed defined in three different
partitions.\\
In Figures~\ref{fig:bbb_boost_k1k2} and~\ref{fig:boost_1d_scanns} we
show the numerical results obtained by direct simulation of the
stroboscopic map, by fixing the initial conditions to
$(x_0,y_0)=(2.5,(2.5)^2\gamma)$ and varying the parameters $k_1$ and
$k_2$.
\begin{figure}
\begin{center}
\includegraphics[angle=-90,width=0.7\textwidth]{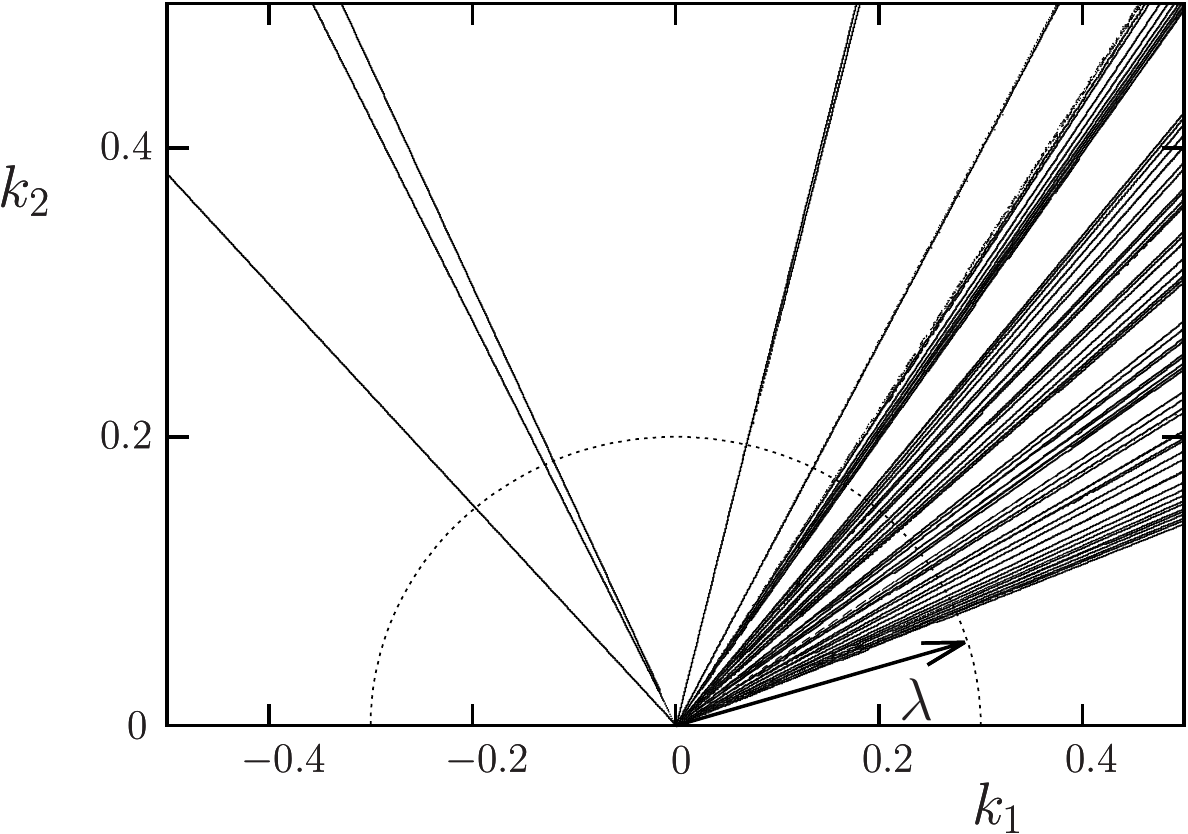}
\end{center}
\caption{Bifurcation curves in the $k_1\times k_2$
parameter space.}
\label{fig:bbb_boost_k1k2}
\end{figure}
\begin{figure}
\begin{center}
\begin{picture}(1,0.5)
\put(0,0.4){
\subfigure[\label{fig:boost_periods}]
{\includegraphics[width=0.35\textwidth,angle=-90]
{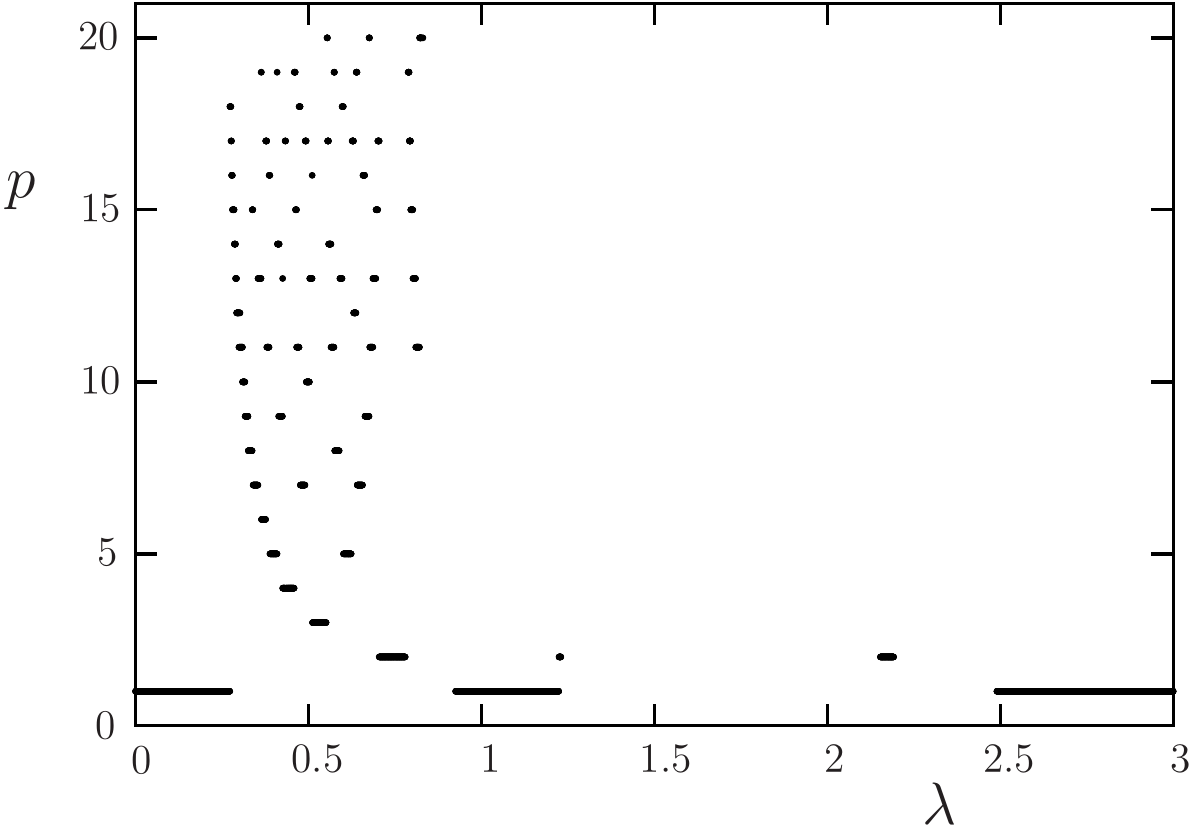}}
}
\put(0.5,0.4){
\subfigure[\label{fig:boost_bif_diagram}]
{\includegraphics[width=0.35\textwidth,angle=-90]
{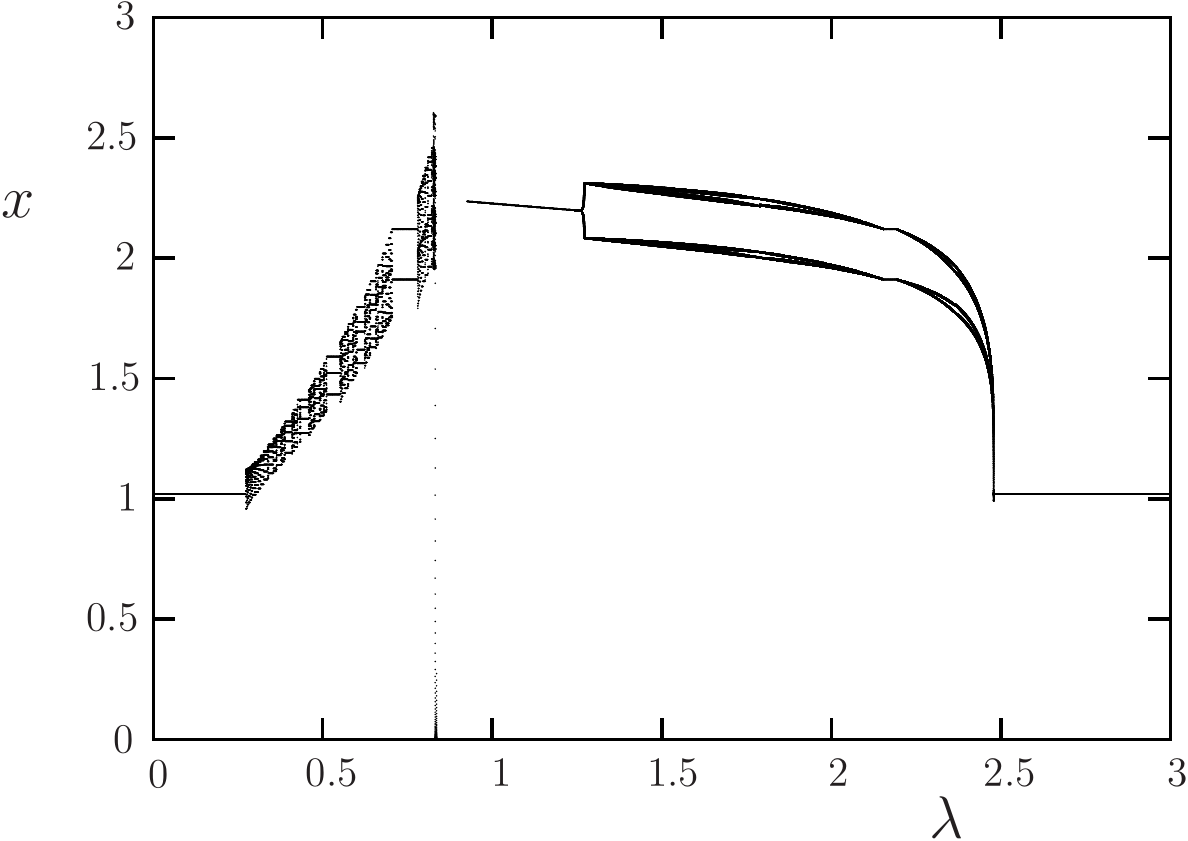}}
}
\end{picture}
\end{center}
\caption{Bifurcation scenario found around the curved labeled in
Figure~\ref{fig:bbb_boost_k1k2} and parametrized by $\lambda$, for parameter values
$\gamma=0.35$, $x_{ref}=2.5$, $T=0.18$. In~(a) we show the periods of
the periodic orbits found when varying $\lambda$. In~(b) we show the
bifurcation diagram corresponding to $x$.}
\label{fig:boost_1d_scanns}
\end{figure}%
In Figure~\ref{fig:bbb_boost_k1k2} we show the bifurcation curves in
the parameter space $(k_1,k_2)$. Due to the linearity  of the system,
these becomes straight lines. In Figure~\ref{fig:boost_1d_scanns} we
show the results obtained along the curve labeled in
Figure~\ref{fig:bbb_boost_k1k2} and parametrized by the angle
$\lambda$. As shown in Figure~\ref{fig:boost_periods}, for
$\lambda\in(0,0.27)$, $\lambda\in (0.95,1.2)$ and $\lambda\in (2.5,3)$
(approximately), the stroboscopic map possesses a fixed point. For the
first and third intervals, the duty cycle associated with the fixed
points is saturated to $0$, whereas for the second one the duty cycle
lies in the interval $(0,1)$. As one can see in
Figure~\ref{fig:boost_periods}, there exist a period adding-like
structure between the first two intervals. The periodic orbits that
one finds there possess points in the partitions where the duty cycles
are saturated to $1$ and to $0$. As a consequence, associated to each
periodic orbit, there exists a sequence of $0$'s and $1$'s, which are
distributed such that such sequences are maximin (see
Definition~\ref{def:maximin}), or, equivalently, they belong to the
Farey tree of symbolic sequences shown in
Figure~\ref{fig:farey_sequences}. For example, for the two found
$5$-periodic orbits one gets the sequences of duty cycles
$(0,0,0,0,1)$ and $(0,0,1,0,1)$, which are maximin. Note that, when
$d_k$ is saturated to $1$, the transistor remains closed for a whole
sampling period and more energy is drained from the source. Therefore,
if periodic orbits with high rotation number are not desired, as they
possess.\\
Although a more accurate study of the properties of the stroboscopic
map is needed, provided that the map contracts, there is evidence that
the stroboscopic map is a quasi-contraction, at least for those values
of $\lambda$ corresponding to the period adding-like structure.

Finally, note that for values of $\lambda$ in $(1.2,2.5)$,
trajectories with the used initial condition converge towards a
chaotic attractor or are aperiodic.

\section{Conclusions and future directions}\label{sec:conclusions_future}
In the recent years, piecewise-smooth maps, both in $\RR$ and in
$\RR^n$, have been widely investigated and very interesting
bifurcation phenomena (border collisions, big bangs, period adding,
period incrementing,...) have been reported by many authors. This
included both works by a more applied community, like in power
electronics~(\cite{AmaCasGraOliHur14,AngFosOli05,AngOliBer08,AngOliTab08,AvrMoseZhuGar15,BanKarYuaYor00,BerBudCha98,HamDeaJef92,KapBanPat10,Saitoetal07}), control
theory~(\cite{BerGarIanVas02,FosGra11,FosGra13,MatSai08,Zhuetal01,ZhuMos03,ZhuMosDeBan08,ZhuMosAndMik13}),
biology~(\cite{FreGal11,MonDanSelTsiHas11}),
neuroscience~(\cite{GraKru15,GraKruCle14,JimMihBroNieRub13,Kee80,KeeHopRin81,MenHugRin12,TieFellSej02,Ton14,TouBre08,TouBre09})
or economy (\cite{TraGarWes11}) but also researchers from a more
theoretical and/or computational perspective in non-smooth systems
(\cite{AvrEckSch08a,AvrEckSch08b,AvrEckSch08c,AvrEckSchSch12,FutAvrSch12,AvrGraSch11,AvrSch04,AvrSch05,AvrSch06,AvrSchBan06,AvrSchBan07,AvrSchGar10,AvrSchGar10b,AvrSchSch10a,AvrSus13,BanGre99,Ber82,DutRouBanAla07,GarAvrSus14,GarTra10,HogHigGri06,Kowalczyk05,MakLam12,NorKow06,RakAprBan10,TraGarWes11,TraWesGar10}).
It is well known that some of these bifurcations resemble phenomena
observed and highly studied in other contexts, specially in theory for
circle maps in the 80's and early 90's
(\cite{AlsFal03,AlsLli89,AlsLliMis00,AlsLliMisTre89,AlsMan90,AlsMan96,ArnCouTre81,Ber82,Boy85,CouGamTre84,Gam87,GamGleTre88,GamLanTre84,GamTre88,GhrHol94,Kaneko82,LyuPikZak89,Mira87,RhoTho86,Swi89,Vee89,Zak93})
and homoclinic bifurcations
(\cite{GamProThoTre86,Hom96,Homburg00,HomKra00,ProThoTre87,Spa82,TurShi87}.

In this survey article we have considered a general setting for
piecewise-smooth contracting maps with a discontinuity at the origin
and exhibiting a co-dimension two border collision bifurcation. It is
well known that, depending on the sign of the slopes of the map near
the discontinuity, one finds mainly two different bifurcation
scenarios: the period adding (for the increasing-increasing case) and
the period incrementing (when the slopes have different sign). In this
survey article we have revisited the latter and shown a path through
the literature providing a rigorous proof for the former. Moreover, we
have discussed up to which extend these results can be applied to
piecewise-smooth expanding maps and to higher dimensions. Finally, we
have shown how they can be used to provide a rigorous basis in applied
examples.

\subsection{Period adding for contracting one-dimensional piecewise-smooth maps}
The main contribution of this review article consists of revisiting
the literature to provide sufficient conditions for the occurrence of
this bifurcation scenario (described in
Section~\ref{sec:inc-inc_overview}). It occurs when the map preserves
orientation  near the discontinuity at the origin (is of the
increasing-increasing type), and the proof in summarized in
Section~\ref{sec:summary}.  The key step is to link such type of maps
with maps onto the circle.  Although the resulting circle map is also
discontinuous, the fact that the map is contracting guarantees that
its lift is a strictly increasing map, which is a sufficient condition
for the main properties of the rotation number of a map to hold:
existence, uniqueness and continuity and monotonicity with respect to
parameters.  By using the concept of ``well ordered'' cycles, periodic
orbits are linked with the symbolic sequences given in the Farey tree
of symbolic sequences. When parameters are varied, theory for
expanding circle maps is applied to its inverse to show that the
rotation number follows a devil's staircase and, when the rotation
number is irrational (which occurs in a Cantor set of zero measure),
the omega limit of the map becomes rather the whole circle or a Cantor
set. The latter occurs when the obtained circle map is discontinuous,
the former may occur otherwise.

\subsection{Period incrementing for one-dimensional piecewise-smooth maps}
The period incrementing scenario occurs when the piecewise-smooth map
has different monotonicity at both sides of the discontinuity at the
origin. We recover results in the literature that provide sufficient
conditions for the contracting case. The bifurcation scenario leads to
the existence of stable periodic orbits of the type $\LL^n\R$ or
$\LL\R^n$, depending on whether the map is
increasing-decreasing or decreasing-increasing, respectively.
Moreover, such a periodic orbit 
may coexist with another periodic orbit of the type
$\LL^{n+1}\R$ (or $\LL\R^{n+1}$). In this case, both existing periodic orbits are stable and their
domains of attraction are separated by the origin (the discontinuity).

\subsection{Period adding and incrementing bifurcations in piecewise-smooth expanding maps}
Most of the results regarding the orientable case
(increasing-increasing) shown in this review rely on the fact that the
lift of the corresponding circle map is a monotonically increasing
map, and indeed do not require the map to be contracting. A necessary
condition becomes then the reversibility of the map, which is
guaranteed when it contracts near the discontinuity.
However, such reversibility condition also holds when the map is not
too expanding. In particular, if this holds, the map possesses the same
type of periodic orbits as in the contracting case, although these may be not
unique nor stable. As opposite to the contracting case, the set
of parameters for which one finds dynamics associated with irrational
rotation numbers becomes of non-zero measure.\\
When the map becomes enough expanding so that the reversibility
condition no longer holds, one loses the uniqueness of the rotation
number and needs to deal with rotation intervals instead. Although
there is large theory and literature on this
(\cite{AlsFal03,AlsLli89,AlsLliMisTre89,AlsMan90,Gle90,GleSpa93,Zak93}),
more work is needed in order to provide general sufficient conditions
leading to precise and concrete description of the possible 
bifurcation scenarios.

The results for the non-orientable case (period incrementing
bifurcation) are based on topological arguments for maps on the
interval. The type of periodic orbits perform a number of steps
(re-injection number) on the domain of the increasing branch and only
one in the decreasing one. The smaller the gap at the origin of the
increasing branch is compared to the gap of the decreasing one, the
larger the re-injection number is. Therefore, when the decreasing
branch is expanding, some of the techniques can also be used when this
re-injection number is large enough to compensate the expansiveness of
the decreasing branch.  This results on the existence of the same type
of bifurcations for periodic orbits of large enough length
(re-injection number). Although the same type of period orbits with
shorter period may also exist even when the decreasing branch is
expanding, they may be unstable and coexist with more than one
periodic orbit of the type $\LL^n\R$.\\
In the case where both branches of the map are expanding, the map may
posses positive topological entropy and may hence be chaotic, even
when the lift is continuous (\cite{MisSzle80,Kop05,Kop00}).

\subsection{Maps in higher dimensions}
The jump from dimension one to dimension two represents a really
challenging problem when it comes to provide general conditions for the
occurrence of concrete type of bifurcations. This is magnified when
they are piecewise-smooth discontinuous maps. However, as we have shown in this
review article (see Section~\ref{sec:maximin_approach}) some results
already existent in the literature might significantly help in order
to achieve such a goal. Of special interest is the concept of
quasi-contractions stated in the 80's by Gambaudo et
al.~\cite{Gam87,GamGleTre88,GamTre88}. Although such results provide
very precise description of the type of symbolic sequences associated
with periodic orbits for such maps, those results heavily rely on 
contraction conditions. Numerical evidences (see references in
Section~\ref{sec:maximin_intro}) show that one may find similar
symbolic properties under more relaxed conditions, as it occurs for
the one-dimensional case.\\
Unfortunately, the rotation theory shown in
Sections~\ref{sec:properties_circle_maps}
and~\ref{sec:dyna_orient_preserv} cannot be exported straightforwardly
for maps in higher dimensions. However, similarly as for the
one-dimensional case, one may convert such maps into discontinuous
maps onto higher-dimensional cylinders by identifying the images of
the switching manifolds from both sides in order to introduce a
rotating behaviour. Under more restrictive conditions as for the
one-dimensional case, proceeding similarly as in~\cite{BroSimTat98}
for a particular example, this would allow one to define the lift of a
map, its rotation number and obtain some theoretical results based on
rotation theory.

\def\zh{Zh}\def\yu{Yu}\def\ya{Ya}

\end{document}